\documentclass[11pt]{article} 
\usepackage[margin=1.1in]{geometry}

\usepackage{amsfonts}
\usepackage{amsthm}
\usepackage{amsmath}
\usepackage{amssymb}
\usepackage{graphicx}
\usepackage{xcolor}
\usepackage{tikz}
\usepackage{url}
\usepackage{thmtools}
\usepackage{thm-restate}
\usepackage[mathscr]{eucal}
\usepackage{bbm}
\usepackage{hyperref}
\usepackage[capitalize]{cleveref}
\hypersetup{colorlinks=true,linkcolor=blue,citecolor=blue}

\usepackage[ruled]{algorithm2e}
\PassOptionsToPackage{algo2e,ruled}{algorithm2e}
\usepackage[authoryear,sort&compress]{natbib}

\renewenvironment{proof}[1][]{\par\noindent{\bf Proof #1\ }}{\hfill$\blacksquare$\\
[2mm]}

\newcommand*\mathinhead[1]{\texorpdfstring{#1}{}}

\title{Correlated Binomial Process}

\author{
  \textbf{Mo\"ise Blanchard}\\
  Massachusetts Institute of Technology\\
  \small{\texttt{moiseb@mit.edu}}
  \and
  \textbf{Doron Cohen}\\
  Ben-Gurion University of the Negev\\
  \small{\texttt{doronv@post.bgu.ac.il}}
  \and
 \textbf{Aryeh Kontorovich}\\
 Ben-Gurion University of the Negev\\
 \small{\texttt{karyeh@cs.bgu.ac.il}}
}
\date{}

\newtheorem{theorem}{Theorem}
\newtheorem{lemma}{Lemma}
\newtheorem{definition}{Definition}
\newtheorem{proposition}{Proposition}

\newtheorem{corollary}{Corollary}

\newcommand{\pred}[1]{\boldsymbol{1}[#1]}

\newcommand{\R}{\mathbb{R}}

\newcommand{\N}{\mathbb{N}}

\newcommand{\inv}{^{-1}}

\DeclareMathOperator*{\argmin}{arg\,min}

\newcommand{\sqprn}[1]{\left[ #1 \right]}
\newcommand{\nrm}[1]{\left\Vert #1 \right\Vert}

\newcommand{\vertiii}[1]{{\left\vert\kern-0.25ex\left\vert\kern-0.25ex\left\vert #1 
    \right\vert\kern-0.25ex\right\vert\kern-0.25ex\right\vert}}

\usepackage{fancyvrb} %

\newcommand{\E}{\mathbb{E}}
\renewcommand{\P}{\mathbb{P}}

\newcommand{\Cov}[1]{\operatorname{Cov}\left[#1\right]}

\newcommand{\eps}{\varepsilon}

\newcommand{\eqdef}{:=}
\newcommand{\defeq}{=:}

\usepackage{xcolor}

\newcommand{\beq}{\begin{eqnarray*}}
\newcommand{\eeq}{\end{eqnarray*}}
\newcommand{\beqn}{\begin{eqnarray}}
\newcommand{\eeqn}{\end{eqnarray}}

\usepackage{ifmtarg}
\usepackage{xifthen}%
\newcommand{\ent}[1][]{%
\ifthenelse{\isempty{#1}}{%
\mathrm{H}
}{
\mathrm{H}^{(#1)}
}}

\newcommand{\loch}[1][]{%
\ifthenelse{\isempty{#1}}{%
\mathrm{h}
}{
\mathrm{h}^{(#1)}
}}

\newcommand{\mathe}{\mathrm{e}}

\newcommand{\hide}[1]{}

\def\longto{\mathop{\longrightarrow}\limits}

\newcommand{\ninf}{\longto_{n\to\infty}}

\newcommand{\mx}{\vee}
\newcommand{\mn}{\wedge}

\renewcommand{\phi}{\varphi}

\newcommand{\Bernu}{\operatorname{Bernoulli}}
\newcommand{\Binom}{\operatorname{Binomial}}
\newcommand{\parent}{\operatorname{parent}}
\newcommand{\iid}{\operatorname{iid}}
\newcommand{\Order}{\operatorname{Order}}
\newcommand{\mathd}{\mathrm{d}}

\newcommand{\p}[1][]{p_{#1}}
\newcommand{\pn}[1][]{\hat p_{#1}}
\newcommand{\q}[1][]{q_{#1}}
\newcommand{\qn}[1][]{\hat q_{#1}}
\renewcommand{\r}[1][]{r_{#1}}

\renewcommand{\u}[1][]{u_{#1}}
\newcommand{\un}[1][]{\hat u_{#1}}

\newcommand{\subplus}{_{
\scalebox{.5}{
\hspace{.1em}
\!\!\!\!\!
$\boldsymbol{+}$}
}}
\newcommand{\subminus}{_{
\scalebox{.5}{
\hspace{+0.1em}
\!\!\!\!\!
$\boldsymbol{-}$}
}}
\newcommand{\pl}[1]{\sqprn{#1}\subplus}
\newcommand{\mns}[1]{\sqprn{#1}\subminus}

\begin{document}

%--------- CUSTOM COMMANDS ----------
% names

%\mathcal commands
\newcommand{\Acal}{\mathcal{A}}
\newcommand{\Bcal}{\mathcal{B}}
\newcommand{\Ccal}{\mathcal{C}}
\newcommand{\Dcal}{\mathcal{D}}
\newcommand{\Ecal}{\mathcal{E}}
\newcommand{\Fcal}{\mathcal{F}}
\newcommand{\Gcal}{\mathcal{G}}
\newcommand{\Hcal}{\mathcal{H}}
\newcommand{\Ical}{\mathcal{I}}
\newcommand{\Kcal}{\mathcal{K}}
\newcommand{\Lcal}{\mathcal{L}}
\newcommand{\Ncal}{\mathcal{N}}
\newcommand{\Pcal}{\mathcal{P}}
\newcommand{\Scal}{\mathcal{S}}
\newcommand{\Tcal}{\mathcal{T}}
\newcommand{\Ucal}{\mathcal{U}}
\newcommand{\Vcal}{\mathcal{V}}
\newcommand{\Wcal}{\mathcal{W}}
\newcommand{\Xcal}{\mathcal{X}}
\newcommand{\Ycal}{\mathcal{Y}}
\newcommand{\Ocal}{\mathcal{O}}
\newcommand{\Qcal}{\mathcal{Q}}
\newcommand{\Rcal}{\mathcal{R}}

%\mathbb commands
\newcommand{\Cbb}{\mathbb{C}}
\newcommand{\Ebb}{\mathbb{E}}
\newcommand{\Nbb}{\mathbb{N}}
\newcommand{\Pbb}{\mathbb{P}}
\newcommand{\Qbb}{\mathbb{Q}}
\newcommand{\Rbb}{\mathbb{R}}
\newcommand{\Sbb}{\mathbb{S}}
\newcommand{\Xbb}{\mathbb{X}}
\newcommand{\Ybb}{\mathbb{Y}}
\newcommand{\Zbb}{\mathbb{Z}}

\newcommand{\one}{\mathbbm{1}}
\newcommand{\1}{\mathbbm{1}}%{{\rm 1}\kern-0.24em{\rm I}}

\newcommand{\mb}[1]{\ensuremath{\boldsymbol{#1}}}

\newcommand{\paren}[1]{\left( #1 \right)}
\newcommand{\sqb}[1]{\left[ #1 \right]}
\newcommand{\set}[1]{\left\{ #1 \right\}}
\newcommand{\floor}[1]{\lfloor #1 \rfloor}
\newcommand{\ceil}[1]{\lceil #1 \rceil}
\newcommand{\abs}[1]{\left|#1\right|}

\newcommand{\comment}[1]{}

\maketitle

\begin{abstract}
Cohen and Kontorovich (COLT 2023) initiated the study of what we call here the Binomial Empirical Process: the maximal absolute value of a sequence of inhomogeneous normalized and centered binomials. They almost fully analyzed the case where the binomials are independent, and the remaining gap was closed by Blanchard and Vor\'{a}\v{c}ek (ALT 2024). In this work, we study the much more general and challenging case with correlations. In contradistinction to Gaussian processes, whose behavior is characterized by the covariance structure, we discover that, at least somewhat surprisingly, for binomial processes covariance does not even characterize convergence. Although a full characterization remains out of reach, we take the first steps with nontrivial upper and lower bounds in terms of covering numbers.
\end{abstract}


\section{Introduction}

\citet{CohenK23} introduced the following problem:
Let $Y_j$, $j\in\N$
be a sequence of independent 
$\Binom(n,\p[j])$
random variables,
where $n,j\in\N$.
Since $\E Y_j=n\p[j]$,
it makes sense to
consider
the centered, normalized process
$\bar Y_j:=
n\inv Y_j-\p[j]
$,
which we will refer to as the {\em binomial empirical process}.
The object of interest is
the
expected uniform absolute deviation:
\beqn
\label{eq:barY}
\Delta_n &:=&
\E\sup_{j\in\N}|\bar Y_j|.
\eeqn

The study of $\Delta_n$
--- and in particular, its decay as $n$ grows
--- is a special case of
estimating the mean of a random variable 
$X\in\R^d$
from
a sample of independent draws,
which is
among the most basic problems of statistics.
Much of the earlier theory has focused on obtaining
efficient estimators $\hat m_n$
of the true mean $m$ and
analyzing the decay of
$\nrm{\hat m_n-m}_2$
as a function of sample size $n$, dimension $d$,
and various moment assumptions on $X$
\citep{catoni2012challenging,Devroye16,LuMen19a,LuMen19b,Cherapanamjeri19,Cherapanamjeri20,diakonikolas2020outlier,hopkins2020mean,LuMen21,lee2022optimal}. 
\citet{CohenK23}
were
inspired by \citet{cstheory-faster-low-ent},
who
considered a distribution $\mu$
on $\set{0,1}^d$ with mean
$p\in[0,1]^d$.
Stated in our terms:
an iid sample
$X^{(i)}\sim\mu$ of size $n$ is drawn,
$\pn:=n\inv\sum_{i=1}^n X^{(i)}$ is the empirical mean,
and
$
\Delta_n(\mu) = \E\nrm{\pn-p}_\infty
$.
(\citet{cstheory-faster-low-ent} had conjectured
that the entropy of $\mu$
might control the decay of $\Delta_n$
but \citeauthor{CohenK23} ruled this out.)
The 
$\ell_\infty$
norm 
is in some sense the most interesting
of the $\ell_r$ norms;
indeed, for $r<\infty$,
$\Delta_n^{(r)}:=\E\nrm{\pn-p}_r^r$
decomposes into a sum of expectations
and the condition
$\Delta_n^{(r)}\to0$
reduces to one of convergence
of the appropriate
series.
The uniform convergence implied by
the $\nrm{\cdot}_\infty$
norm as well as the dependence on the particular
distribution $\mu$
led the authors to refer to this problem
as {\em local Glivenko-Cantelli}.

\citet{CohenK23} obtained an almost complete
understanding of the behavior of
$\Delta_n(\mu)$
in the case of product measures
--- i.e., where the $X^{(i)}$ and hence also the $Y^{(i)}$ are independent,
and $\mu$ is 
a product measure
determined entirely by its mean $p$.
Restricting
$p\in[0,1]^\N$
to the range $[0,\frac12]$
and requiring that $\p[j]\downarrow0$
as $j\to\infty$
(which incurs no loss of generality, as shown ibid.)
\citeauthor{CohenK23}
defined the functional
\beq
T(p) &:=& \sup_{j\in\N}\frac{\log (j+1)}{\log(1/\p[j])}
\eeq
and showed that 
$\Delta_n\to0$
iff
$T(p)<\infty$. They also characterized up to constants the asymptotic convergence of $\Delta_n$ when $T(p)<\infty$ via the functional
\beq
S(p) &:=& \sup_{j\in\N}\p[j]\log (j+1),
\eeq
establishing that $\Delta_n$ decays as $\sqrt{S(p)/n}$. 
Additional finite-sample bounds provided therein
were tightened by \citet{blanchard2023tight} as follows:
\begin{align}
    \Delta_n &\asymp 1\land \paren{\sqrt{\frac{S(p)}{n}} + \sup_{j\geq 1}\frac{\log(j+1)}{n\log\paren{2+\frac{\log(j+1)}{n\p[j]}}}}, &\text{if } n\cdot \sup_{j\geq 1} 2j\p[j] >1, \label{eq:tight_bound_independent_1} \\
    \Delta_n &\asymp \frac{1}{n}\land \sum_{j\geq 1}\p[j], & \text{otherwise.} \label{eq:tight_bound_independent_2}
\end{align}

\paragraph{Our contributions.}
In this paper, we tackle the much more challenging
general case of correlated $Y_j
\sim \Binom(n,\p[j])
$. 
The decoupling result in
Theorem~\ref{thm:neg-corr} shows that when the 
pairwise
correlations
are negative,
the behavior of $\Delta_n(\mu)$
remains as in the independent case, up to universal constants.
This might lead one to conjecture that the pairwise correlations
suffice
to characterize the decay of $\Delta_n$.
Theorem~\ref{thm:2procs} decisively shatters this conjecture,
by exhibiting two processes $\mu,\nu$ with identical pairwise covariances
but for which $\Delta_n(\mu)$ decays to $0$ while 
$\Delta_n(\nu)$ does not.
While characterizing the measures $\mu$ for which $\Delta_n(\mu)\ninf 0$ remains a challenging
open problem, we take nontrivial first steps
in this direction. In particular, we define two metrics and show how covering numbers
with respect to these provide upper and lower bounds on $\Delta_n(\mu)$.
More specifically, we give necessary and sufficient conditions on these covering numbers 
for $\Delta_n(\mu)\ninf 0$ 
to hold
and show that 
among the covering-number-based bounds,
this is essentially the best one can do.
Some of the techniques developed for the results may be of independent interest, including decoupling  (Proposition~\ref{prop:ber-decoup}), subgaussian increments (Lemma~\ref{lem:pqr-mgf}), a recursive argument to extract components with large deviations (\cref{thm:finite_cov_num}), and probability distributions with tree structures (Proposition~\ref{prop:example_covering_nb}).

\paragraph{Notation.}

For any two quantities $a,b\in(0,\infty)$, we 
use the shorthand $a\lesssim b$ when there exists a universal constant $C>0$ (independent of any parameters of the problem) such that $a\leq Cb$. Likewise, $a\gtrsim b$ if $b\lesssim a$ and $a\asymp b$ if both $a\lesssim b$ and $a\gtrsim b$ hold.
The floor and ceiling functions, $\floor{t}$,
$\ceil{t}$,
map $t\in\R$
to its closest integers below and above, respectively;
also,
$s\mx t \eqdef \max\set{s,t}$,
$s\mn t \eqdef \min\set{s,t}$,
$\pl{s} \eqdef 0\mx s$
and
$\mns{s} \eqdef 0\mx (-s)$.
Unspecified constants
such as $c,c'$
may change value
from line to line.
We use superscripts to denote distinct random vectors
and subscripts to denote indices within a given vector.
Thus, if $X^{(1)},,\ldots,X^{(n)}$ are independent copies of
$X$, then $X^{(i)}_j$ denotes the $j$th component of the $i$th copy.

The measure-theoretic subtleties of defining distributions
on $\set{0,1}^\N$ are addressed in \citet{CohenK23}.
If $\mu$ is a probability measure on
$\set{0,1}^\N$, we say that $\tilde\mu$ is its
{\em product version} if $\tilde\mu$ is a product
measure on $\set{0,1}^\N$ that agrees with $\mu$ on all of the marginals.
Equivalently, if $X\sim\mu$ and $\tilde X\sim\tilde\mu$,
then 
the components $X_j$ and $\tilde X_j$
are equal in distribution and the $
\set{X_j:j\in\N}$
are mutually independent. In this case, we say that $\tilde X$
is the {\em independent version} of $X$. If $X^{(1)},\ldots,X^{(n)}$
are independent copies of 
$X\sim\mu$, then
\beqn
\Delta_n(\mu) := \E\sup_{j\in\N}\abs{
\frac{1}{n}\sum_{i=1}^n X_j^{(i)}-\E[X_j]
};
\eeqn
this is equivalent to the definition in \eqref{eq:barY}.

A probability measure $\mu$ on $\set{0,1}^\N$
induces the metrics $\xi$ and $\rho$ on $\N$ as follows.
Putting
\beqn
\label{eq:prij}
\p[i] \;:=\; \E_{X\sim\mu}[X_i],
\qquad
r_{ij} \;:=\; \E_{X\sim\mu}[X_i X_j],
\qquad i,j\in\N,
\eeqn
we define
\beqn
\label{eq:xi-metric}
\xi(i,j) &:=& \Pbb(X_i\neq X_j) = \p[i]+\p[j]-2r_{ij} , \\
\label{eq:rho-metric}
\rho(i,j) &:=& 
\frac{2}{\sqrt{3}}\mn \sqrt{\frac{2}{\log {\frac{2}{
\xi(i,j)
}}}}
,
\eeqn
where we set \(\rho(i,j) = 0\) whenever \(\xi(i,j) = 0\).
It is straightforward to verify that both are metrics;
for the former, we have
$\xi(i,j) = \E[(X_i - X_j)^2]$
and the latter is verified as such in Lemma~\ref{lem:rho-metric}.

\section{Main results}
\label{sec:main_results}

Our first result is a decoupling inequality
comparing a negatively correlated binomial process
to its independent version.
If 
$\mu$ is a probability measure on $\set{0,1}^\N$
and $\tilde\mu$ its {\em product version}
as defined above (the two agree on the marginals and $\tilde\mu$
is a product measure),
then ``decoupling from above'', i.e., $\Delta_n(\mu)\lesssim\Delta_n(\tilde\mu)$,
holds without any structural assumptions on $\mu$ and the proof is
quite straightforward \citep[Proposition 3.1]{CHOLLETE202351}.
The other direction is far less trivial
and is clearly not true
in general:
\begin{theorem}[Decoupling from below]
    \label{thm:neg-corr}
Let $\mu$ be a probability measure on $\set{0,1}^\N$ 
with negatively correlated coordinates
(i.e., $X\sim\mu$
verifies
$\E[X_i X_j]\le \E[X_i]\E[X_j]$
for $i,j\in\N$)
and $\tilde\mu$ its {\em product version}.
    Then
\beq
\Delta_n(\mu)
&\ge&
\frac14
\Delta_n(\tilde\mu), \qquad n\ge 1
.
\eeq

\end{theorem}

The proof is given in \cref{sec:proof_thm_1}. In particular, together with the tight bounds Eq~\eqref{eq:tight_bound_independent_1} and \eqref{eq:tight_bound_independent_2} from \citet{blanchard2023tight} 
for
independent Bernoulli sequences, we directly obtain tight non-asymptotic bounds for negatively correlated Bernoulli sequences.

\begin{corollary}
\label{cor:neg-cor}
    Let $\mu$ be a probability measure on $\{0,1\}^\N$ with negatively correlated coordinates. Then, Eq~\eqref{eq:tight_bound_independent_1} and \eqref{eq:tight_bound_independent_2} continue to hold.
\end{corollary}
In the previous result, only the lower bounds are worsened by a factor $\frac{1}{4}$ compared to those for independent Bernoulli sequences in Eq~\eqref{eq:tight_bound_independent_1} and \eqref{eq:tight_bound_independent_2}. It is worth noting that the upper bounds are exactly the same as for the independent case---these only require union bounds and Markov's inequality, hence also hold for general distributions, as noted in \citet[Corollary 3]{blanchard2023tight}.

Based on
Corollary~\ref{cor:neg-cor} 
as well as Gaussian process theory,
one might plausibly conjecture that
the behavior of $\Delta_n(\mu)$
is determined by the covariance structure
of $X\sim\mu$. It is therefore at least somewhat surprising that
this is very much not the case:
\begin{theorem}[Covariance does not characterize $\Delta_n$]
\label{thm:2procs}
There exist probability measures
$\mu,\nu$ on $\set{0,1}^\N$ that 
agree on their
covariance matrices,
\beq
\E_{X\sim\mu}[X_i X_j]
&=&
\E_{X\sim\nu}[X_i X_j]
,
\qquad i,j\in\N,
\eeq
while $\Delta_n(\mu) \ninf 0$ and $\Delta_n(\nu) \ninf \frac{1}{2}$.
\end{theorem}

The proof is given in \cref{sec:proof_thm_3}. It turns that the two distributions we construct also share the exact same third-order moments; that is,
\beq
\E_{X\sim\mu}[X_i X_j X_k ] &=& \E_{X\sim\nu}[X_i X_j X_k], \qquad i,j,k\in\N.
\eeq
In particular, this shows that third-order moments still fail to
characterize the convergence of $\Delta_n$ to $0$
(the proof of this claim also appears in \cref{sec:proof_thm_3}, see Proposition~\ref{prop:3rd_order_moments}).
We conjecture that similar examples could be constructed to show that for any $k\geq 1$, knowledge of the $k$th order moments does not suffice to characterize the decay of $\Delta_n$.

Although a complete understanding of the behavior of $\Delta_n$ as a function of $\mu$ so far remains out of reach,
we take nontrivial steps in this direction by providing upper and lower bounds in terms of the covering numbers. To state our bounds, we fix a probability measure $\mu$ on $\set{0,1}^\N$. We define the metric spaces $(\N,\xi)$ and $(\N,\rho)$ 
as 
in
Eq~\eqref{eq:xi-metric} and \eqref{eq:rho-metric}. For $\eps\in(0,1]$, we will denote by $\Ncal_\xi(\eps)$ and $\Ncal_\rho(\eps)$ the $\eps$-covering numbers of $(\N,\xi)$ and $(\N,\rho)$ respectively. Recall that a subset of a metric space is {\em totally bounded}
if its $\eps$-covering numbers are finite for each $\eps>0$.

We start by showing that 
total boundedness
is a necessary condition for convergence. The proof is deferred to \cref{sec:xi_covering_numbers}.

\begin{theorem}\label{thm:finite_cov_num}
    Let $\mu$ be a distribution on $\{0,1\}^\N$ such that $(\N,\xi)$ is not totally bounded and let $\eps\in(0,1]$ such that $\Ncal_\xi(\eps)=\infty$.  Then, for all $n\geq 1$,
    \begin{equation*}
        \Delta_n(\mu) \geq \frac{\eps^2}{6}.
    \end{equation*}
    Hence, if $\Delta_n(\mu)\to 0$, then $(\N,\xi)$ is totally bounded.
\end{theorem}

The proof uses the following idea: given an $\eps$-separated subset of $A\subset\N$, that is, 
one for which $\xi(i,j)=\Pbb(x_i\neq X_j)>\eps$ for all $i,j\in A$, 
we can show that the event that 
$\set{X_i:i\in A}$
contains infinitely many realizations
of both $0$ and $1$
occurs with probability at least $\eps$. 

The proof then constructs a sequence of decreasing (random) infinite sets $S(n)\subset \N$ for $n\geq 1$ such that on the first $n$ iid samples $X^{(1)},\ldots,X^{(n)}$, all components $i\in S(n)$ diverged from their mean significantly, $|\pn[i]-\p[i]|\gtrsim \eps \p[i]$. These sets can be constructed recursively as follows. Given an infinite set $S(n)\subset A$, we can still apply the initial result to $S(n)$ because it is still $\eps$-separated and infinite. Hence, for any $j>n$, with probability at least $\eps$, the sequence $\{X_i^{(j)}:i\in S(n)\}$ contains infinitely many realizations of both $0$ and $1$. Whenever this event occurs for some samples $n'>n$ (after a geometric waiting time $n'-n\sim \Gcal(\eps)$), we can either select the components $S(n')\subset S(n)$ as those with realization $0$, or those with realization $1$, whichever induces maximal deviation from the means. In both cases, $S(n')$ is still infinite hence we can continue the induction. The deviation on some component $i\in S(n')$ is of order $\p[i]$ after a waiting time $n'-n\approx 1/\eps$, which corresponds to a $\approx \p[i]\eps$ deviation in average; this can be replaced with $\approx \eps^2$ with additional technicalities.

Although 
the necessary condition in
Theorem~\ref{thm:finite_cov_num}
might at first appear somewhat weak,
it turns out to be 
essentially
the best one can do 
solely based on
covering-number
information:

\begin{proposition}\label{prop:ex_fast_converging}
    Let $N:(0,1]\to\Nbb$ be a non-increasing function. Then, there exists a distribution $\mu$ on $\{0,1\}^\N$ and a countable set $E$ such that for any $\eps\in(0,\frac{1}{2}]\setminus E$, one has $\Ncal_\xi(\eps) = N(\eps)$, and
    \begin{equation*}
        \Delta_n(\mu)\ninf 0.
    \end{equation*}
\end{proposition}

The proof is given in \cref{sec:xi_covering_numbers}. The previous result shows that if $\mu$ satisfies the necessary condition from \cref{thm:finite_cov_num}, one can construct a distribution $\tilde\mu$ on $\{0,1\}^\N$ whose covering numbers agree with those of $\mu$ (up to a negligible countable set)
for which $\Delta_n(\tilde \mu)$ does converge to zero. Here, the 
exceptional set
$E$ exactly corresponds to the scales $\eps$ on which the covering number function $N(\cdot)$ is discontinuous: for the constructed distribution $\mu$, the covering numbers are right-continuous, which may not necessarily be the case for a generic covering number function $\Ncal_\xi(\cdot)$.

In light
of \cref{thm:2procs}, there is no loss of generality in assuming
suppose that all covering numbers $\Ncal_\xi(\eps)$ (a fortiori for $\Ncal_\rho(\eps)$) are finite.
In this case,
we have the following sufficient condition for the convergence of $\Delta_n(\mu)$ to $0$, written in terms of the $\xi$-covering numbers.

\begin{theorem}\label{thm:sufficient_condition_covering_nb}
    Let $\mu$ be a distribution on $\{0,1\}^\N$ such that
    \begin{equation}\label{eq:sufficient_condition}
        C_\mu:= \int_0^1 \Ncal_\xi(\eps)\mathd\eps <\infty.
    \end{equation}
    Then, $\Delta_n(\mu)\to 0$ as $n\to\infty$. Further, in that case, for any $n\geq 1$,
    \begin{equation*}
        \Delta_n(\mu) \leq C_1  \inf_{\eps\in(0,1]} \eps + \sqrt{\frac{\log (\Ncal_\xi(\eps)+1)}{n}} \leq C_2\paren{   1\land \sqrt{\frac{\log(n(1+C_\mu))}{n}}   },
    \end{equation*}
    for universal constants $C_1,C_2>0$.
\end{theorem}
The proof, which we defer to \cref{sec:xi_covering_numbers}, proceeds via a chaining
argument.
We consider the directed graph such that a node
in $\Scal_\xi(2^{-k})$ (covering set at level $2^{-k}$) has as parent its nearest-neighbor within $\Scal_\xi(2^{-k+1})$. 
We show that the property $C_\mu<\infty$ ensures that with good probability, starting for any point at level $k>k_0$, the value at the node and its parent coincide exactly. This overall probability grows as $k_0$
is made larger. Conditional on this good event,
it suffices to bound the deviation of components in $\Scal_\xi(2^{-k})$ for $k\leq k_0$, 
for which convergence
is trivial due to the finite
number of points.

We note that the non-asymptotic upper bounds for $\Delta_n(\mu)$ of the previous result are in general not tight. In particular, under stronger assumptions, we can use tools from Gaussian process theory to exhibit subgaussian rates of convergence ($1/\sqrt n$ instead of $\sqrt{\log n/n}$) for $\Delta_n(\mu)$. More precisely, we show that $X\sim\mu$ is a subgaussian process on $\N$ with respect to the metric $\rho$ (Lemma~\ref{lem:pqr-mgf}). Hence,
a direct application of
Dudley's theorem
\citet[Corollary 5.25]{van2014probability}
yields
the following result.

\begin{proposition}\label{prop:dudley_thm}
    Let $\mu$ be a distribution on $\{0,1\}^\N$ such that
    \begin{equation*}
        D_\mu:= \int_0^1 \sqrt{\log \Ncal_\rho(\eps)} \mathd\eps <\infty.
    \end{equation*}
    Then, for any $n\geq 1$, we have
    \begin{equation*}
        \Delta_n(\mu) \leq \frac{24 D_\mu}{\sqrt n}.
    \end{equation*}
\end{proposition}
The proof that $X\sim\mu$ is a subgaussian process (Lemma~\ref{lem:pqr-mgf}) 
along with
the proof of the above proposition are deferred to \cref{sec:subgaussian}.

In light of \cref{thm:2procs}, the condition from Eq~\eqref{eq:sufficient_condition} cannot also be a necessary condition for
$\Delta_n(\mu)\ninf 0$, 
since it only involves pairwise correlations between coordinates of the distribution. However, we show that this is also essentially the tightest sufficient condition that 
can be stated solely in terms of the
covering numbers $\Ncal_\xi(\cdot)$. The proof can be found in \cref{sec:xi_covering_numbers}.

\begin{proposition}\label{prop:example_covering_nb}
    Let $N:(0,1]\to\N$ be a non-increasing function such that
    \begin{equation*}
        \int_0^{1/2} N(\eps)\mathd\eps  = \infty.
    \end{equation*}
    Then, there exists a distribution $\mu$ on $\{0,1\}^\N$ and a countable set $E$ such that for any $\eps\in (0,\frac{1}{2}]\setminus E$, one has $\Ncal_\xi(\eps) = N(\eps)$, and for any $n\geq 1$,
    \begin{equation*}
        \Delta_n(\mu) = \frac{1}{2}.
    \end{equation*}
\end{proposition}

Towards characterizing distributions $\mu$ for which $\Delta_n(\mu)$ decays to $0$, we have the following sufficient condition, which subsumes the condition from Eq~\eqref{eq:sufficient_condition} and does not involve covering numbers. The proof is given in \cref{sec:examples}.

\begin{theorem}\label{thm:more_general_suff_condition}
    Let $\mu$ be a distribution on $\{0,1\}^\N$ such that $(\N,\xi)$ is totally bounded. Suppose that there exists $K\geq 1$ such that for any $\eps>0$, there exist events $(E_k)_{k\in\N}$ and a finite set $J\subset \N$ with
    \begin{itemize}
        \item $\Pbb(E_k)\leq \eps,\;\forall k\in\N$,
        \item $\displaystyle \sup_{k\in\N} \frac{\log(k+1)}{\log\frac{1}{\Pbb(E_k)}}<\infty$,
        \item $\forall i\in\N, \exists j\in J, \exists \Kcal\subset \N,  \text{ such that } |\Kcal|\leq K \text{ and } \{X_i\neq X_j\}\subset \bigcup_{k\in \Kcal} E_k$.
    \end{itemize}
    Then $\Delta_n(\mu) \ninf 0$.
\end{theorem}

While covering numbers $\Ncal_\xi(\eps)$ aim to 
account for
the bad events when pairs of coordinates differ $\{X_i\neq X_j\}$ for $i,j\in\N$, the condition from \cref{thm:more_general_suff_condition} generalizes this covering number approach in two distinct ways: 
(i) Instead of considering the bad events $\{X_i\neq X_j\}$ only through the expectation $\xi(i,j) = \Pbb(X_i\neq X_j)$, the condition allows for positive correlations of these bad events. Hence, instead of defining balls in the space $(\N,\xi)$, we directly define events $E_k$ in the probability space of $\mu$. 
(ii) Instead of covering a bad event $\{X_i\neq Z_j\}$ with a single event $E_k$ (or ball in the covering number approach), we allow for the event to be covered by several events $E_k$. However, the number of covering events $E_k$ for each bad event needs to be bounded (condition $|\Kcal|\leq K$).

Both generalizations are important to define the exact characterization for distributions $\mu$ such that $\Delta_n(\mu)\to 0$. In particular, the necessity of the first generalization (i) is already exemplified by \cref{thm:2procs}. 
In Section~\ref{sec:examples}, following the proofs of the latter claims,
we provide illustrative examples of distributions demonstrating the necessity of these generalizations (distributions $\mu$ for which $\Delta_n(\mu)\to 0$ and that wouldn't satisfy the condition of \cref{thm:more_general_suff_condition} without such a generalization). 
These are meant to guide the intuition in future work towards
the full
characterization, which we leave as an open question.

\paragraph{Open problem.}
Characterize the distributions
$\mu$ on $\set{0,1}^\N$
for which $\Delta_n(\mu)
\to0$ as $n\to\infty$.
Is the sufficient condition from \cref{thm:more_general_suff_condition} also necessary to have $\Delta_n(\mu)\to0$?

\paragraph{Open problem.}
How can the covering numbers
$\Ncal_\rho(\eps)$
or
$\Ncal_\xi(\eps)$
be calculated or bounded
explicitly in terms of $\mu$?

\section{Proof of Theorem \ref{thm:neg-corr} (Decoupling from below)}
\label{sec:proof_thm_1}
\paragraph{Notation.}

Throughout this section, whenever $X_1,X_2,\ldots$
is a collection of (one-dimensional) random variables,
we denote by $\tilde X_1,\tilde X_2,\ldots$
its {\em independent version}:
the $\tilde X_i$s are 
mutually
independent (and independent of the $X_i$s),
and each $X_i$ is equivalent to $\tilde X_i$
in distribution.
In this section, we will provide comparison results
of the type
$
\E\max_{i\in[d]}X_i
\gtrsim
\E\max_{i\in[d]}\tilde X_i
$
under negative pairwise correlation 
conditions on the $X_i$.

\paragraph{Bernoulli case.}
We begin with the case where
$X_i\sim\Bernu(\p[i])$, $i\in[d]$.
Letting
$Z=\sum_{i=1}^d X_i$
and
$\tilde Z=\sum_{i=1}^d \tilde X_i$,
we have
\beqn
\label{eq:ZZtil}
\E\max_{i\in[d]}X_i=\P(Z>0),
\qquad
\E\max_{i\in[d]}\tilde X_i=\P(\tilde Z>0).
\eeqn
Let us recall the notion of
pairwise independence:
for each $i\neq j\in [d]$, we have
$\E[X_iX_j]=
\E[X_i]\E[X_j]
$. 
\begin{proposition}
\label{prop:ber-decoup}
Let $X_i,\tilde X_i,Z,\tilde Z,\p[i]$ 
be as in \eqref{eq:ZZtil}
and assume additionally that the $X_i$
are pairwise independent.
Then
\beq
\P(Z>0) \ge \frac12\P(\tilde Z>0).
\eeq
\end{proposition}
\begin{proof}
By the Paley-Zygmund inequality,
\beq
\P(Z>0) \ge
\frac{(\E Z)^2}{\E[Z^2]}.
\eeq
Now $\E Z=\sum_{i=1}^d \p[i]$, which we assume without loss of generality to be non-zero (otherwise $\Pbb(Z>0)=\Pbb(\tilde Z>0)=0$),
and by pairwise independence,
\beqn
\label{eq:Z^2}
\E[Z^2]
=
\sum_{i=1}^d \p[i]
+
2\sum_{1\le i<j\le d} \p[i]\p[j]
=
\sum_{i=1}^d \p[i]
+
\paren{\sum_{i=1}^d \p[i]}^2
-
\sum_{i=1}^d \p[i]^2
\le
\sum_{i=1}^d \p[i]
+
\paren{\sum_{i=1}^d \p[i]}^2
.
\eeqn
Hence,
\beq
\frac{(\E Z)^2}{\E[Z^2]}
\ge
\frac{
\paren{\sum_{i=1}^d \p[i]}^2
}{
\sum_{i=1}^d \p[i]
+
\paren{\sum_{i=1}^d \p[i]}^2
} = \frac{\sum_{i=1}^d \p[i]}{ 1+\sum_{i=1}^d \p[i]}
.
\eeq
On the other hand,
$\P(\tilde Z>0)$ is readily computed:
\beq
\P(\tilde Z>0)
=
1-\prod_{i=1}^d(1-\p[i]).
\eeq
Therefore, to prove the claim, it suffices to show that
\beq
G(\p[1],\ldots,\p[d])
:=
2 \sum_{i=1}^d \p[i] - 
\paren{1+\sum_{i=1}^d \p[i] }
\paren{1-\prod_{i=1}^d(1-\p[i])}
\ge0.
\eeq
To this end, we write $G=S+P+SP-1$, where
$S=\sum_i \p[i]$
and $P=\prod_i(1- \p[i])$.
Now if $S\ge1$ then obviously $G\ge0$
and we are done.
Otherwise,
since $P\ge1-S$
trivially holds,
we have
$G\ge S(1-S)$.
In this case, $S<1\implies G\ge0$.
\end{proof}

\paragraph{Relaxing pairwise independence.}
An inspection of the proof
shows that we do not actually need
$\E[X_iX_j]=\p[i]\p[j]$,
but rather only
$\E[X_iX_j]\le \p[i]\p[j]$.
This condition is called {\em negative (pairwise) covariance} \citep{Dubhashi:1998:BBS:299633.299634}.
\begin{corollary}
    \label{cor:ber-neg}
Let $X_i,\tilde X_i,Z,\tilde Z,\p[i]$ 
be as in \eqref{eq:ZZtil}
and assume additionally that the $X_i$
satisfy negative pairwise covariance: $\E[X_iX_j]\le \p[i]\p[j]$
for $i\neq j$.
Then
\beq
\P(Z>0) \ge \frac12\P(\tilde Z>0).
\eeq
\end{corollary}

\paragraph{General positive random variables.}
Now let
$X_1,\ldots,X_d$ be non-negative
integrable random variables
and the
$\tilde X_1,\ldots,\tilde X_d$
are their independent copies:
each $\tilde X_i$ is distributed identically to $X_i$
and the $\tilde X_i$ are mutually independent.

\begin{proposition}
\label{prop:ber-decoup-cont}
Let
$X_1,\ldots,X_d$ be non-negative and
integrable
with independent copies
$\tilde X_i$ as above.
If additionally the $X_i$
are pairwise independent, then
\beq
\E\max_{i\in[d]}X_i
\ge
\frac12\E\max_{i\in[d]}\tilde X_i
.
\eeq
\end{proposition}

\begin{proof}
For $t>0$ and $i\in[d]$,
put 
$Y_i(t)=\pred{X_i>t}$,
$\tilde Y_i(t)=\pred{\tilde X_i>t}$
and
$Z(t)=\sum_{i=1}^d Y_i(t)$,
$\tilde Z(t)=\sum_{i=1}^d Y_i(t)$.
Then
\beq
\E\max_{i\in[d]}X_i
&=&
\int_0^\infty
\P\paren{
\max_{i\in[d]}X_i>t
}\mathd t
\\&=&
\int_0^\infty
\P\paren{
Z(t)>0
}\mathd t
\\
&\ge&
\frac12\int_0^\infty
\P\paren{
\tilde Z(t)>0
}\mathd t
\\&=&
\frac12\int_0^\infty
\P\paren{
\max_{i\in[d]}\tilde X_i>t
}\mathd t
\\&=&
\frac12\E\max_{i\in[d]}\tilde X_i
,
\eeq
where Proposition~\ref{prop:ber-decoup}
was invoked in the inequality step.
\end{proof}

\paragraph{Relaxing pairwise independence.} \label{par:pnuod}
As before, the full strength of pairwise
independence of the $X_i$ is not needed.
The condition $\P(X_i>t,X_j>t)\le \P(X_i>t)\P(X_j>t)$
for all $i\neq j\in[d]$ and $t>0$, called pairwise
{\em negative upper orthant dependence} \citep[Definition~2.3]{10.1214/aos/1176346079},  would suffice;
it is weaker than 
pairwise independence.

\begin{corollary}
    \label{cor:neg-orth}
Let
$X_1,\ldots,X_d$ be non-negative and
integrable
with independent copies
$\tilde X_i$ as above.
If additionally the $X_i$
verify
$\P(X_i>t,X_j>t)\le \P(X_i>t)\P(X_j>t)$
for all $i\neq j\in[d]$ and $t>0$.
Then
\beq
\E\max_{i\in[d]}X_i
\ge
\frac12\E\max_{i\in[d]}\tilde X_i
.
\eeq
\end{corollary}

\begin{definition}[\cite{10.1214/aos/1176346079}, Definition 2.1]
Random variables $X_1, X_2, ..., X_k$ are said to be \emph{negatively associated} (NA) if for every pair of disjoint subsets $A_1, A_2$ of $\set{1,2, ..., k}$,
\beq
\Cov{f_1(X_i,\ i\in A_1), f_2(X_j,\ j \in A_2)}
\leq
0
\eeq
whenever $f_1$ and $f_2$ are increasing. NA may also refer to the vector $X = (X_1, \dots, X_k)$ or to the underlying distribution of $X$. Additionally, NA may denote negative association. If $|A_1|=|A_2|=1$ we say that $X$ is \emph{pairwise negatively associated} (PNA). Note that the definition is the same if both $f_1$ and $f_2$ are decreasing.
\end{definition}
\begin{theorem}[Restatement of Theorem~\ref{thm:neg-corr}]
    \label{thm:neg-corr-restated}
    Let $\mu$ be a probability measure on $\set{0,1}^\N$
    such that for $X\sim\mu$, the $i$th and $j$th components
    of $X$ satisfy $\E[X_iX_j]\le\E[X_i]\E[X_j]$.
    Let $X^{(1)},\ldots,X^{(n)}$ be $n$ independent copies of $X$
    and define $Y=n\inv\sum_{i=1}^nX^{(i)}-\E[X]$.
    Let $\tilde Y$ be the independent version of $Y$:
    each component $\tilde Y_i$ is equal to $Y_i$ in distribution
    and the $(\tilde Y_i)_{i\in\N}$ are mutually independent.
    Then
\beq
\E\sup_{i\in\N}\abs{Y_i}
\ge
\frac14
\E\sup_{i\in\N}\abs{\tilde Y_i}.
\eeq
\end{theorem}

\begin{proof}
    The proof consists of three parts. We first show that the condition \(\E[X_iX_j] \le \E[X_i]\E[X_j]\) implies that $X$ is pairwise negatively associated. Then, we show that this implies $Y$ is also PNA. Finally, we invoke Corollary~\ref{cor:neg-orth} for the vector $\abs{Y}$ and we are done.

For any $i,j$ let $\p[i] \eqdef \E[X_i]$ and $\r[ij] \eqdef \E[X_iX_j]$. By assumption,
$\r[ij] \leq \p[i]\p[j]$.
    Let \(f,g\) be two real-valued non-decreasing functions, then
\begin{align*}
    \E & [f(X_i)g(X_j)] - \E[f(X_i)]\E[g(X_j)] \\
    & = rf(1)g(1) + (\p[i]-\r[ij])f(1)g(0) + (\p[j]-\r[ij])f(0)g(1) + (1-\p[i]-\p[j]+\r[ij])f(0)g(0) \\
    & - \p[i] \p[j] f(1)g(1) - \p[i](1-\p[j])f(1)g(0) - (1-\p[i]) \p[j] f(0)g(1) - (1-\p[i])(1-\p[j]) f(0)g(0) \\
    & = (\r[ij]-\p[j]\p[j])(f(1)g(1)-f(1)g(0)-f(0)g(1)+f(0)g(0)) \\
    & = (\r[ij]-\p[j]\p[j])(f(1)-f(0))(g(1)-g(0)) \\
    & \leq 0
\end{align*}
    for all $i,j$. This shows that \(X\) is PNA.
    Since any pair $(X_i, X_j)$ is NA, by \citet[Property~P$_7$]{10.1214/aos/1176346079}, the union of independent sets of NA random variables are NA, so  
    $$(X_i^{(1)}, X_i^{(2)}, \dots, X_i^{(n)}, X_j^{(1)}, X_j^{(2)}, \dots, X_j^{(n)})$$ is also NA. Now, we apply \citet[Property~P$_6$]{10.1214/aos/1176346079}: increasing (or decreasing) functions defined on disjoint subsets of a set of NA random variables are NA, where the increasing functions are \(\pl{Y_i} \eqdef f_+(X_i^{(1)}, X_i^{(2)}, \dots, X_i^{(n)}) \eqdef \pl{n\inv\sum_{k=1}^n X_i^{(k)}-\E[X_i]} \) for all $i$. We conclude that $\pl{Y} = (\pl{Y_1}, \pl{Y_2}, \dots)$ is PNA. In the same manner, $\mns{Y}$ is also PNA.
    Moreover, PNA obviously implies pairwise negative upper orthant dependence.
    Now,
  \begin{align*}
    \E\sup_{i\in\N}\abs{Y_i}
    &=
    \E\sup_{i\in\N}\pl{Y_i}+\mns{Y_i}
    \\
    &\geq
    \frac{1}{2}\E\sup_{i\in\N}\pl{Y_i}
    +\frac{1}{2}\E\sup_{i\in\N}\mns{Y_i}
    \\
    &\geq
    \frac{1}{4}\E\sup_{i\in\N}\pl{\tilde Y_i}
    +\frac{1}{4}\E\sup_{i\in\N}\mns{\tilde Y_i}
    \\&
    \geq
    \frac{1}{4}\E\sup_{i\in\N}\abs{\tilde Y_i},
  \end{align*}
  where the second inequality is due to Corollary~\ref{cor:neg-orth}.
\end{proof}

\section{Proof of Theorem~\ref{thm:2procs} (Covariance does not characterize \mathinhead{$\Delta_n$})}
\label{sec:proof_thm_3}

In this section, we provide an example of two distributions $\mu$ and $\nu$ on $\{0,1\}^\N$ that share the same covariance matrix but for which $\Delta_n(\mu)\to 0$ and $\Delta_n(\nu)\to \frac{1}{2}$ as $n\to\infty$.

\paragraph{Construction of the example.}
We partition $\N$ as follows: $\Scal_k = \{2^{(k-1)^3} < t \leq 2^{k^3}\}$ for $k\geq 1$. Let $(Z_1)_{i\geq 1}\overset{\iid}{\sim}\Bernu(1/2)$ and $(Y_k)_{k\geq 0}\overset{\iid}{\sim}\Bernu(1/2)$ be 
iid sequences.

Let $\gamma_k,\delta_k\in(0,1)$ be the solutions to
\begin{equation*}
    \begin{cases}
        (1-\delta_k)(2\gamma_k-\gamma_k^2) &= 2^{-k}\\
        \gamma_k+\delta_k - \gamma_k\delta_k &= 2^{-k}.
    \end{cases}
\end{equation*}
We can check that $\gamma_k = \frac{2^{-k}-\delta_k}{1-\delta_k}$ and $\delta_k=1-2^{-k-1} -\sqrt{(1-2^{-k-1})^2-2^{-k}+2^{-2k}}$. In particular, $\gamma_k\sim \delta_k\sim 2^{-k-1}$ as $k\to\infty$.

\begin{itemize}
    \item Let $(B_k)_{k\geq 1}$ be a sequence of independent 
    random variables with $B_k\sim\Bernu(2^{-k})$. Put
    \begin{equation*}
        X_t^\mu = (1-B_k)Z_0 + B_k Z_t,\quad t\in\Scal_k,\, k\geq 1.
    \end{equation*}
    We define $\mu$ as the distribution of $(X_i^\mu)_{i\geq 1}$.
    \item Let $(C_i)_{i\geq 1}$ and $(D_k)_{k\geq 1}$ be independent sequences of independent
    random variables with $C_i\sim\Bernu(\gamma_k)$ for $i\in\Scal_k$ and $D_k\sim\Bernu(\delta_k)$. We let
    \begin{equation*}
         X_t^\nu = D_k Y_k + (1-D_k)((1-C_t)Z_0 + C_t Z_t),\quad t\in\Scal_k,\, k\geq 1.
    \end{equation*}
    We define $\nu$ as the distribution of $(X_i^\nu)_{i\geq 1}$.
\end{itemize}
We can easily check that these distributions have the same covariance matrix.

\paragraph{Computing the covariance matrix of $\mu$ and $\nu$}
Both sequences $(X^\mu_i)_{i\geq 1}$ and $(X^\nu_i)_{i\geq 1}$ are sequences of Bernoullis of parameter $1/2$ because they are mixtures of $Z_i$ and $Y_k$ that are independent 
$\Bernu(1/2)$. Now for any $i,j\in\Scal_k$ with $i\neq j$, one has
\begin{equation*}
    \Ebb[X_i^\mu X_j^\mu] = \frac{1}{2}\Pbb[B_k=0] + \frac{1}{4}\Pbb[B_k=1] = \frac{1}{2} - \frac{1}{2^{k+2}}.
\end{equation*}
Also, for $i\in\Scal_k$ and $j\in \Scal_l$ with $k\neq l$, one has
\begin{equation*}
    \Ebb[X_i^\mu X_j^\mu] = \frac{1}{2}\Pbb[B_k=B_l=0] + \frac{1}{4}(1-\Pbb[B_k=B_l=0]) = \frac{1}{2} - \frac{2^{-k} + 2^{-l} - 2^{-k-l}}{4}.
\end{equation*}

Next, we turn to the sequence $(X^\nu_i)_{i\geq 1}$. For any $i,j\in\Scal_k$ with $i\neq j$, one has
\begin{align*}
    \Ebb[X_i^\nu X_j^\nu] &= \frac{1}{2}\Pbb[D_k=1] + \frac{1}{2}\Pbb[D_k=0]\Pbb[C_i=C_j=0] + \frac{1}{4}\Pbb[D_k=0](1-\Pbb[C_i=C_j=0])\\
    &=\frac{1}{2} - \frac{(1-\delta_k)(2\gamma_k-\gamma_k^2)}{4}=\Ebb[X_i^\mu X_j^\mu].
\end{align*}
Last, for $i\in\Scal_k$ and $j\in \Scal_l$ with $k\neq l$, one has
\begin{align*}
    \Ebb[X_i^\nu X_j^\nu] &= \frac{1}{2}\Pbb[D_k=D_l=C_i=C_j=0] + \frac{1}{4}(1-\Pbb[D_k=D_l=C_i=C_j=0])\\
    &=\frac{1}{2} - \frac{1-(1-\gamma_k)(1-\delta_k)(1-\gamma_l)(1-\delta_l)}{4}\\
    &= \frac{1}{2} - \frac{1-(1-2^{-k})(1-2^{-l})}{4} = \Ebb[X_i^\mu X_j^\mu].
\end{align*}

As a result, if $\mu$ (resp. $\nu$) denotes the distribution of $(X^\mu_i)_{i\geq 1}$ (resp. $(X^\nu_i)_{i\geq 1}$), they both share the same covariance matrix.

At the high level, both distributions are constructed by block $\Scal_k$ such that within the block, the random variables are correlated up to a level $\approx 1-2^{-k}$. In the first distribution $\mu$, the correlation between variables $X_i$ for $i\in\Scal_k$ is made uniform through the decision variable $B_k$. Hence, to have convergence of the maximum deviation on $\Scal_k$ it suffices to handle a single decision variable $B_k$, which is a Bernoulli with parameter $2^{-k}$. Because these parameters are summable, we can control these variables uniformly for all $k$ sufficiently large.

In the second case for $\nu$, this correlation is made heterogeneous, by introducing decision variables $C_i$ for each $i\in \Scal_k$. These are independent and by taking $|\Scal_k|$ sufficiently large, one can enforce rare deviation events to happen with high probability for one of the variables $i\in\Scal_k$.

We formalize these ideas in the rest of this section. For clarity, for any $i\geq 1$, we denote by $\pn[i](\mu)$ and $\pn[i](\nu)$ the quantities $\pn[i]$ for $\mu$ and $\nu$ respectively. They share the same means $\p[i]$ so we need not make the distinction here. We also denote with an exponent $U^{(1)},U^{(2)},\ldots$ iid samples from any random variable $U$.

\paragraph{Expected maximum deviation for $\mu$.} Fix $k\geq 1$. One has
\begin{align*}
   \max_{i\in\Scal_k} |\pn[i](\mu) - \p[i]| &\leq \left|\frac{1}{n} \sum_{i=1}^n Z_0^{(i)} -\frac{1}{2}\right| + \max_{i\in\Scal_k} \left|\pn[i](\mu) - \frac{1}{n} \sum_{i=1}^n Z_0^{(i)}\right|\\
   &\leq \left|\frac{1}{n} \sum_{i=1}^n Z_0^{(i)} -\frac{1}{2}\right| + \frac{1}{n}\sum_{i=1}^n B_k^{(i)}.
\end{align*}
As a result,
\begin{align*}
    \Ebb \sup_{l\geq k}\max_{i\in\Scal_l} |\pn[i](\mu) - \p[i]| &\leq \Ebb \left|\frac{1}{n} \sum_{i=1}^n Z_0^{(i)} -\frac{1}{2}\right| + \Ebb \sup_{l\geq k} \frac{1}{n}\sum_{i=1}^n B_l^{(i)}\\
    &\leq \Ebb \left|\frac{1}{n} \sum_{i=1}^n Z_0^{(i)} -\frac{1}{2}\right| + \sum_{l\geq k}\Ebb \sqb{ \frac{1}{n}\sum_{i=1}^n B_l^{(i)} }\\
    &= \Ebb \left|\frac{1}{n} \sum_{i=1}^n Z_0^{(i)} -\frac{1}{2}\right| + 2^{-k+1}.
\end{align*}
Then,
\begin{align*}
    \Delta_n(\mu) &\leq \sum_{t=1}^{2^{(k-1)^3}}\Ebb|\pn[t](\mu)-\p[t]| + \Ebb \sup_{l\geq k}\max_{i\in\Scal_l} |\pn[i](\mu) - \p[i]|\\
    &\leq (1+2^{(k-1)^3})\Ebb|\pn[1](\mu)-\p[1]| + 2^{-k+1}.
\end{align*}
In particular, this shows that $\limsup_{n\to\infty}\Delta_n(\mu)\leq 2^{-k+1}$. Because this holds for all $k\geq 1$, we obtained $\Delta_n(\mu)\to 0$ as $n\to\infty$.

\paragraph{Expected maximum deviation for $\nu$.} Fix $n\geq 1$. We have
\begin{equation*}
    \Pbb(\exists i\in\Scal_n,\forall j\in[n],C_i^{(j)}=1,\, Z_i^{(j)}=1) = 1-\paren{1-\paren{\frac{\gamma_n}{2}}^n}^{|\Scal_n|} \geq 1-\exp \paren{-\paren{\frac{\gamma_n}{2}}^n |\Scal_n| }
\end{equation*}
Denote by $\Ecal_n$ this event. On this event, there exists $i\in\Scal_n$ such that
\begin{equation*}
    \pn[i](\nu) = \frac{1}{n}\sum_{j=1}^n {X_i^\nu}^{(j)} = \frac{1}{n}\sum_{i=1}^n D_n^{(j)} Y_n^{(j)} + 1-  D_n^{(j)}.
\end{equation*}
Then,
\begin{align*}
    \Delta_n(\nu) &\geq \Ebb\sqb{\1_{\Ecal_n} \sup_{i\in\Scal_n} |\pn[i](\nu) -\p[i]|}\\
    &\geq \Ebb\sqb{\1_{\Ecal_n} \left|\frac{1}{n}\sum_{j=1}^n (D_n^{(j)}Y_n^{(j)} + 1-D_n^{(j)}) -\frac{1}{2}\right|}\\
    &\geq \Ebb\sqb{\frac{1}{n}\sum_{j=1}^n (D_n^{(j)}Y_n^{(j)} + 1-D_n^{(j)}) -\frac{1}{2}} - \Pbb(\Ecal_n^c)\\
    &\geq \frac{1}{2}-\frac{\delta_n}{2} - \exp \paren{-\paren{\frac{\gamma_n}{2}}^n |\Scal_n| }.
\end{align*}
We recall that $\gamma_n\sim 2^{-n-1}$ so that $\paren{\frac{\gamma_n}{2}}^n |\Scal_n|\to \infty $ as $n\to\infty$. As a result, we obtain
\begin{equation*}
    \lim_{n\to\infty}\Delta_n(\nu) =\frac{1}{2}.
\end{equation*}
This ends the proof of \cref{thm:2procs}.

\paragraph{Higher-order moments}
\cref{thm:2procs} shows that knowledge of the covariance matrix is not sufficient to characterize the behavior of $\Delta_n(\mu)$ or not. We suspect that this negative result holds more generally, that is, for any $k\geq 1$, the $k$th order moments are not sufficient to characterize whether $\Delta_n(\mu)\to 0$ or not. As it turns out, the distributions $\mu$ and $\nu$ constructed for \cref{thm:2procs} also share the same $3$rd-order moments.

\begin{proposition}[Third-order moments do not characterize $\Delta_n$]
\label{prop:3rd_order_moments}
    There exist probability measures $\mu$ and $\nu$ on $\{0,1\}^\N$ that agree on their $3$rd-order moments,
\beq
\E_{X\sim\mu}[X_i X_j X_k] &=& \E_{X\sim\nu}[X_i X_j X_k], \qquad i,j,k\in\N,
\eeq
    while $\Delta_n(\mu) \ninf 0$ and $\Delta_n(\nu) \ninf \frac{1}{2}$.
\end{proposition}

\begin{proof}
    We simply check that the two distributions $\mu$ and $\nu$ from the proof of \cref{thm:2procs} also share the same 3rd-order moments. Let $i,j,k\in\N$ be three indices. When they are not all distinct, the desired equation follows from the proof of \cref{thm:2procs} because the coordinates $X_i$ are binary. Without loss of generality, we then suppose that they are distinct. We denote by $l(i)$ the index of the block corresponding to $i$, that is, such that $i\in\Scal_{l(i)}$. We define similarly $l(j)$ and $l(k)$.

    \paragraph{Case 1.} We first treat the simple case when $l(i)\notin\{l(j),l(k)\}$. Then, note that we can write
    \begin{equation*}
        X_i^\nu = (1-B_{l(i)}') Z_0 + B_{l(i)'}Z_{i,l(i)}'
    \end{equation*}
    where $B_{l(i)}':= D_{l(i)} + C_i - D_{l(i)} C_i$ and $Z_{i,l(i)}' = \1_{D_{l(i)}=1} Y_{l(i)} + \1_{D_{l(i)}=0} Z_i$. Note that $B_{l(i)}'\sim B_{l(i)}$ since $\delta_{l(i)} + \gamma_i - \delta_{l(i)}\gamma_i = 2^{-l(i)}$. Also,  $Z_{i,l(i)}'\sim\Bernu(\frac{1}{2})$, hence $Z_{i,l(i)}'\sim Z_i$. Further, they are independent, but more importantly, they are independent from all the variables that would be used to define $X_j$ and $X_k$, except for $Z_0$. As a result, all that remains to check is that
    \begin{equation*}
        \Ebb[Z_0 X^\mu_i X^\mu_j] = \Ebb[Z_0 X^\nu_i X^\nu_j].
    \end{equation*}
    We start with the distribution $\mu$.
    If $l(j)\neq l(k)$, then,
    \begin{align*}
        \Ebb[Z_0 X^\mu_i X^\mu_j] &= \frac{1}{2}\Pbb[B_{l(j)}=B_{l(k)}=0] + \frac{1}{8}(1-\Pbb[B_{l(j)}=B_{l(k)}=0])\\
        &= \frac{1}{2} - \frac{3}{8}(2^{-l(j)} + 2^{-l(k)} - 2^{-l(j)-l(k)}).
    \end{align*}
    On the other hand, if $l(j)= l(k)$, then
    \begin{equation*}
        \Ebb[Z_0 X^\mu_i X^\mu_j] = \frac{1}{2}\Pbb[B_{l(j)}=0] + \frac{1}{8}\Pbb[B_{l(j)}=1 ] = \frac{1}{2} - \frac{3}{8}2^{-{l(j)}}.
    \end{equation*}
    Next, for $\nu$, if $l(j)\neq l(k)$, we have
    \begin{align*}
        \Ebb[Z_0 X^\mu_i X^\mu_j] &= \frac{1}{2}\Pbb[D_{l(j)}=C_j=D_{l(k)}=C_k=0] + \frac{1}{8}(1-\Pbb[D_{l(j)}=C_j=D_{l(k)}=C_k=0])\\
        &= \frac{1}{2} - \frac{3}{8}(1-(1-\gamma_{l(j)})(1-\delta_{l(j)})(1-\gamma_{l(k)})(1-\delta_{l(k)}))  \\
        &= \frac{1}{2} - \frac{3}{8}(2^{-l(j)} + 2^{-l(k)} - 2^{-l(j)-l(k)}).
    \end{align*}
    Last, if $l(j)=l(k)$,
    \begin{align*}
        \Ebb[Z_0 X^\nu_i X^\nu_j] &= \frac{1}{2} \Pbb[D_{l(j)}=C_j=C_k=0] + \frac{1}{8}\Pbb[D_{l(j)}=0,C_j=C_k=1]\\
        &\qquad \qquad + \frac{1}{4} \Ebb[\1_{D_{l(j)}=1} + \1_{D_{l(j)}=0}(\1_{C_j=0}\1_{C_k=1}+ \1_{C_j=1}\1_{C_k=0})] \\
        &= \frac{1}{2} - \frac{4\gamma_{l(j)}+2\delta_{l(j)} - 4\gamma_{l(j)}\delta_{l(j)} - \gamma_{l(j)}^2 +\gamma_{l(j)}^2\delta_{l(j)}}{8}\\
        &=\frac{1}{2}-\frac{\gamma_{l(j)}+\delta_{l(j)} - \gamma_{l(j)}\delta_{l(j)}}{4} -\frac{(1-\delta_{l(j))}(2\gamma_{l(j)} - \gamma_{l(j)}^2)}{8} = \frac{1}{2}-\frac{3}{8}2^{-l(j)}.
    \end{align*}
    This concludes the first case.

    \paragraph{Case 2.} By symmetry, the only remaining case is if $l(i)=l(j)=l(k)$. For simplicity, we then denote $l:=l(i)$. Then,
    \begin{equation*}
        \Ebb[X_i^\mu X_j^\mu X_k^\mu] = \frac{1}{2}\Pbb[B_l=0] + \frac{1}{8}\Pbb[B_l=1] = \frac{1}{2} - \frac{3}{8}2^{-l}.
    \end{equation*}
    On the other hand,
    \begin{align*}
        \Ebb[X_i^\nu X_j^\nu X_k^\nu] &= \frac{1}{2}\Ebb[\1_{D_l=1} + \1_{D_l=0}\1_{C_i=0}\1_{C_j=0}\1_{C_k=0}] + \frac{3}{4}\Ebb[\1_{D_l=0}\1_{C_i=1}\1_{C_j=0}\1_{C_k=0}]\\
        &\qquad \qquad + \frac{3}{8}\Ebb[\1_{D_l=0}\1_{C_i=1}\1_{C_j=1}\1_{C_k=0}] + \frac{1}{8}\Ebb[\1_{D_l=0}\1_{C_i=1}\1_{C_j=1}\1_{C_k=1}]\\
        &=\frac{1}{2}- \frac{3(2\gamma_l - 2\gamma_l\delta_l - \gamma_l^2 + \gamma_l^2\delta_l)}{8}\\
        &= \frac{1}{2}- \frac{3(1-\delta_l)(2\gamma_l-\gamma_l^2)}{8} = \frac{1}{2}-\frac{3}{8}2^{-l}.
    \end{align*}
    This ends the proof that all $3$rd-order moments agree for $\mu$ and $\nu$.
\end{proof}

\section{Bounds on the expected maximum empirical deviation with \mathinhead{$\xi$}-covering numbers}
\label{sec:xi_covering_numbers}

The previous result \cref{thm:2procs} shows that having tight characterizations of when $\Delta_n(\mu)$ converges $0$, cannot be achieved by focusing solely on pair-wise correlations of the coordinates. Nevertheless, we are still able to give useful bounds on $\Delta_n(\mu)$ using such information. In this section, we focus on the metric space $(\N,\xi)$ where the metric $\xi$ is defined as in Eq~\eqref{eq:xi-metric} and provide necessary and sufficient conditions for the decay of the expected maximum empirical deviation $\Delta_n(\mu)$.

We start with proving \cref{thm:finite_cov_num} which shows that having finite $\xi$-covering numbers is necessary. 

\vspace{2mm}

\begin{proof}{\textbf{of \cref{thm:finite_cov_num}}}
    Let $\eps>0$ such that the $\eps$-covering number of  $(\N,\xi)$ is infinite and let $S_0\subset \N$ be an infinite set such that for all $i,j\in S_0$, $\xi(i,j) = \Pbb(X_i\neq X_j)\geq \eps$. Further, there must exist an interval $I=[p-\frac{\eta}{2},p+\frac{\eta}{2}]$ of length $\eta>0$ to be fixed later, such that $S_1:=S_0\cap\{i\in\N: \p[i]\in I\}$ is also infinite. Since for distinct $i,j\in S_1$,
    \begin{equation*}
        \eps \leq \xi(i,j) \leq \p[i] + \p[j] \leq 2\p[i] + \eta,
    \end{equation*}
    necessarily, for any $i\in S_1$, one has
    \begin{equation}\label{eq:lower_bound_p}
        \p[i] \geq \frac{\eps-\eta}{2}.
    \end{equation}
    By symmetry,  we also have
    \begin{equation}\label{eq:upper_bound_p}
        1-\p[i] \geq \frac{\eps-\eta}{2}.
    \end{equation}
    We will focus only on the indices in $S_1$, using
    \begin{equation*}
        \Delta_n(\mu) \geq \Ebb \sup_{i\in S} |\pn[i] - \p[i]|.
    \end{equation*}
    Hence, without loss of generality, we will suppose that $S_1=\N$.
    In the rest of the proof, we will use the notations $(X_i^{(1)})_i,(X_i^{(2)})_i,\ldots \overset{\iid}{\sim}\mu$ for an iid sequence of samples of $\mu$. Similarly, we will indicate by an exponent $(n)$ any event corresponding to the sample  sequence $(X_i^{(n)})_i$. For convenience, we will also consider a sample sequence $(X_i)_i\sim\mu$.
    
    Fix an arbitrary infinite subset $S\in\N$. We consider the event that $(X_i)_{i\in S}$ contains an infinite number of 0 and 1,
    \begin{equation*}
        \Ecal(S) :=\set{\min \paren{\sum_{i\in S}X_i, \sum_{i\in S}(1-X_i)} = \infty }.
    \end{equation*}
    On $\Ecal(S)^c$, the sequence $(X_i)_{i\in S}$ contains either a finite number of 0 or 1. We then denote by $\bar X(S)$ the random variable on $\{0,1\}$ equal to the (infinite) majority among the sequence $(X_i)_{i\in S}$, that is,
    \begin{equation*}
        \bar X(S) = \1_{\Ecal(S)^c} \cdot  \1\sqb{\sum_{i\in S}(1-X_i)<\infty }.
    \end{equation*}
    As a first step, we will show that $\Pbb(\Ecal(S))\geq \eps$. By Fatou's lemma, enumerating $S = \{i_1<i_2<\ldots\}$, one has
    \begin{align*}
        \Pbb(\Ecal(S)) &= \Pbb\paren{\limsup_{j\to\infty}  \1[X_{i_j}\neq X_{i_{j+1}}]>0}\\
        &=\Ebb \sqb{ \limsup_{j\to\infty} \1[X_{i_j}\neq X_{i_{j+1}}]  } \\
        &\geq \limsup_{j\to\infty} \Pbb(X_{i_j}\neq X_{i_{j+1}}) \geq \eps.
    \end{align*}
    Now supposing that $\Pbb(\Ecal(S))<1$ we define $\bar p(S) = \Ebb[\bar X(S) \mid \Ecal(S)^c]$ (otherwise, we can set it to $\bar p(S)=p$ for instance), the expected value of the majority vote on $(X_i)_{i\in S}$ provided that there is consensus (only a finite number of disagreements).

    \paragraph{Case 1.} We first consider the case when there exists an infinite subset $S\in\N$ such that for any infinite subset $S'\subset S$, one has $\Pbb(\Ecal(S'))<1$ and $\bar p(S) \leq p$. We now prove a lower bound on $\Delta_n(\mu)$ by showing that with good probability there is an index $i\in S$ for which $\pn[i]$ deviates from $\p[i]$ from below.

    To do so,  given the iid sequence $(X_i^{(1)})_i,(X_i^{(2)})_i,\ldots \overset{\iid}{\sim}\mu$, let $N(S)$ be the first index $n$ for which the event $\Ecal(S)$ holds. For every $n<N(S)$, there are only a finite number of disagreements in the sequences $(X_i^{(n)})_{i\in S}$, while $(X_i^{(n)})_{i\in S}$ contains an infinite number of 0. Hence, provided that $N(S)<\infty$, the set
    \begin{equation*}
        S(1) = \{i\in S, \forall n<N(S), X_i^{(n)}=\bar X^{(n)},\text{ and } X_i^{(N(S))} = 0\}
    \end{equation*}
    is infinite. All indices in $S(1)$ share the same values for the samples $n\in[N(S)]$. We denote that value $X_{S(1)}^{(n)}$ for convenience. We can now repeat the construction process with $S(1)$ under the almost-sure event $\{N(S)<\infty\}$: we can define the geometric random variable $N(S(1))$ which is the waiting time starting from $n=N(S)+1$ for the event $\Ecal(S(1))$ to occur. On the almost sure event $\{N(S)<\infty\}\cap \{N(S(1))<\infty\}$, this defines a new set $S(2)\subset S(1)$ for which all indices in $S(2)$ shared the same values for the samples in $n\in[N(S)+1,N(S)+N(S(1))]$, and we denote by $\bar X_{S(2)}^{(n)}$ these common values. 

    As a result of the construction, on an almost sure event $\Fcal$, we obtain a sequence of geometric random variables $(N(S(k)))_{k\geq 1}$ with parameter $\Pbb(\Ecal(S(k)))\geq \eps$, together with random sets $(S(k))_{k\geq 1}$ that are decreasing and all infinite. Intuitively, $\Fcal = \bigcap_k\{N(S(k))<\infty\}$. For convenience, letting $S(0)=S$, we denote by $N_k = N(S(0))+\ldots + N(S(k))$. Under $\Fcal$, we have that for any $N_{k-1}< n<N_k$,
    \begin{equation*}
        \Ebb[\bar X_{S(k)}^{(n)}\mid N_l,S(l) ,l\leq k] = \bar p(S(k)) \leq p
    \end{equation*}
    On the other hand, under $\Fcal$, we have $\bar X_{S(k)}^{(N_k)}=0$ by construction.
    
    Now fix $n\geq 1$. We recall that on $\Fcal$, there always exists an index $i_n\in\N$ for which $(X_i^{(l)})_{l\in [n]}$ coincides exactly with the sequence $\bar X_{S(0)}^{(1)},\ldots,\bar X_{S(0)}^{(N_0)},\bar X_{S(1)}^{(N_0+1)},\ldots, X_{S(1)}^{(N_1)},\ldots$ truncated at $n$ samples. We denote by $k_n$ the number of finished periods before sample $n$, that is, the index $k$ such that $N_k\leq n < N_{k+1}$. Combining the previous equations gives
    \begin{equation*}
        \Ebb\sqb{\sum_{l=1}^n X_{i_n}^{(l)} } \leq (n-\Ebb[k_n])  p.
    \end{equation*}
    We recall that $k_n$ corresponds to the maximum number of geometric variables with parameter at least $\eps$ such that the sum is at most $n$. Hence, $k_n$ is dominated by the maximum number of $\Gcal(\eps)$ random variables such that the sum is at most $n$. That is, if $(T_k)_{k\geq 1}\overset{\iid}{\sim}\Gcal(\eps)$ and we let $\tilde k_n$ be the index such that $\sum_{l=1}^k T_l \leq n < \sum_{l=1}^{k+1} T_l$, we have $\Ebb[k_n] \leq \Ebb[\tilde k_n].$
    Further, we have that
    \begin{align*}
        n < \Ebb\sqb{\sum_{k=1}^{\tilde k_n+1} T_k} &= \Ebb\sqb{\sum_{k\geq 1} T_k \1_{k\leq \tilde k_n+1}}\\
        &=\sum_{k\geq 1} \frac{1}{\eps} \Pbb(k\leq \tilde k_n+1)\\
        &= \frac{1+\Ebb[\tilde k_n]}{\eps}.
    \end{align*}
    In the second inequality, we used the Wald observation that knowing whether $k\leq \tilde k_n+1$ only requires knowing $T_1,\ldots,T_{k-1}$ (this corresponds exactly to the event $T_1+\ldots +T_{k-1}\leq n$). For small values of $n$, we can use the following simple bound
    \begin{equation*}
        \Ebb[k_n] \geq \Pbb(k_n\geq 1) = \Pbb(T_1\leq n) = 1-(1-\eps)^n.
    \end{equation*}
    This implies that $\Ebb[k_n]\geq \Ebb[\tilde k_n]> (\eps n-1)\lor 1-(1-\eps)^n$. We now show that this implies $\Ebb[k_n]\gtrsim n\eps$. For $n\geq \frac{2}{\eps}$, this already shows $\Ebb[k_n]\geq \frac{\eps n}{2}$. Note that the function $n\mapsto 1-(1-\eps)^n$ is concave. Given that its value is 0 for $n=0$, this gives for $n\leq \frac{2}{\eps}$,
    \begin{equation*}
        1-(1-\eps)^n \geq \frac{n\eps}{2} \paren{1-(1-\eps)^{2/\eps}} \geq \frac{n\eps}{2} (1-\mathe^{-2}) \geq \frac{n\eps}{3}.
    \end{equation*}
    This shows that in all cases, we obtained
    \begin{equation*}
        \Ebb[k_n] \geq \frac{n\eps}{3}.
    \end{equation*}
    In particular, recalling that $\p[i_n]\in [p-\frac{\eta}{2},p+\frac{\eta}{2}]$, and combining the previous equations we obtain
    \begin{align*}
        \Delta_n(\mu) &\geq \Ebb\sqb{  p-\frac{\eta}{2} -\pn[i_n] } \\
        &\geq  \frac{\Ebb[k_n]}{n}  p - \frac{\eta}{2}\\
        &\geq \frac{\eps}{6}(\eps-\eta) -  \frac{\eta}{2}.
    \end{align*}
    In the last inequality, we used the lower bound on $p$ from Eq~\eqref{eq:lower_bound_p}. We now turn to the second case.
    
    \paragraph{Case 2.} We now suppose that for every infinite subset $S\subset \N$, there exists an infinite subset $S'\subset S$ such that either $\Pbb(\Ecal(S'))=1$ or $\bar p(S')\geq p$. We now give a lower bound on $\Delta_n(\mu)$ by showing that with good probability there is an index $i\in\N$ for which $\pn[i]$ deviates from $\p[i]$ by above.

    To do so, we construct some decreasing random sets similarly to Case 1. To begin, there exists an infinite subset $S(0)\subset \N$ such that either $\Pbb(\Ecal(S(0)))=1$ or $\bar p(S(0))\geq p$. As in Case 1, let $N(S(0))$ be the number of samples to wait before the event $\Ecal(S(0))$ occurs. As before, on the event $\{N(S(0))<\infty\}$, the set
    \begin{equation*}
        \tilde S(1) = \{i\in S(0), \forall n<N(S(0)), X_i^{(n)}=\bar X^{(n)},\text{ and } X_i^{(N(S(0)))} = 1\}
    \end{equation*}
    is infinite. Hence, there exists a subset $S(1)\subset \tilde S(1)$ for which either $\Pbb(\Ecal(S(1)))=1$ or $\bar p(S(1))\geq p$. We can now repeat the process starting from $S(1)$. We use the same notations as in Case 1: the induction constructs under an event $\Fcal$ of full probability, some decreasing infinite sets $(S(k))_{k\geq 1}$, as well as their sequence of geometric random variables $(N(S(k)))_{k\geq 1}$ with parameter at least $\eps$. We denote $N_k = N(S(0))+\ldots + N(S(k))$. Similarly to before, under $\Fcal$, we have that for any $N_{k-1}<n<N_k$,
    \begin{equation*}
        \Ebb[\bar X_{S(k)}^{(n)}\mid N_l,S(l) ,l\leq k] = \bar p(S(k)) \geq  p,
    \end{equation*}
    and for any $k\geq 1$, by construction $\bar X_{S(k)}^{(N_k)}=1$. Hence, for a fixed number of samples $n\geq 1$, there exists an index $i_n$ for which until $n$ the values $(X_{i_n}^l)_{l\in[n]}$ coincide with that of $\bar X$. As before, if $k_n$ is the index $k$ for which $N_k\leq n<N_{k+1}$, we obtain
    \begin{equation*}
        \Ebb\sqb{\sum_{l=1}^n X_{i_n}^{(l)} } \geq np + (1-p)\Ebb[k_n].
    \end{equation*}
    As a result, we obtain
    \begin{align*}
        \Delta_n(\mu) \geq \Ebb [\pn[i_n] - \p[i_n]]
        &\geq \Ebb[\pn[i_n] - p]-\frac{\eta}{2}\\
        &\geq \frac{\Ebb[k_n]}{n}(1-p) - \frac{\eta}{2}\\
        &\geq \frac{\eps}{6}(\eps-\eta) -  \frac{\eta}{2}.
    \end{align*}
    In the last inequality, we used Eq~\eqref{eq:upper_bound_p}. Combining the two cases and noting that these hold for any value of $\eta>0$ yields the desired result.
\end{proof}

We next show that essentially is the tightest necessary conditions that can be obtained using only the covering number $\Ncal_\xi(\cdot)$.

\vspace{2mm}

\begin{proof}{\textbf{of Proposition~\ref{prop:ex_fast_converging}}}
    Fix such a non-increasing function $N:(0,1]\to\Nbb$. We start by constructing a distribution $\mu$, then we check that it has the same covering numbers as $N(\cdot)$. Last we prove the desired convergence $\Delta_n(\mu)\ninf  0$.

    \paragraph{Constructing the distribution $\mu$.} Since we will only focus on covering numbers for $\eps\in(0,\frac{1}{2}]$ anyways, without loss of generality, we suppose that $N(\eps)=1$ for $\eps\in(\frac{1}{2},1]$. Note that because $N$ is non-increasing, it is discontinuous on a countable (potentially finite) number of points $(\eps_k)_{k\geq 1}$, which we ordered by decreasing order, that is $0<\eps_{k+1}<\eps_k\leq\frac{1}{2}$ for all $k\geq 1$. For convenience, let $\eps_0=\frac{1}{2}$. Since $N$ takes values only on integers $\N$, we can also define the sequence $(N_k)_{k\geq 1}$ such that $N_k$ is the value taken by $N$ on the interval $(\eps_{k-1},\eps_k)$, with the convention $N_1=1$ in the specific case when $\eps_0=\eps_1$.

    Let $(Z_i)_{i\geq 1}\overset{iid}{\sim}\Bernu(\frac{1}{2})$ an iid sequence of Bernoulli random variables and $(U_k)_{k\geq 1}\overset{iid}{\sim}\Ucal([0,1])$ be an independent iid sequence of uniform random variables in $[0,1]$. We construct recursively the sequence $(X_i)_{i\geq 1}$ together with a sequence of decreasing events $(\Ecal_k)_{k\geq 1}$. These will be constructed by level $k\geq 1$ corresponding to the points of discontinuity of the covering function $(\eps_k)_{k\geq 1}$. We first pose $X_1 = Z_1$ and $\Ecal_0 = \Omega$ the full event. Now for $k\geq 1$, suppose we have constructed the random variables $X_i$ for $i\leq N_k$ together with decreasing events $\Ecal_l$ for $l\leq k-1$ that $\Pbb(\Ecal_l) = 2\eps_l$.
    
    We then construct the event $\Ecal_k$ via
    \begin{equation*}
        \Ecal_k := \Ecal_{k-1} \cap \set{ U_k \leq \frac{\eps_{k}}{\eps_{k-1}} }
    \end{equation*}
    and new random variables as follows,
    \begin{equation*}
        X_i := \1[\Ecal_k^c] Z_1 + \1[\Ecal_k] Z_i,\quad N_k<i\leq N_{k+1}.
    \end{equation*}
    We can directly check that $\Pbb(\Ecal_k) = \Pbb(\Ecal_{k-1}) \Pbb(U_k\leq \eps_k/\eps_{k-1}) = 2\eps_k$, which concludes the construction.

    We first note that $\Ecal_k$ only depends on the variables $U_l$ for $l\leq k$ and as such is independent of $Z_1$. In particular, this shows that $X_i\sim\Bernu(\frac{1}{2})$ for all $i\geq 1$. Also, for $i\geq 2$ with $N_k<i\leq N_{k+1}$ for some $k\geq 2$, this implies
    \begin{equation*}
        \xi(i,1) = \Pbb(X_i\neq X_1) = \frac{\Pbb(E_k)}{2} = \eps_k.
    \end{equation*}
    We next compute distances between any two distinct components $i<j\geq 2$. We let $k_i,k_j\geq 1$ such that $N_{k_i}<i\leq N_{k_i+1}$ and similarly for $j$.
    \begin{equation*}
        \xi(i,j) = \Pbb(X_i\neq X_j)  = \frac{\Pbb(\Ecal_{k_i}\cup\Ecal_{k_j})}{2} = \frac{\Pbb(\Ecal_{k_i})}{2} = \eps_{k_i}.
    \end{equation*}
    In the third equality, we used the fact that the events $(\Ecal_k)_{k\geq 0}$ are decreasing.

    \paragraph{Computing the covering numbers of $\mu$} We clearly have $\Ncal_\xi(\frac{1}{2})=1$ because $\xi(1,i)\leq \frac{1}{2}$ for all $i\geq 1$. Next, for any fixed $\eps\in(0,\frac{1}{2})$, let $k\geq 1$ such that $\eps\in[\eps_k,\eps_{k-1})$. We first note that the random variables $[N_k]$ form an $\eps$-cover of $(\N,\xi)$. Indeed, for $i>N_k$, if $N_{k_i}<i\leq N_{k_i+1}$, one has $k_i\geq k$ so that
    \begin{equation*}
        \xi(1,i) = \eps_{k_i}\leq \eps_k \leq \eps.
    \end{equation*}
    As a result, $\Ncal_\xi(\eps) \leq N_{k-1}$. On the other hand, for any $2\leq i,\leq N_k$, we observe that for any $j\neq i$, $\xi(i,j)\geq \eps_{k-1}>\eps$. As a result, an $\eps$-cover of $(\Nbb,\xi)$ must contain all elements $\{2,\ldots,N_k\}$ which has $N_k-1$ elements. Note that this set does not $\eps$-cover the element $1$ since $\xi(1,i)\geq \eps_{k-1}>\eps$ for all $i\in\{2,\ldots,N_k\}.$ Hence, the $\eps$-cover must have at least $N_k$ elements. Together with the previous remark, we obtained
    \begin{equation*}
        \Ncal_\xi(\eps) = N_k,\quad \eps\in[\eps_k,\eps_{k-1}).
    \end{equation*}
    As a result, wit $E = \{\eps_k,k\geq 1\}$, for any $\eps\in(0,\frac{1}{2}]\setminus E$, we obtained $\Ncal_\xi(\eps) = N(\eps)$.

    \paragraph{Proving the convergence.} We show that $\Delta_n(\mu)\ninf 0$ by checking that it satisfies the condition from \cref{thm:more_general_suff_condition}, proved later in \ref{sec:examples}. The proof of that result is completely separate so that there is no circular logic. We take $K=1$ and for $\eps>0$, fix $k\geq 1$ such that $\eps_k\leq \eps/2$. We simply take one event $E_1 = \Ecal_k$, and we use the centers $J=[N_k]$. For any $i>N_k$, letting $k_i\geq k$ such that $N_{k_i}<i\leq N_{k_i+1}$, we indeed have
    \begin{equation*}
        \{X_i\neq X_1\} \subset \Ecal_{k_i}\subset \Ecal_k=E_1.
    \end{equation*}
    This ends the proof of the proposition.
\end{proof}

Because of the necessary condition in \cref{thm:finite_cov_num}, there is no loss of generality
in assuming
that for any $\eps>0$, the $\eps$-covering number for $(\N,\xi)$ is finite: $\Ncal_\xi(\eps)<\infty$. On the other hand, \cref{thm:sufficient_condition_covering_nb} shows that the condition given in Eq~\eqref{eq:sufficient_condition}, which we restate here for convenience,
\begin{equation*}
    \int_0^1 \Ncal_\xi (\eps)\mathd\eps<\infty,
\end{equation*}
is a sufficient condition for $\Delta_n(\mu)\to 0$.

\vspace{2mm}

\begin{proof}{\textbf{of \cref{thm:sufficient_condition_covering_nb}}}
    We first note that this condition is equivalent to $\sum_{k\geq 0} 2^{-k} \Ncal_\xi (2^{-k})<\infty$. Indeed, $\Ncal_\xi(\cdot)$ is non-increasing, hence
    \begin{equation*}
        \sum_{k\geq 0} 2^{-k-1} \Ncal_\xi (2^{-k}) \leq \int_0^1 \Ncal_\xi(\eps)\mathd\eps \leq \sum_{k\geq 1} 2^{-k} \Ncal_\xi (2^{-k}).
    \end{equation*}
    To obtain our bounds on $\Delta_n(\mu)$, we will use chaining techniques. First, let $\Scal_\xi(\eps)$ be an $\eps$-covering of $(\N,\xi)$ with minimal cardinality $\Ncal_\xi(\eps)$. For any $k\geq 1$ and $i\in\Scal_\xi(2^{-k})$, we denote by $\hat i$ the $\xi$-nearest neighbor of $i$ within $\Scal_\xi(2^{-k+1})$. In particular, by definition of $\Scal_\xi(2^{-k+1})$ we have $\xi(i,\hat i)\leq 2^{-k+1}$.

    Fix $\eps\in(0,1]$ and consider $(X_i)_{i\geq 1}\sim\mu$. By hypothesis, there exists $k_\eps\geq \log_2\frac{1}{\eps}$ such that
    \begin{equation*}
        \sum_{k\geq k_\eps} 2^{-k} \Ncal_\xi (2^{-k}) \leq \eps.
    \end{equation*}
    Then,
    \begin{align*}
        \Pbb\paren{\exists i\in \bigcup_{k\geq k_\eps} \Scal_\xi(2^{-k}) , X_i\neq X_{\hat i}} &\leq
        \Ebb\sqb{ \abs{\set{i\in \bigcup_{k\geq k_\eps} \Scal_\xi(2^{-k}): X_i\neq X_{\hat i}}}} \\
        &= \sum_{k\geq k_\eps} \sum_{i\in \Scal_\xi(2^{-k})} \Pbb(X_i\neq X_{\hat i}) \\
        &\leq  \sum_{k\geq k_\eps} 2^{-k} \Ncal_\xi (2^{-k}) \leq \eps. 
    \end{align*}
    Hence, if $\Ecal_\eps:= \{\forall i\in \bigcup_{k\geq k_\eps} \Scal_\xi(2^{-k}) , X_i= X_{\hat i} \}$, we have $\Pbb(\Ecal_\eps) \geq 1-\eps$.

    For any $i\in \Scal:=\bigcup_{k\geq 0}\Scal_\xi(2^{-k})$, let $k\geq 0$ be such that $i\in\Scal_\xi(2^{-k})\setminus \Scal_\xi(2^{-k+1})$, with the convention $\Scal_\xi(2)=\emptyset$. We can then construct the sequence $i_k=i,i_{k-1},\ldots, i_0$ such that $i_{p-1}=\hat i_p$. We denote by $i(\eps)$ the first element of this list within $ \Scal_\eps:=\bigcup_{k\leq k_\eps}\Scal_\xi(2^{-k})$, that is $i(\eps) = i_{k_\eps\land k}$. Note that by the triangle inequality,
    \begin{equation*}
        \xi(i,i(\eps))\leq \xi(i_k,i_{k-1}) + \ldots + \xi(i_{k_\eps\land k+1}, i_{k_\eps\land k}) \leq \sum_{k-1\leq l\leq k_\eps} 2^{-l} \leq 2^{-k_\eps +1} \leq 2\eps.
    \end{equation*}
    Also, under $\Ecal_\eps$, for any $i\in\Scal$, one has $X_i = X_{i(\eps)}$.
   
    We now focus on indices in $i\notin \Scal$ and aim to prove an equivalent equation. Fix any $k\geq 0$, we denote by $i_k$ the $\xi$-nearest neighbor of $i$ within $\Scal_\xi(2^{-k})$. Because $\Scal_\eps$ is finite, we can fix some element that we denote $i(\eps)$ that appears infinitely often in the sequence $(i_k(\eps))_{k\geq 0}$. In particular, for any $k\geq 0$ such that $i_k(\eps)=i(\eps)$, we have
    \begin{equation*}
        \xi(i,i(\eps))\leq \xi(i,i_k) + \xi(i_k,i_k(\eps)) \leq 2^{-k} + 2\eps.
    \end{equation*}
    Because this holds for an infinite number of indices $k\geq 0$, this shows that $\xi(i,i(\eps))\leq 2\eps$. Next, since $\Pbb(X_i\neq X_{i(k)})\leq 2^{-k}$ for $k\geq 0$, which forms a summable sequence, by the Borel-Cantelli lemma and the union bound, the following event
    \begin{equation*}
        \Fcal:= \{\forall i\notin \Scal, \exists \hat k_i\geq 0,  \forall k\geq \hat k_i, X_i = X_{i(k)} \},
    \end{equation*}
    has probability one.
    As a result, on $\Ecal_\eps \cap \Fcal$, the sequence $(X_{i(k)})_{k\geq 0}$ is equal to $X_i$ for $k$ large enough but also contains an infinite number of times the value $X_{i(\eps)}$. Hence $X_i=X_{i(\eps)}$. In summary, we obtained
    \begin{equation*}
        \Ecal_\eps\cap\Fcal\subset \{\forall i\geq 1, X_i = X_{i(\eps)}\},
    \end{equation*}
    and for all $i\geq 1$, $\Pbb(X_i\neq X_{i(\eps)}) = \xi(i,i(\eps))\leq 2\eps$.

    Now consider iid samples $(X_i^{(n)})_{i\geq 1}\sim \mu$ for $n\geq 1$ and denote by $\Ecal_\eps^{(n)}$ and $\Fcal^{(n)}$ the corresponding event. In particular, $(\mathbbm{1}(\Ecal_\eps^{(n)}\cap\Fcal^{(n)}))_{n\geq 1}$ is an iid Bernoulli sequence of parameter $\Pbb(\Ecal)$. For any $i\geq 1$,
    \begin{align*}
        |\pn[i] - \p[i]| &\leq |\pn[i] - \pn[i(\eps)]| + |\pn[i(\eps)] - \p[i(\eps)]| + |\p[i(\eps)]-\p[i]|\\
        &\leq \frac{1}{n}\sum_{m=1}^n \mathbbm{1}((\Ecal_\eps^{(m)}\cap\Fcal^{(m)})^c)  + \sup_{j\in\Scal_\eps} |\pn[j] - \p[j]|  +  \Pbb(X_i\neq X_{i(\eps)})\\
        &\leq \frac{1}{n}\sum_{i=1}^n \mathbbm{1}((\Ecal_\eps^{(m)}\cap\Fcal^{(m)})^c)  + \sup_{j\in\Scal_\eps} |\pn[j] - \p[j]|  +  2\eps.
    \end{align*}
    As a result,
    \begin{align*}
        \Delta_n(\mu) = \Ebb \sup_{i\geq 1} |\pn[i] - \p[i]| &\leq (1-\Pbb(\Ecal)) + \Ebb \sup_{j\in\Scal_\eps} |\pn[j] - \p[j]| + 2\eps \\
        &\leq \Ebb \sup_{j\in\Scal_\eps} |\pn[j] - \p[j]| + 3\eps.
    \end{align*}
    In particular, this gives $\limsup_{n\to\infty}\Delta_n(\mu) \leq 3\eps$ because $\Scal_\eps$ is finite. This holds for any $\eps$, which ends the proof that $\Delta_n(\mu)\to 0$ as $n \to \infty$.
    
    Additionally, the previous equation holds for any $\eps$, hence
    \begin{equation*}
        \Delta_n(\mu) \leq \inf_{\eps\in(0,1]} 3\eps + \Ebb \sup_{j\in\Scal_\eps} |\hat \p[j] - \p[j]|.
    \end{equation*}
    The right-most term can be upper bounded using the upper bound on maximum empirical mean deviations for general distributions on $\{0,1\}^\N$ from \citet[Corollary 3]{blanchard2023tight}. Precisely, we need to order the elements $j\in \Scal_\eps$ by decreasing order of $\p[j]\land (1-\p[j])$. The worst case upper bound is achieved when all these mean probabilities are equal to $\frac{1}{2}$. Hence, \citet[Corollary 3]{blanchard2023tight} yields
    \begin{equation*}
        \Ebb \sup_{j\in\Scal_\eps} |\hat \p[j] - \p[j]| \lesssim 1\land \paren{\sqrt{\frac{\log (1+|\Scal_\eps|)}{n}} + \frac{\log(1+|\Scal_\eps|)}{n\log\paren{2+\frac{2\log(1+|\Scal_\eps|)}{n}}}} \asymp 1\land \sqrt{\frac{\log (1+|\Scal_\eps|)}{n}}.
    \end{equation*}
    Putting the upper bounds together yields the following bound
    \begin{equation*}
        \Delta_n(\mu) \lesssim \inf_{\eps\in(0,1]} \eps + \sqrt{\frac{\log (\Ncal_\xi(\eps)+1)}{n}} \leq   \inf_{\eps\in(0,1]} \eps + \sqrt{\frac{\log (1+\frac{1}{\eps}\int_0^1\Ncal_\xi(\eta)d\eta)}{n}}
    \end{equation*}
    Recalling the notation $C_\mu:=\int_0^1\Ncal_\xi(\eta)d\eta$, and using the value of $\eps_{\mu,n}:= 1\land \sqrt{\log(n(1+C_\mu))/n}$, we then obtain $\Delta_n(\mu)\leq 1$ if $\eps_{\mu,n}=1$, or if $\eps_n<1$,
    \begin{equation*}
        \Delta_n(\mu) \lesssim   \eps_{\mu,n} + \sqrt{\frac{\log (1+\frac{1}{\eps_{\mu,n}}C_\mu)}{n}} \asymp \sqrt{\frac{\log(n(1+C_\mu))}{n}}.
    \end{equation*}
    This ends the proof of the proposition.
\end{proof}

We next show that this sufficient condition, Eq~\eqref{eq:sufficient_condition}, is as tight as can be using the covering numbers $\Ncal_\xi(\cdot)$. We recall that this cannot be a necessary condition in view of \cref{thm:2procs}---instead, we show in Proposition~\ref{prop:example_covering_nb} that if the covering numbers do not satisfy Eq~\eqref{eq:sufficient_condition}, one can construct some distribution with (almost) the same covering numbers but for which the expected maximum deviation does not converge to $0$.

\vspace{2mm}

\begin{proof}{\textbf{of Proposition~\ref{prop:example_covering_nb}}}
    The proof has three steps, first we define the distribution $\mu$, then we prove that its covering numbers coincide with $N(\cdot)$, then we show that $\Delta_n(\mu)\to \frac{1}{2}$.
    
    \paragraph{Constructing the distribution $\mu$.} Fix the non-increasing function $N:(0,1]\to\N$. We use similar notations as in the proof of Proposition~\ref{prop:ex_fast_converging}. Since we will only focus on covering numbers for $\eps\in(0,\frac{1}{2}]$ anyways, without loss of generality we suppose that $N(\eps)=1$ for $\eps\in(\frac{1}{2},1]$. Given such a function $N$, we start by constructing an equidistant directed tree (with edge lengths) representing the function. Note that because $N$ is non-increasing, it is discontinuous on a countable (potentially finite) number of points $(\eps_k)_{k\geq 1}$, which we ordered by decreasing order, that is $0<\eps_{k+1}<\eps_k\leq\frac{1}{2}$ for all $k\geq 1$. For convenience, let $\eps_0=\frac{1}{2}$. Since $N$ takes values only on integers $\N$, we can also define the sequence $(N_k)_{k\geq 1}$ such that $N_k$ is the value taken by $N$ on the interval $(\eps_{k-1},\eps_k)$, with the convention $N_1=1$ in the specific case when $\eps_0=\eps_1$.

    We construct the tree by recursion, starting for $k=0$ with only a root denoted $v(0,1)$. For context, inner nodes will be denoted $v(k,p)$ where $k$ will correspond to level $\eps_k$ and $p$ will correspond to the index of the node by order of construction.
    Now suppose that we have constructed the tree up to level $k\geq 0$ and that we have constructed a total of $N_k$ nodes. At level $k+1$, we construct $N_{k+1}-N_k$ new nodes. If for $k\in[N_{k+1}-N_k]$, we link the node $v(k+1,N_k+l)$ to some node $v(k',p')$, the length of the edge is set to $\eps_{k'} - \eps_{k+1}$. 
    Deciding of which node to link to the new nodes at level $k+1$ is done in a specific manner to balance the construction of the overall tree. A formal construction of the tree is given in the pseudo-code Algorithm~\ref{alg:construct_tree}. Intuitively, the construction of the tree emulates the construction of a full binary tree: if we had $N_{k+1}-N_k=1$ for all $k\geq 1$, the output tree would exactly be a binary tree that is constructed layer by layer in order. Because the jumps $N_{k+1}-N_k$ may be larger, the tree is instead the binary tree where some edges are collapsed. To keep at all times a balance in the binary tree, splits are added according to a fractal manner. For layer $r$, we split $2^r$ edges which are denoted $e(r,s)$ for $s\in\{0,1,\ldots,2^r-1\}$, in the following order:
    \begin{align*}
        \Order(0) &:= (0),\\
        \Order(r) &:= (2i, i\in \Order(r-1)) \cup (2i+1, i\in \Order(r-1)) ,\quad r\geq 1.
    \end{align*}
    For instance, $\Order(1) = (0,1)$, $\Order(2) = (0,2,1,3)$, and $\Order(3) = (0,4,2,6,1,5,3,7)$. A visualization of the trees constructed for two different covering number functions $N(\cdot)$ are given in Figure~\ref{fig:skeleton_tree}, one for the simpler case when $N_{k+1}-N_k=1$ for all $k\geq 1$ and one for the general case.
    
    \begin{algorithm}
    
    \hrule height\algoheightrule\kern3pt\relax
    \caption{Constructing the tree skeleton for the distribution in Proposition~\ref{prop:example_covering_nb}}\label{alg:construct_tree}
    \hrule height\algoheightrule\kern3pt\relax
    
    \KwData{$(\eps_k)_{k\geq 1}, (N_k)_{k\geq 1}$}
    \KwResult{An equidistant skeleton tree $\Tcal$}
    Initialize $\Tcal$ as a root $v(0,1)$ at level $\frac{1}{2}$ with an exiting edge denoted $e(0,0)$\;

    $k\gets 1$, $n\gets 1$ and $m\gets 1$\;
    
    \For{$r\geq 0$}{
        \For{$s\in \Order(r)$}{
            Split edge $e(r,s)$ in two at level $\eps_k$. That is:

            \eIf{the top end node of edge $e(r,s)$ is a node $v(k',n')$ with $k'<k$}{
                Create a node $v(k,n+1)$ at level $\eps_k$ to end edge $e(r,s)$: $e(r,s)$ has length $\eps_{k'}-\eps_k$\;

                $n\gets n+1$\;

                Create two edges $e(r+1,2s), e(r+1,2s+1)$ exiting from node $v(k,n+1)$\;
            }{
                Delete edge $e(r,s)$ and create two edges $e(r+1,2s), e(r+1,2s+1)$ exiting from node $v(k',n') = v(k,n')$
            }
            $m\gets m+1$\;
            
            \lIf{$m=N_{k+1}$}{
                $k\gets k+1$
            }
        }
    }

    \hrule
    \end{algorithm}

    \begin{figure}[ht]
        \centering

\tikzset{every picture/.style={line width=0.75pt}}       

\begin{tikzpicture}[x=0.75pt,y=0.75pt,yscale=-1,xscale=1]

\draw [line width=1.5]    (28.25,285.5) -- (28.25,32.47) ;
\draw [shift={(28.25,29.47)}, rotate = 90] [color={rgb, 255:red, 0; green, 0; blue, 0 }  ][line width=1.5]    (14.21,-4.28) .. controls (9.04,-1.82) and (4.3,-0.39) .. (0,0) .. controls (4.3,0.39) and (9.04,1.82) .. (14.21,4.28)   ;

\draw [line width=1.5]    (23.15,54.23) -- (33.35,54.23) ;

\draw [line width=1.5]    (23.15,82.57) -- (33.35,82.57) ;

\draw [line width=1.5]    (23.15,108.65) -- (33.35,108.65) ;

\draw [line width=1.5]    (23.15,148.89) -- (33.35,148.89) ;

\draw [line width=1.5]    (23.15,277.86) -- (33.35,277.86) ;

\draw  [fill={rgb, 255:red, 0; green, 0; blue, 0 }  ,fill opacity=1 ] (168.63,54.89) .. controls (168.63,53.33) and (169.9,52.06) .. (171.47,52.06) .. controls (173.03,52.06) and (174.3,53.33) .. (174.3,54.89) .. controls (174.3,56.46) and (173.03,57.73) .. (171.47,57.73) .. controls (169.9,57.73) and (168.63,56.46) .. (168.63,54.89) -- cycle ;

\draw  [fill={rgb, 255:red, 0; green, 0; blue, 0 }  ,fill opacity=1 ] (168.63,82.57) .. controls (168.63,81.01) and (169.9,79.74) .. (171.47,79.74) .. controls (173.03,79.74) and (174.3,81.01) .. (174.3,82.57) .. controls (174.3,84.14) and (173.03,85.41) .. (171.47,85.41) .. controls (169.9,85.41) and (168.63,84.14) .. (168.63,82.57) -- cycle ;

\draw    (171.47,54.89) -- (171.47,82.57) ;

\draw  [fill={rgb, 255:red, 0; green, 0; blue, 0 }  ,fill opacity=1 ] (110.25,108.65) .. controls (110.25,107.08) and (111.52,105.81) .. (113.08,105.81) .. controls (114.65,105.81) and (115.92,107.08) .. (115.92,108.65) .. controls (115.92,110.21) and (114.65,111.48) .. (113.08,111.48) .. controls (111.52,111.48) and (110.25,110.21) .. (110.25,108.65) -- cycle ;

\draw  [fill={rgb, 255:red, 0; green, 0; blue, 0 }  ,fill opacity=1 ] (234.95,120.93) .. controls (234.95,119.37) and (236.22,118.1) .. (237.79,118.1) .. controls (239.35,118.1) and (240.62,119.37) .. (240.62,120.93) .. controls (240.62,122.5) and (239.35,123.76) .. (237.79,123.76) .. controls (236.22,123.76) and (234.95,122.5) .. (234.95,120.93) -- cycle ;

\draw    (113.08,108.65) -- (171.47,82.57) ;

\draw    (171.47,82.57) -- (237.79,120.93) ;

\draw  [fill={rgb, 255:red, 0; green, 0; blue, 0 }  ,fill opacity=1 ] (75.86,148.89) .. controls (75.86,147.33) and (77.13,146.06) .. (78.7,146.06) .. controls (80.26,146.06) and (81.53,147.33) .. (81.53,148.89) .. controls (81.53,150.46) and (80.26,151.73) .. (78.7,151.73) .. controls (77.13,151.73) and (75.86,150.46) .. (75.86,148.89) -- cycle ;

\draw [line width=1.5]    (23.15,120.93) -- (33.35,120.93) ;

\draw    (78.7,148.89) -- (113.08,108.65) ;

\draw  [fill={rgb, 255:red, 0; green, 0; blue, 0 }  ,fill opacity=1 ] (137.08,172.7) .. controls (137.08,171.14) and (138.35,169.87) .. (139.91,169.87) .. controls (141.48,169.87) and (142.75,171.14) .. (142.75,172.7) .. controls (142.75,174.27) and (141.48,175.53) .. (139.91,175.53) .. controls (138.35,175.53) and (137.08,174.27) .. (137.08,172.7) -- cycle ;

\draw  [fill={rgb, 255:red, 0; green, 0; blue, 0 }  ,fill opacity=1 ] (204.53,161.36) .. controls (204.53,159.8) and (205.8,158.53) .. (207.37,158.53) .. controls (208.93,158.53) and (210.2,159.8) .. (210.2,161.36) .. controls (210.2,162.93) and (208.93,164.2) .. (207.37,164.2) .. controls (205.8,164.2) and (204.53,162.93) .. (204.53,161.36) -- cycle ;

\draw  [fill={rgb, 255:red, 0; green, 0; blue, 0 }  ,fill opacity=1 ] (268.56,185.36) .. controls (268.56,183.79) and (269.82,182.53) .. (271.39,182.53) .. controls (272.96,182.53) and (274.22,183.79) .. (274.22,185.36) .. controls (274.22,186.93) and (272.96,188.19) .. (271.39,188.19) .. controls (269.82,188.19) and (268.56,186.93) .. (268.56,185.36) -- cycle ;

\draw  [fill={rgb, 255:red, 0; green, 0; blue, 0 }  ,fill opacity=1 ] (52.62,205.2) .. controls (52.62,203.63) and (53.89,202.36) .. (55.46,202.36) .. controls (57.02,202.36) and (58.29,203.63) .. (58.29,205.2) .. controls (58.29,206.76) and (57.02,208.03) .. (55.46,208.03) .. controls (53.89,208.03) and (52.62,206.76) .. (52.62,205.2) -- cycle ;

\draw  [fill={rgb, 255:red, 0; green, 0; blue, 0 }  ,fill opacity=1 ] (185.83,213.13) .. controls (185.83,211.57) and (187.1,210.3) .. (188.66,210.3) .. controls (190.23,210.3) and (191.49,211.57) .. (191.49,213.13) .. controls (191.49,214.7) and (190.23,215.97) .. (188.66,215.97) .. controls (187.1,215.97) and (185.83,214.7) .. (185.83,213.13) -- cycle ;

\draw  [fill={rgb, 255:red, 0; green, 0; blue, 0 }  ,fill opacity=1 ] (117.81,227.87) .. controls (117.81,226.31) and (119.08,225.04) .. (120.64,225.04) .. controls (122.21,225.04) and (123.48,226.31) .. (123.48,227.87) .. controls (123.48,229.44) and (122.21,230.71) .. (120.64,230.71) .. controls (119.08,230.71) and (117.81,229.44) .. (117.81,227.87) -- cycle ;

\draw  [fill={rgb, 255:red, 0; green, 0; blue, 0 }  ,fill opacity=1 ] (245.34,232.97) .. controls (245.34,231.41) and (246.61,230.14) .. (248.18,230.14) .. controls (249.74,230.14) and (251.01,231.41) .. (251.01,232.97) .. controls (251.01,234.54) and (249.74,235.81) .. (248.18,235.81) .. controls (246.61,235.81) and (245.34,234.54) .. (245.34,232.97) -- cycle ;

\draw  [fill={rgb, 255:red, 0; green, 0; blue, 0 }  ,fill opacity=1 ] (90.6,243.74) .. controls (90.6,242.18) and (91.87,240.91) .. (93.43,240.91) .. controls (95,240.91) and (96.27,242.18) .. (96.27,243.74) .. controls (96.27,245.31) and (95,246.58) .. (93.43,246.58) .. controls (91.87,246.58) and (90.6,245.31) .. (90.6,243.74) -- cycle ;

\draw  [fill={rgb, 255:red, 0; green, 0; blue, 0 }  ,fill opacity=1 ] (220.4,250.55) .. controls (220.4,248.98) and (221.67,247.71) .. (223.24,247.71) .. controls (224.8,247.71) and (226.07,248.98) .. (226.07,250.55) .. controls (226.07,252.11) and (224.8,253.38) .. (223.24,253.38) .. controls (221.67,253.38) and (220.4,252.11) .. (220.4,250.55) -- cycle ;

\draw  [fill={rgb, 255:red, 0; green, 0; blue, 0 }  ,fill opacity=1 ] (154.08,253.95) .. controls (154.08,252.38) and (155.35,251.11) .. (156.92,251.11) .. controls (158.48,251.11) and (159.75,252.38) .. (159.75,253.95) .. controls (159.75,255.51) and (158.48,256.78) .. (156.92,256.78) .. controls (155.35,256.78) and (154.08,255.51) .. (154.08,253.95) -- cycle ;

\draw  [fill={rgb, 255:red, 0; green, 0; blue, 0 }  ,fill opacity=1 ] (285.59,260.18) .. controls (285.59,258.62) and (286.86,257.35) .. (288.42,257.35) .. controls (289.99,257.35) and (291.26,258.62) .. (291.26,260.18) .. controls (291.26,261.75) and (289.99,263.02) .. (288.42,263.02) .. controls (286.86,263.02) and (285.59,261.75) .. (285.59,260.18) -- cycle ;

\draw    (55.46,205.2) -- (78.7,148.89) ;

\draw    (93.43,243.74) -- (78.7,148.89) ;

\draw    (120.64,227.87) -- (139.91,172.7) ;

\draw    (156.92,253.95) -- (139.91,172.7) ;

\draw    (139.91,172.7) -- (113.08,108.65) ;

\draw    (237.79,120.93) -- (207.37,161.36) ;

\draw    (271.39,185.36) -- (237.79,120.93) ;

\draw    (188.66,213.13) -- (207.37,161.36) ;

\draw    (223.24,250.55) -- (207.37,161.36) ;

\draw    (271.39,185.36) -- (248.18,232.97) ;

\draw    (271.39,185.36) -- (288.42,260.18) ;

\draw  [dash pattern={on 4.5pt off 4.5pt}]  (28.25,82.57) -- (171.47,82.57) ;

\draw  [dash pattern={on 4.5pt off 4.5pt}]  (28.25,120.55) -- (177.89,120.82) -- (237.79,120.93) ;

\draw  [dash pattern={on 4.5pt off 4.5pt}]  (28.81,185.55) -- (271.39,185.36) ;

\draw  [dash pattern={on 4.5pt off 4.5pt}]  (27.93,260.18) -- (288.42,260.18) ;

\draw    (45.16,265.62) -- (55.46,205.2) ;

\draw    (67.36,264.53) -- (55.46,205.2) ;

\draw    (85.34,266.05) -- (93.43,243.74) ;

\draw    (99.1,264.53) -- (93.43,243.74) ;

\draw    (113.84,264.53) -- (120.64,227.87) ;

\draw    (128.58,264.53) -- (120.07,227.87) ;

\draw    (152.45,267.33) -- (156.92,253.95) ;

\draw    (162.28,264.77) -- (156.92,253.95) ;

\draw    (179.02,264.53) -- (188.66,213.13) ;

\draw    (196.6,264.53) -- (188.66,213.13) ;

\draw    (217.85,264.34) -- (223.24,250.55) ;

\draw    (228.96,265.62) -- (223.24,250.55) ;

\draw    (242.51,264.53) -- (248.18,232.97) ;

\draw    (255.55,264.53) -- (248.18,235.81) ;

\draw    (284.77,267.24) -- (288.42,260.18) ;

\draw    (292.1,267.97) -- (288.42,260.18) ;

\draw  [dash pattern={on 0.84pt off 2.51pt}]  (45.16,265.62) -- (43.02,277.86) ;

\draw  [dash pattern={on 0.84pt off 2.51pt}]  (67.21,265.65) -- (69.52,277.86) ;

\draw  [dash pattern={on 0.84pt off 2.51pt}]  (85.34,266.05) -- (81.49,277.86) ;

\draw  [dash pattern={on 0.84pt off 2.51pt}]  (99.27,265.65) -- (101.58,277.86) ;

\draw  [dash pattern={on 0.84pt off 2.51pt}]  (113.37,265.65) -- (110.99,277.86) ;

\draw  [dash pattern={on 0.84pt off 2.51pt}]  (128.76,265.65) -- (131.08,277.86) ;

\draw  [dash pattern={on 0.84pt off 2.51pt}]  (152.45,267.33) -- (148.17,277.86) ;

\draw  [dash pattern={on 0.84pt off 2.51pt}]  (162.28,264.77) -- (166.55,277.86) ;

\draw  [dash pattern={on 0.84pt off 2.51pt}]  (179.02,264.53) -- (177.24,277.86) ;

\draw  [dash pattern={on 0.84pt off 2.51pt}]  (196.6,264.53) -- (197.76,277.86) ;

\draw  [dash pattern={on 0.84pt off 2.51pt}]  (218.15,264.53) -- (214.85,277.86) ;

\draw  [dash pattern={on 0.84pt off 2.51pt}]  (242.51,264.53) -- (240.5,277.86) ;

\draw  [dash pattern={on 0.84pt off 2.51pt}]  (255.55,264.53) -- (258.03,277.86) ;

\draw  [dash pattern={on 0.84pt off 2.51pt}]  (285.32,266.48) -- (280.19,277.86) ;

\draw  [dash pattern={on 0.84pt off 2.51pt}]  (292.1,267.97) -- (295.52,277.86) ;

\draw [line width=1.5]    (333.19,285.11) -- (333.19,32.08) ;
\draw [shift={(333.19,29.08)}, rotate = 90] [color={rgb, 255:red, 0; green, 0; blue, 0 }  ][line width=1.5]    (14.21,-4.28) .. controls (9.04,-1.82) and (4.3,-0.39) .. (0,0) .. controls (4.3,0.39) and (9.04,1.82) .. (14.21,4.28)   ;

\draw [line width=1.5]    (328.09,53.84) -- (338.29,53.84) ;

\draw [line width=1.5]    (328.09,82.18) -- (338.29,82.18) ;

\draw [line width=1.5]    (328.09,108.26) -- (338.29,108.26) ;

\draw [line width=1.5]    (328.09,148.5) -- (338.29,148.5) ;

\draw [line width=1.5]    (328.09,277.47) -- (338.29,277.47) ;

\draw  [fill={rgb, 255:red, 0; green, 0; blue, 0 }  ,fill opacity=1 ] (473.57,54.5) .. controls (473.57,52.94) and (474.84,51.67) .. (476.41,51.67) .. controls (477.97,51.67) and (479.24,52.94) .. (479.24,54.5) .. controls (479.24,56.07) and (477.97,57.34) .. (476.41,57.34) .. controls (474.84,57.34) and (473.57,56.07) .. (473.57,54.5) -- cycle ;

\draw  [fill={rgb, 255:red, 0; green, 0; blue, 0 }  ,fill opacity=1 ] (473.57,82.18) .. controls (473.57,80.62) and (474.84,79.35) .. (476.41,79.35) .. controls (477.97,79.35) and (479.24,80.62) .. (479.24,82.18) .. controls (479.24,83.75) and (477.97,85.02) .. (476.41,85.02) .. controls (474.84,85.02) and (473.57,83.75) .. (473.57,82.18) -- cycle ;

\draw    (476.41,54.5) -- (476.41,82.18) ;

\draw  [fill={rgb, 255:red, 0; green, 0; blue, 0 }  ,fill opacity=1 ] (401.17,108.65) .. controls (401.17,107.08) and (402.44,105.81) .. (404,105.81) .. controls (405.57,105.81) and (406.84,107.08) .. (406.84,108.65) .. controls (406.84,110.21) and (405.57,111.48) .. (404,111.48) .. controls (402.44,111.48) and (401.17,110.21) .. (401.17,108.65) -- cycle ;

\draw  [fill={rgb, 255:red, 0; green, 0; blue, 0 }  ,fill opacity=1 ] (539.32,108.92) .. controls (539.32,107.35) and (540.59,106.08) .. (542.16,106.08) .. controls (543.72,106.08) and (544.99,107.35) .. (544.99,108.92) .. controls (544.99,110.48) and (543.72,111.75) .. (542.16,111.75) .. controls (540.59,111.75) and (539.32,110.48) .. (539.32,108.92) -- cycle ;

\draw    (404,108.65) -- (476.41,82.18) ;

\draw    (476.41,82.18) -- (542.16,109.03) ;

\draw  [fill={rgb, 255:red, 0; green, 0; blue, 0 }  ,fill opacity=1 ] (571.44,148.72) .. controls (571.44,147.16) and (572.71,145.89) .. (574.28,145.89) .. controls (575.84,145.89) and (577.11,147.16) .. (577.11,148.72) .. controls (577.11,150.29) and (575.84,151.56) .. (574.28,151.56) .. controls (572.71,151.56) and (571.44,150.29) .. (571.44,148.72) -- cycle ;

\draw  [fill={rgb, 255:red, 0; green, 0; blue, 0 }  ,fill opacity=1 ] (361.41,148.39) .. controls (361.41,146.82) and (362.68,145.55) .. (364.24,145.55) .. controls (365.81,145.55) and (367.08,146.82) .. (367.08,148.39) .. controls (367.08,149.95) and (365.81,151.22) .. (364.24,151.22) .. controls (362.68,151.22) and (361.41,149.95) .. (361.41,148.39) -- cycle ;

\draw  [fill={rgb, 255:red, 0; green, 0; blue, 0 }  ,fill opacity=1 ] (490.77,148.2) .. controls (490.77,146.63) and (492.04,145.37) .. (493.6,145.37) .. controls (495.17,145.37) and (496.43,146.63) .. (496.43,148.2) .. controls (496.43,149.76) and (495.17,151.03) .. (493.6,151.03) .. controls (492.04,151.03) and (490.77,149.76) .. (490.77,148.2) -- cycle ;

\draw  [fill={rgb, 255:red, 0; green, 0; blue, 0 }  ,fill opacity=1 ] (422.32,194.99) .. controls (422.32,193.43) and (423.59,192.16) .. (425.15,192.16) .. controls (426.72,192.16) and (427.99,193.43) .. (427.99,194.99) .. controls (427.99,196.56) and (426.72,197.83) .. (425.15,197.83) .. controls (423.59,197.83) and (422.32,196.56) .. (422.32,194.99) -- cycle ;

\draw  [fill={rgb, 255:red, 0; green, 0; blue, 0 }  ,fill opacity=1 ] (552.85,195.39) .. controls (552.85,193.83) and (554.12,192.56) .. (555.68,192.56) .. controls (557.25,192.56) and (558.52,193.83) .. (558.52,195.39) .. controls (558.52,196.96) and (557.25,198.23) .. (555.68,198.23) .. controls (554.12,198.23) and (552.85,196.96) .. (552.85,195.39) -- cycle ;

\draw  [fill={rgb, 255:red, 0; green, 0; blue, 0 }  ,fill opacity=1 ] (384.43,195.05) .. controls (384.43,193.49) and (385.69,192.22) .. (387.26,192.22) .. controls (388.82,192.22) and (390.09,193.49) .. (390.09,195.05) .. controls (390.09,196.62) and (388.82,197.88) .. (387.26,197.88) .. controls (385.69,197.88) and (384.43,196.62) .. (384.43,195.05) -- cycle ;

\draw  [fill={rgb, 255:red, 0; green, 0; blue, 0 }  ,fill opacity=1 ] (519.36,213.39) .. controls (519.36,211.83) and (520.63,210.56) .. (522.19,210.56) .. controls (523.76,210.56) and (525.03,211.83) .. (525.03,213.39) .. controls (525.03,214.96) and (523.76,216.23) .. (522.19,216.23) .. controls (520.63,216.23) and (519.36,214.96) .. (519.36,213.39) -- cycle ;

\draw  [fill={rgb, 255:red, 0; green, 0; blue, 0 }  ,fill opacity=1 ] (456.89,233.04) .. controls (456.89,231.47) and (458.16,230.2) .. (459.72,230.2) .. controls (461.29,230.2) and (462.55,231.47) .. (462.55,233.04) .. controls (462.55,234.6) and (461.29,235.87) .. (459.72,235.87) .. controls (458.16,235.87) and (456.89,234.6) .. (456.89,233.04) -- cycle ;

\draw  [fill={rgb, 255:red, 0; green, 0; blue, 0 }  ,fill opacity=1 ] (588.82,244.4) .. controls (588.82,242.84) and (590.09,241.57) .. (591.65,241.57) .. controls (593.22,241.57) and (594.49,242.84) .. (594.49,244.4) .. controls (594.49,245.97) and (593.22,247.24) .. (591.65,247.24) .. controls (590.09,247.24) and (588.82,245.97) .. (588.82,244.4) -- cycle ;

\draw    (364.24,148.39) -- (404,108.65) ;

\draw    (387.26,195.05) -- (404,108.65) ;

\draw    (424.59,194.99) -- (404,108.65) ;

\draw    (459.72,233.04) -- (404,108.65) ;

\draw    (542.5,108.65) -- (493.6,148.2) ;

\draw    (522.19,213.39) -- (542.16,109.03) ;

\draw    (574.28,148.72) -- (542.16,109.03) ;

\draw    (574.28,148.72) -- (555.68,195.39) ;

\draw    (574.28,148.72) -- (591.65,244.4) ;

\draw  [dash pattern={on 4.5pt off 4.5pt}]  (333.19,82.18) -- (476.41,82.18) ;

\draw  [dash pattern={on 4.5pt off 4.5pt}]  (333.26,244.25) -- (591.65,244.4) ;

\draw    (349.93,261.78) -- (364.24,148.39) ;

\draw    (370.87,250.66) -- (364.24,148.39) ;

\draw    (381.56,257.08) -- (387.26,195.05) ;

\draw    (397.8,261.35) -- (387.26,195.05) ;

\draw    (417.03,259.64) -- (425.15,194.99) ;

\draw    (433.7,258.78) -- (424.59,194.99) ;

\draw    (454.22,257.93) -- (459.72,233.04) ;

\draw    (465.76,257.5) -- (459.72,233.04) ;

\draw    (483.29,256.65) -- (493.6,148.2) ;

\draw    (504.23,254.51) -- (493.6,148.2) ;

\draw    (516.2,253.66) -- (522.19,213.39) ;

\draw    (531.59,252.8) -- (522.19,213.39) ;

\draw    (546.12,254.08) -- (555.68,195.39) ;

\draw    (561.94,252.8) -- (555.68,198.23) ;

\draw    (588,251.46) -- (591.65,244.4) ;

\draw    (595.33,252.2) -- (591.65,244.4) ;

\draw  [dash pattern={on 0.84pt off 2.51pt}]  (350.1,265.23) -- (347.96,277.47) ;

\draw  [dash pattern={on 0.84pt off 2.51pt}]  (370.87,250.66) -- (372.15,277.59) ;

\draw  [dash pattern={on 0.84pt off 2.51pt}]  (381.56,257.08) -- (380.02,277.47) ;

\draw  [dash pattern={on 0.84pt off 2.51pt}]  (398.22,264.83) -- (400.54,277.04) ;

\draw  [dash pattern={on 0.84pt off 2.51pt}]  (417.03,259.64) -- (413.79,277.47) ;

\draw  [dash pattern={on 0.84pt off 2.51pt}]  (433.7,258.78) -- (436.01,277.47) ;

\draw  [dash pattern={on 0.84pt off 2.51pt}]  (454.22,257.93) -- (450.8,275.88) ;

\draw  [dash pattern={on 0.84pt off 2.51pt}]  (465.76,257.5) -- (468.75,275.88) ;

\draw  [dash pattern={on 0.84pt off 2.51pt}]  (483.29,256.65) -- (479.61,276.62) ;

\draw  [dash pattern={on 0.84pt off 2.51pt}]  (504.23,254.51) -- (506.37,273.75) ;

\draw  [dash pattern={on 0.84pt off 2.51pt}]  (516.2,253.66) -- (512.95,273.62) ;

\draw  [dash pattern={on 0.84pt off 2.51pt}]  (546.12,254.08) -- (542.88,275.33) ;

\draw  [dash pattern={on 0.84pt off 2.51pt}]  (561.94,252.8) -- (565.96,276.19) ;

\draw  [dash pattern={on 0.84pt off 2.51pt}]  (588,251.46) -- (580.86,275.76) ;

\draw  [dash pattern={on 0.84pt off 2.51pt}]  (595.33,252.2) -- (601.69,273.32) ;

\draw  [dash pattern={on 4.5pt off 4.5pt}]  (332.39,108.65) -- (542.5,108.65) ;

\draw  [dash pattern={on 4.5pt off 4.5pt}]  (334.11,148.5) -- (574.28,148.72) ;

\draw [line width=1.5]    (328.09,194.67) -- (338.29,194.67) ;

\draw  [dash pattern={on 0.84pt off 2.51pt}]  (531.59,252.8) -- (535.44,274.6) ;

\draw (35.88,14.38) node  [font=\normalsize] [align=left] {\begin{minipage}[lt]{43.96pt}\setlength\topsep{0pt}
Level $\displaystyle \varepsilon $
\end{minipage}};

\draw (5.14,78) node [anchor=north west][inner sep=0.75pt]    {$\varepsilon _{1}$};

\draw (5.14,102) node [anchor=north west][inner sep=0.75pt]    {$\varepsilon _{2}$};

\draw (5.14,145) node [anchor=north west][inner sep=0.75pt]    {$\varepsilon _{4}$};

\draw (8.92,268.75) node [anchor=north west][inner sep=0.75pt]    {$0$};

\draw (284.72,36.88) node [anchor=north west][inner sep=0.75pt]  [font=\footnotesize]  {${\displaystyle \varepsilon _{0} =\frac{1}{2}}$};

\draw (5.14,117) node [anchor=north west][inner sep=0.75pt]    {$\varepsilon _{3}$};

\draw (310.08,77) node [anchor=north west][inner sep=0.75pt]    {$\varepsilon _{1}$};

\draw (310.08,102) node [anchor=north west][inner sep=0.75pt]    {$\varepsilon _{2}$};

\draw (310.08,142) node [anchor=north west][inner sep=0.75pt]    {$\varepsilon _{3}$};

\draw (313.86,268.35) node [anchor=north west][inner sep=0.75pt]    {$0$};

\draw (8.81,38.63) node [anchor=north west][inner sep=0.75pt]  [font=\footnotesize]  {${\displaystyle \frac{1}{2}}$};

\draw (310.08,191) node [anchor=north west][inner sep=0.75pt]    {$\varepsilon _{4}$};

\draw (173.14,63.9) node [anchor=north west][inner sep=0.75pt]  [font=\scriptsize]  {$e( 0,0)$};

\draw (116.64,83.2) node [anchor=north west][inner sep=0.75pt]  [font=\scriptsize]  {$e( 1,0)$};

\draw (199.92,88.69) node [anchor=north west][inner sep=0.75pt]  [font=\scriptsize]  {$e( 1,1)$};

\draw (151,90) node [anchor=north west][inner sep=0.75pt]  [font=\footnotesize] [align=left] {{\footnotesize $\displaystyle v( 1,1)$}};

\draw (155.62,37.47) node [anchor=north west][inner sep=0.75pt]  [font=\scriptsize]  {$v( 0,0)$};

\draw (72,99.33) node [anchor=north west][inner sep=0.75pt]  [font=\footnotesize] [align=left] {{\footnotesize $\displaystyle v( 2,2)$}};

\draw (241.6,112.09) node [anchor=north west][inner sep=0.75pt]  [font=\footnotesize] [align=left] {{\footnotesize $\displaystyle v( 2,3)$}};

\draw (39,140.3) node [anchor=north west][inner sep=0.75pt]  [font=\footnotesize] [align=left] {{\footnotesize $\displaystyle v( 3,4)$}};

\draw (167,151.76) node [anchor=north west][inner sep=0.75pt]  [font=\footnotesize] [align=left] {{\footnotesize $\displaystyle v( 3,5)$}};

\draw (100,165.09) node [anchor=north west][inner sep=0.75pt]  [font=\footnotesize] [align=left] {{\footnotesize $\displaystyle v( 3,6)$}};

\draw (271.57,170.14) node [anchor=north west][inner sep=0.75pt]  [font=\footnotesize] [align=left] {{\footnotesize $\displaystyle v( 3,7)$}};

\draw (327,9.4) node [anchor=north west][inner sep=0.75pt]    {$\varepsilon $};

\end{tikzpicture}

        \caption{Two examples of skeleton trees constructed by Algorithm~\ref{alg:construct_tree}. The components of the distribution $\mu$ for Proposition~\ref{prop:example_covering_nb} are associated to leaves (or rather paths) of this infinite tree so that the distance $\xi$ between leaves coincides with the natural tree metric (up to a constant factor $2$). On the left we represent the simpler case when $N_{k+1}-N_k=1$ for $k\geq 1$, that is, covering numbers $N(\eps)$ grow by one at a time as $\eps\to 0$. In this case, the constructed tree is exactly a binary tree, constructed according to the exact ordering given by $Order(l)$ for $l\geq 1$. On the right, we represent a general case when covering numbers can grow via jumps $N_{k+1}-N_k\geq 1$. In the specific example, we have $(N_i,i\leq 7)=(1,2,7,10,13,14,15)$. Although the tree is not formally a complete binary tree, the ordering choice balances all subtrees evenly. We represent with dashed lines, all levels $\eps$ which complete a layer of the constructed binary tree.}
        \label{fig:skeleton_tree}
    \end{figure}
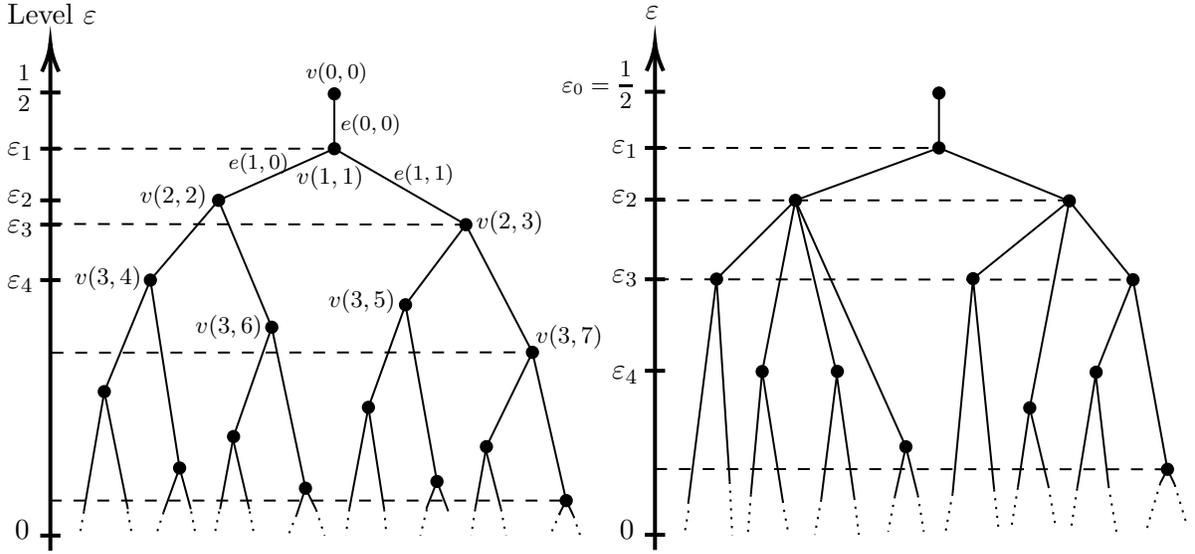

    Note that at every level $\eps\in (0,\frac{1}{2}]\setminus\{\eps_k,k\geq 1\}$, the tree has $N(\eps)$ edges. Because of the hypothesis $\int_0^{1/2}N(\eps)\mathd\eps=\infty$, we have $N(\eps)\to \infty$ as $n\to\infty$ or equivalently $N_k\to\infty$ as $k\to\infty$
    In particular, the tree output by Algorithm~\ref{alg:construct_tree} has an infinite number of edges. We denote by $\Lcal$ the set of leaves of the tree, which corresponds to sequences of nodes $(v_i=v(k_i,n_i))_{i\geq 1}$ that start from the root $v(0,1)$, follow edges of the tree, and go down in the tree, that is the sequence $(k_i)_{i\geq 1}$ is increasing. We note that because of the breadth-first search procedure to construct the tree, all leaves $l=(v_i)_{i\geq 1}$ contain an infinite number of edges---that is, no path ended after a finite number of edges. The tree naturally induces a distance $\rho$ on leaves $\Lcal$ such that two leaves $l_u = (v(k^{(u)}_i,n^{(u)}_i))_{i\geq 1}$ for $u\in\{1,2\}$ have distance
    \begin{equation*}
        d(l_1,l_2) := \eps_{k_i},\quad i = \max\{j:n^{(1)}_j = n^{(2)}_j\}.
    \end{equation*}
    Our goal is to use the tree to construct a binary stochastic process $(X_l)_{l\in\Lcal}$ on $\Lcal$, for which the induced metric $\xi(l_1,l_2) = \Pbb(X_{l_1}\neq X_{l_2})$ coincides with $d$. We start by constructing a distribution on the inner nodes of the tree recursively. First, let $(U_n)_{n\geq 1}\overset{\iid}{\sim}\Ucal([0,1])$ be a sequence of iid uniform random variables on $[0,1]$ and $(Z_n)_{n\geq 1}\overset{i.i.d}{\sim}\Bernu(\frac{1}{2})$ be an independent iid sequence of Bernoulli random variables. At the root $r=v(0,1)$, we let $Y_r=\1[Z_1\geq \frac{1}{2}]$. Next, for any node $v=v(k,n)$ that has parent $v'=v(k',n')$ (they are linked by an edge and $k'<k$), we pose 
    \begin{equation*}
        Y_v := \1[U_n \geq \eta_v] Y_{v'} + \1[U_n<\eta_v] Z_n,\quad \eta_v:= \frac{\sqrt{1-2\eps_k} - \sqrt{1-2\eps_{k'}}}{\sqrt{1-2\eps_k}}.
    \end{equation*}
    We now define the binary stochastic process on $\Lcal$ as follows. For any leaf $l=(v_i)_{i\geq 1}$,
    \begin{equation*}
        X_l:= \begin{cases}
            \lim_{i\to\infty} Y_{v_i} & \liminf_{i\to\infty} Y_{v_i} = \limsup_{i\to\infty} Y_{v_i},\\
            0 &\text{otherwise}.
        \end{cases}
    \end{equation*}
    We can now check that on $\Lcal$ introduced by the distribution of $(X_l)_{l\in\Lcal}$ coincides with $d$. For a leaf $l=(v_i=v(k_i,n_i))_{i\geq 1}$, we first show that $\Pbb(X_l\neq Y_{v_i}) = \frac{1}{2}(1-\sqrt{1-2\eps_{k_i}})$. Indeed, recalling that the sequence $(Z_n)_{n\geq 1}$ is iid $\Bernu(\frac{1}{2})$, for any $j>i$, we can write
    \begin{equation}\label{eq:writing_leaf_nodes}
        Y_{v_j} = \1[\forall i+1\leq p\leq j: U_{n_p}\geq \eta_{v_p}] Y_{v_i}  + \1[\exists i+1\leq p\leq j: U_{n_p}<\eta_{v_p}] A_j,
    \end{equation}
    where $A_j\sim\Bernu(\frac{1}{2})$ is a function of the variables $Z_{n_p}$ for $i+1\leq p\leq j$. 
    Next, observe that since for $i\to\infty$, one has $\eta_{v_i}\sim \eps_{k_{i-1}} - \eps_{k_i}$, then $\sum_{i\geq 1}\eta_{n_i}<\infty$. In particular, the Borel-Cantelli lemma implies that on an almost sure event $\Ecal_l$, for sufficiently large index $i$, one has $U_{n_i}\geq \eta_{v_i}$. Hence, on $\Ecal_l$ the sequence $(Y_{v_i})_{i\geq 1}$ is either finite or converges, that is
    \begin{equation*}
        \Ecal_l \subset \{ X_l = \lim_{i\to\infty} Y_{v_i}\}.
    \end{equation*}
    Hence, using Eq~\eqref{eq:writing_leaf_nodes}, we can write with $\Fcal_l(i):=\{\exists j\geq i+1, U_{n_j}\geq \eta_{v_j}\}$,
    \begin{equation*}
        X_l = \1_{\Ecal_l} \paren{\1_{\Fcal_l(i)^c} Y_{v_i} + \1_{\Fcal_l(i)}B_l(i)},
    \end{equation*}
    where $B_l\sim\Bernu(\frac{1}{2})$ is a function of the variables $Z_{n_j}$ for $j\geq i+1$. Also, note that
    \begin{equation*}
        \Pbb[\Fcal_l(i)] = \Pbb[\exists j\geq i+1, U_{n_j}\geq \eta_{v_j}] = 1- \prod_{j\geq i+1} (1-\eta_{v_p}) = 1-\sqrt{1-2\eps_{k_i}},
    \end{equation*}
    where in the last equality, we used a telescoping argument. To summarize, we showed that $X_l$ coincides with $Y_{v_i}$ except on an (independent) event of probability $1-\sqrt{1-2\eps_{n_i}}$ on which it is an independent Bernoulli $B_l(i)$. We are now ready to compute the $\Pbb(X_{l_1}\neq X_{l_2})$ for two leaves $l_1,l_2\in\Lcal$. Let $v=v(k_i,n_i)$ be the first node for which the paths $l_1=(v_i^{(1)} )_{i\geq 1}$ and $l_2=(v_i^{(2)})_{i\geq 1}$ differ, that is
    \begin{equation*}
        i = \max\{j\geq 1: v_j^{(1)} = v_j^{(2)}\}.
    \end{equation*}
    Then, using the previous characterization of $X_{l_1}$ and $X_{l_2}$, we obtain
    \begin{align*}
        \xi(l_1,l_2)=\Pbb(X_{l_1}\neq X_{l_2}) &= \frac{1}{2}\Pbb\paren{ \Fcal_{l_1}(i)\cup \Fcal_{l_2}(i)} \\
        &= \frac{1}{2}\paren{2(1-\sqrt{1-2\eps_{k_i}}) - (1-\sqrt{1-2\eps_{k_i}})^2}\\
        &=\eps_{k_i} = d(l_1,l_2).
    \end{align*}
    This ends the proof that $\xi$ can directly be computed as the tree distance (up to a factor $2$). As defined currently, the space of leaves $\Lcal$ can potentially be uncountable. Because we need to construct a distribution on $\{0,1\}^\N$, we restrict ourselves to a countable subset of leaves $\Qcal$, one at most for each inner node. Precisely, to any node $v=v(k,n)$ we associate the leaf $l(v)$ which arrives at $v(k,n)$ and from there always selects the left-most edge (first added in the FIFO pile) at any intersection. We then pose $\Qcal = \{l(v), \text{ nodes }v\}$. We recall that there are countably-many nodes, hence $\Qcal$ is countable. The distribution $\mu$ is defined as the distribution of $(X_l)_{l\in\Qcal}$. We note that because $\Qcal$ is now countable, the event $\Ecal:=\bigcap_{l\in\Qcal} \Ecal_l$ has full probability. Hence with probability one,
    \begin{equation*}
        \forall l = (v_i)_{i\geq 1}\in\Qcal, X_l = \lim_{i\to\infty}Y_{v_i}.
    \end{equation*}

    \paragraph{Computing the covering numbers of $\xi$.} Let $E=\{\eps_k,k\geq 1\}$ and fix $\eps\in(0,\frac{1}{2}]\setminus E$. By construction of the skeleton tree, at level $\eps$ there are exactly $N(\eps)$ edges $f_1,\ldots,f_{N(\eps)})$. For each of these edges say $f_p$ for $p\in[N(\eps)]$ if its end nodes are $u_p,v_p$ with $u_p$ being the parent of $v_p$ (that is, has lower index $n$ and number $k$ as well), we now show that $\{l(v_p),p\in[N(\eps)\}$ is an $\eps$-covering of $(\Qcal,\xi)$. Indeed, for $l = (w_i)_{i\geq 1}\in\Qcal$, one of the nodes on the corresponding leaf path must belong to $\{v_p,p\in[N(\eps)\}$. Hence, for some $i\geq 1$ and $p\in[N(\eps)]$, we have $w_i = v_p:= v(k_p,n_p)$. Then, we directly have
    \begin{equation*}
        \xi(l, l(v_p)) = d(l,l(v_p)) \leq \eps_{k_p} < \eps.
    \end{equation*}
    Hence, we obtained $\Ncal_\xi(\eps) \leq N(\eps)$. 

    We now turn to the lower bound. We will show that to $\eps$-cover the set of leaves $\{l(v_p), p\in[N(\eps)]\}$ one needs $N(\eps)$ elements. Suppose that this is not the case, then we have a leaf $l(v)$ such that $\xi(l(v),l(v_p)),\xi(l(v),l(v_q))\leq \eps$. In particular, $l(v)$ and $l(v_p)$ share the same path until length $\eps$, hence the path of $l(v)$ contains edge $f_p$. By symmetry, this shows it also contains $f_{p'}$ which is impossible because they are at same level (and paths only go ``down''). This ends the proof that $\Ncal_\xi(\eps) \geq N(\eps)$.

    In summary, we have that for any $\eps\in(0,\frac{1}{2}]\setminus E$, $\Ncal_\xi(\eps) = N(\eps)$. Additionally, the same arguments show that for any $k\geq 1$, one has $\Ncal_\xi(\eps_k) = N_{k-1}$: we again look at level $\eps_k$ of the tree. If this cuts edges, we proceed similarly as above. However, it will also be the case that at this level are nodes $v=v(k,n)$. These are then also included to construct a set of $N_{k-1}$ nodes $v_1,\ldots,v_{N_{k-1}}$ at level $\eps_k$ or below. The same proof shows that they $\eps_k$-cover the space $\Qcal$, and that $\{l(v_p),p\in[N_{k-1}]\}$ requires at least $N_{k-1}$ elements to be $\eps_k$-covered. This shows in particular that $\Ncal_\xi(\cdot)$ is right-continuous.

    \paragraph{Estimating $\Delta_n(\mu)$.} In this last step, we show that $\Delta_n(\mu)\to\frac{1}{2}$. First, note that by construction and from the above estimates, for any $l\in\Qcal$, we have $X_l\sim\Bernu(\frac{1}{2})$ so that $\p[l]=\Ebb[X_l]=\frac{1}{2}$.

    We recall that the construction of the skeleton tree emulates binary tree that is constructed layer by layer $r$. Consider the state of the tree at the very beginning of the construction of the $r$th layer for some fixed $r\geq 1$. At this point, there are $2^r$ edges $e(r,s)$ for $s\in\{0,1,\ldots,2^r-1\}$ and we can consider the corresponding subtrees at each of these edges $e(r,s)$, which correspond to the set of nodes and edges descendants from $e(r,s)$ (in the case when an edge $e$ was removed and replaced by two new edges, these new edges are also considered descendants of $e$), which we will denote $\Tcal(r,s)$. The main interest of the fractal order for the construction of the tree is that starting from the $r$th layer and for all next layers, we are adding a single edge to $\Tcal(r,s)$ in the order of $s\in \Order)$ then the process is repeated indefinitely. As a result, the subtrees $\Tcal(r,s)$ for $s\in\{0,\ldots,2^r-1\}$ are always filled equally, up to at most one edge. The property that the trees are filled evenly is crucial for the proof.

    In the rest of the proof, we denote by $\parent(v)$ the parent node of any node $v$. For any fixed $s\in\{0,\ldots,2^r-1\}$, we define
    \begin{equation*}
        A(r,s):=\sum_{v = v(k,n)\in\Tcal(r,s)} \eps_{k(\parent(v))}  - \eps_k,
    \end{equation*}
    where the parent node of $v$ is written as $\parent(v)=v(k(\parent(v)), n')$. Our goal is to show that the above quantity is infinite. Let $k_r$ be the value of the level $k$ at the beginning of the construction of the $r$th layer in Algorithm~\ref{alg:construct_tree}. Without loss of generality, we will assume that $\{\eps_k,k\geq 1\}\cap \{2^{-t},t\geq 1\}=\emptyset$. If that is not the case, we can replace all terms $2^{-t}$ with some terms $c2^{-t}$ for some constant $c\in[\frac{1}{2},1]$ since there will exist such a constant $c$ for which  $\{\eps_k,k\geq 1\}\cap \{c2^{-t},t\geq 1\}=\emptyset$. For any $t\geq t_r:=\ceil{\log_2 1/\eps_{k_r} }$, we let $M(r,s;t)$ be the number of edges within $\Tcal(r,s)$ at level $2^{-t}$. Then, we can check that
    \begin{equation}\label{eq:ineq_A}
        A(r,s) \geq \sum_{t\geq t_r} M(r,s;t) (2^{-t} - 2^{-t-1})= \frac{1}{2}\sum_{t\geq t_r} 2^{-t} M(r,s;t).
    \end{equation}
    Now we recall that the number of edges in the complete tree $\Tcal$ at level $2^{-t}$ is precisely $\Ncal_\xi(2^{-t}) = N(2^{-t})$, where we used the result from the previous steps on covering numbers of $\xi$. As a result, we have that for $t\geq t_r$,
    \begin{equation*}
        \sum_{s'=0}^{2^r-1} M(r,s';t) = N(2^{-t}).
    \end{equation*}
    Now from the above discussion on the evenness of the tree construction, the number of edges at any level $\eps\leq \eps_{k_r}$ for the subtrees $\Tcal(r,s)$ can differ at most by one. Hence, we obtain,
    \begin{equation*}
        2^r (M(r,s;t) + 1) \geq \sum_{s'=0}^{2^r-1} M(r,s';t) = N(2^{-t}).
    \end{equation*}
    Plugging this into Eq~\eqref{eq:ineq_A} yields
    \begin{equation*}
        A(r,s) \geq \frac{1}{2^{r+1}}\sum_{t\geq t_r} 2^{-t} N(2^{-t}) - 2^{-t_r} = \infty.
    \end{equation*}
    The last inequality use the hypothesis $\int_0^{1/2}N(\eps)\mathd\eps = \infty$ and the fact that $N$ is non-increasing, so that this condition is equivalent to $\sum_{t\geq 1} 2^{-t} N(2^{-t})=\infty$. As a result, this shows that $A(r,s)=\infty$. Now let $v_{r,s}$ be the top end node of edge $e(r,s)$, which is intuitively the ``root'' of $\Tcal(r,s)$. We obtained
    \begin{align*}
        \Pbb(\exists v \in \Tcal(r,s), Y_v\neq Y_{v_{r,s}}) &\geq \Pbb ( \exists  v \in \Tcal(r,s) , Y_v\neq Y_{\parent(v)})\\
        &= 1- \prod_{v \in \Tcal(r,s)}(1-  \Pbb(Y_v\neq Y_{\parent(v)}) )\\
        &=  1- \prod_{v \in \Tcal(r,s)} \paren{1-  \frac{\eta_v}{2} } \geq 1- \exp\paren{-\frac{1}{2}\sum_{v \in \Tcal(r,s)} \eta_v }.
    \end{align*}
    Now as the index $n$ of $v=v(k,n)$ grows to infinity, we have $\eta_v\sim \eps_{k(\parent(v))}-\eps_k$ because $\eps_{k(\parent(v))}, \eps_k\to 0$. The fact that $A(r,s)=\infty$ then shows that $\sum_{v \in \Tcal(r,s)} \eta_v=\infty$. Hence, we showed that for any $r\geq 1$ and $s\in\{0,\ldots,2^r-1\}$,
    \begin{equation*}
        \Pbb(\exists v \in \Tcal(r,s), Y_v\neq Y_{v_{r,s}}) = 1.
    \end{equation*}
    We denote by $\Gcal(r,s)$ the above event. Hence $\Gcal = \bigcap_{r\geq 1}\bigcap_{0\leq s\leq 2^r-1} \Gcal(r,s)$ has probability one. The main property of $\Gcal$ is that on this almost sure event, for any node $v$ of the tree, there exists a descendant node $v'$ that disagrees in the sense $Y_v\neq Y_{v'}$.

    We are now ready to show that $\Delta_n(\mu)$ does not decay to 0 as $n\to\infty$. Fix $n\geq 1$ and $\delta>0$. We will indicate that we consider the 
    $i$th iid sample of a certain random variable (or event) $V$ with an exponent as in $V^{(i)}$. We construct a sequence of nodes $(\hat v_i)_{i\geq 0}$ recursively. We let $\hat v_0 := v_\delta$ where $v_\delta=v(k_\delta,n_\delta)$ is an arbitrary node for which $\eps_{k_\delta}$ is sufficiently small such that
    \begin{equation}\label{eq:needed_bound_eps}
        \frac{n}{2}\paren{1-\sqrt{1-2\eps_{k_\delta}}} <\delta.
    \end{equation}
    Next, for $i=1$, we define
    \begin{equation*}
        \hat v_i = \begin{cases}
            v_\delta &\text{if } \Gcal^{(i)} \text{ is not satisfied},\\
            \hat v &\text{otherwise, and } \hat v = \argmin\{n: v=v(k,n) \text{ is a descendant of }\hat v_{i-1} \text{ s.t. } Y_v = 1\}.
        \end{cases}
    \end{equation*}
    On the almost sure event $\bigcap_{i\geq 1}\Gcal^{(i)}$, this constructs a sequence of nodes descendants from each other and such that we have 
    \begin{equation*}
        Y_{\hat v_i}^{(i)} = 1,\quad i\geq 1.
    \end{equation*}
    Our candidates for variables whose empirical mean deviates highly from the mean $\frac{1}{2}$ will be the leaves $l(\hat v_i)$ for $i\geq 1$. Importantly, having constructed $\hat v_{i-1}$, the construction of $\hat v_i = v(\hat k_i,\hat n_i)$ can be done completely independently from all the variables $U_n^{(i)},Z_n^{(i)}$ where $n>\hat n_i$. This is because we can simply generate the variables $Y_v^{(i)}$ for nodes $v$ with index $n\in \{\hat n_{i-1}, \hat n_{i-1} +1,\ldots\}$ and stop whenever the conditions $Y_v^{(i)}=1$ and $v$ is a descendant of $\hat v_{i-1}$ are met.

    We now reason conditionally on $\bigcap_{i\geq 1}\Gcal^{(i)}$ and $(\hat v_j)_{j\leq i}$. Note that up to this conditioning, the variables $(\hat v_j)_{j>i}$ only depend on the iid samples of the distribution with index $j>i$. In particular, for any $j\geq i$, this shows that all variables use to define $X_{l(\hat v_j)}$ starting from $Y_{\hat v_i}$ are still all distributed according to their distribution without conditioning. Precisely, write $\hat v_i = v(\hat k_i,\hat n_i)$. Conditioned on $\bigcap_{i\geq 1}\Gcal^{(i)}$ and $(\hat v_j)_{j\geq 1}$, all variables $U^{(i)}_n$ and $Z^{(i)}_n$ for $n> \hat n_i$ and $i\geq 1$ are still all independent and distributed as $\Ucal([0,1])$ and $\Bernu(\frac{1}{2})$ respectively. In particular, for a fixed $n\geq 1$, conditionally on $\bigcap_{i\geq 1}\Gcal^{(i)}$ and $(\hat v_j)_{j\geq 1}$, we have
    \begin{equation*}
        \paren{\1\sqb{X^{(i)}_{l(\hat v_n)} \neq 1 } }_{i\in[n]} = \paren{\1\sqb{X^{(i)}_{l(\hat v_n)} \neq Y_{\hat v_i}^{(i)} } }_{i\in[n]}   \sim \bigotimes_{i\in[n]} \Bernu\paren{\frac{1}{2}\paren{1-\sqrt{1-2\eps_{\hat k_i}}}}.
    \end{equation*}
    As a result,
    \begin{align*}
        \Pbb\sqb{\exists i\in[n], X_{l(\hat v_n)}^{(i)}\neq 1 \mid \bigcap_{i\geq 1}\Gcal^{(i)}, (\hat v_j)_{j\geq 1}} & \leq \sum_{i\in[n]} \frac{1}{2}\paren{1-\sqrt{1-2\eps_{\hat k_i}}}\\
        &\leq \frac{n}{2}\paren{1-\sqrt{1-2\eps_{k_\delta}}} <\delta.
    \end{align*}
    In the last inequality, we used Eq~\eqref{eq:needed_bound_eps}. Hence,
    \begin{align*}
        \Delta_n(\mu) \geq \Ebb\sqb{  \abs{ \pn[l(\hat v_n)] - \frac{1}{2}} \mid \bigcap_{i\geq 1}\Gcal^{(i)}  } 
        \geq \frac{1}{2}  \Pbb\sqb{\forall i\in[n], X_{l(\hat v_n)}^{(i)}= 1 \mid \bigcap_{i\geq 1}\Gcal^{(i)}} \geq \frac{1-\delta}{2}.
    \end{align*}
    This holds for any $\delta>0$. Thus, we showed that $\Delta_n(\mu) \geq \frac{1}{2}$. Also, we clearly have $\|\pn - \frac{1}{2}\|_\infty \leq \frac{1}{2}$ since the empirical means lie in $[0,1]$. This shows that
    \begin{equation*}
        \Delta_n(\mu) = \frac{1}{2},\quad n\geq 1,
    \end{equation*}
    and ends the proof of the result.
\end{proof}

The proof of the previous result introduces a tree structure for the components of $\mu$. In order to have as much deviations as possible (so that $\Delta_n(\mu)=\frac{1}{2}$ for $n\geq 1$, this tree was constructed using a ``wide'' full binary tree (see Figure~\ref{fig:skeleton_tree}). As a remark, we can compare this to the tree generated by the distribution from Proposition~\ref{prop:ex_fast_converging}. In that result, the goal is instead to construct distributions with large covering numbers, but that still have $\Delta_n(\mu)\ninf 0$. As a result, the corresponding tree is as ``thin'' as possible, as represented in Figure~\ref{fig:thin_tree}.

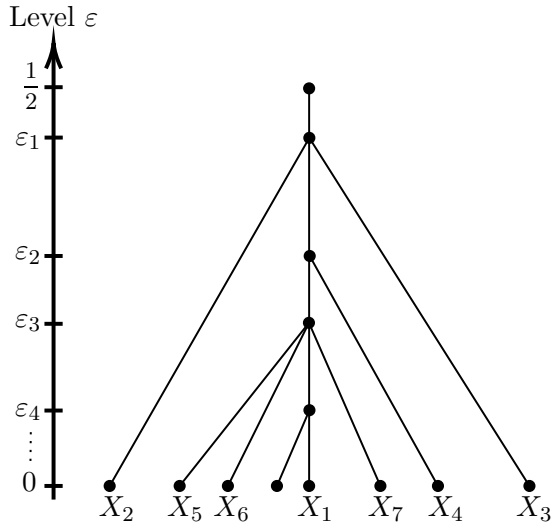
\begin{figure}[ht]
    \centering

\tikzset{every picture/.style={line width=0.75pt}}    

\begin{tikzpicture}[x=0.75pt,y=0.75pt,yscale=-0.9,xscale=0.9]

\draw [line width=1.5]    (28.25,285.5) -- (28.25,32.47) ;
\draw [shift={(28.25,29.47)}, rotate = 90] [color={rgb, 255:red, 0; green, 0; blue, 0 }  ][line width=1.5]    (14.21,-4.28) .. controls (9.04,-1.82) and (4.3,-0.39) .. (0,0) .. controls (4.3,0.39) and (9.04,1.82) .. (14.21,4.28)   ;

\draw [line width=1.5]    (23.15,54.23) -- (33.35,54.23) ;

\draw [line width=1.5]    (23.15,235.57) -- (33.35,235.57) ;

\draw [line width=1.5]    (23.15,148.89) -- (33.35,148.89) ;

\draw [line width=1.5]    (23.15,277.86) -- (33.35,277.86) ;

\draw  [fill={rgb, 255:red, 0; green, 0; blue, 0 }  ,fill opacity=1 ] (168.63,54.89) .. controls (168.63,53.33) and (169.9,52.06) .. (171.47,52.06) .. controls (173.03,52.06) and (174.3,53.33) .. (174.3,54.89) .. controls (174.3,56.46) and (173.03,57.73) .. (171.47,57.73) .. controls (169.9,57.73) and (168.63,56.46) .. (168.63,54.89) -- cycle ;

\draw  [fill={rgb, 255:red, 0; green, 0; blue, 0 }  ,fill opacity=1 ] (168.63,277.86) .. controls (168.63,276.3) and (169.9,275.03) .. (171.47,275.03) .. controls (173.03,275.03) and (174.3,276.3) .. (174.3,277.86) .. controls (174.3,279.43) and (173.03,280.7) .. (171.47,280.7) .. controls (169.9,280.7) and (168.63,279.43) .. (168.63,277.86) -- cycle ;

\draw    (171.47,54.89) -- (171.47,277.86) ;

\draw  [fill={rgb, 255:red, 0; green, 0; blue, 0 }  ,fill opacity=1 ] (168.75,82.65) .. controls (168.75,81.08) and (170.02,79.81) .. (171.58,79.81) .. controls (173.15,79.81) and (174.42,81.08) .. (174.42,82.65) .. controls (174.42,84.21) and (173.15,85.48) .. (171.58,85.48) .. controls (170.02,85.48) and (168.75,84.21) .. (168.75,82.65) -- cycle ;

\draw  [fill={rgb, 255:red, 0; green, 0; blue, 0 }  ,fill opacity=1 ] (168.95,148.89) .. controls (168.95,147.33) and (170.22,146.06) .. (171.79,146.06) .. controls (173.35,146.06) and (174.62,147.33) .. (174.62,148.89) .. controls (174.62,150.46) and (173.35,151.73) .. (171.79,151.73) .. controls (170.22,151.73) and (168.95,150.46) .. (168.95,148.89) -- cycle ;

\draw    (59.7,277.86) -- (171.58,82.65) ;

\draw    (171.58,82.65) -- (294.92,277.86) ;

\draw  [fill={rgb, 255:red, 0; green, 0; blue, 0 }  ,fill opacity=1 ] (56.86,277.86) .. controls (56.86,276.3) and (58.13,275.03) .. (59.7,275.03) .. controls (61.26,275.03) and (62.53,276.3) .. (62.53,277.86) .. controls (62.53,279.43) and (61.26,280.7) .. (59.7,280.7) .. controls (58.13,280.7) and (56.86,279.43) .. (56.86,277.86) -- cycle ;

\draw [line width=1.5]    (23.15,186.93) -- (33.35,186.93) ;

\draw  [fill={rgb, 255:red, 0; green, 0; blue, 0 }  ,fill opacity=1 ] (96.08,277.86) .. controls (96.08,276.3) and (97.35,275.03) .. (98.91,275.03) .. controls (100.48,275.03) and (101.75,276.3) .. (101.75,277.86) .. controls (101.75,279.43) and (100.48,280.7) .. (98.91,280.7) .. controls (97.35,280.7) and (96.08,279.43) .. (96.08,277.86) -- cycle ;

\draw  [fill={rgb, 255:red, 0; green, 0; blue, 0 }  ,fill opacity=1 ] (168.53,186.36) .. controls (168.53,184.8) and (169.8,183.53) .. (171.37,183.53) .. controls (172.93,183.53) and (174.2,184.8) .. (174.2,186.36) .. controls (174.2,187.93) and (172.93,189.2) .. (171.37,189.2) .. controls (169.8,189.2) and (168.53,187.93) .. (168.53,186.36) -- cycle ;

\draw  [fill={rgb, 255:red, 0; green, 0; blue, 0 }  ,fill opacity=1 ] (240.83,277.86) .. controls (240.83,276.3) and (242.1,275.03) .. (243.67,275.03) .. controls (245.23,275.03) and (246.5,276.3) .. (246.5,277.86) .. controls (246.5,279.43) and (245.23,280.7) .. (243.67,280.7) .. controls (242.1,280.7) and (240.83,279.43) .. (240.83,277.86) -- cycle ;

\draw  [fill={rgb, 255:red, 0; green, 0; blue, 0 }  ,fill opacity=1 ] (123.12,277.86) .. controls (123.12,276.3) and (124.39,275.03) .. (125.96,275.03) .. controls (127.52,275.03) and (128.79,276.3) .. (128.79,277.86) .. controls (128.79,279.43) and (127.52,280.7) .. (125.96,280.7) .. controls (124.39,280.7) and (123.12,279.43) .. (123.12,277.86) -- cycle ;

\draw  [fill={rgb, 255:red, 0; green, 0; blue, 0 }  ,fill opacity=1 ] (168.81,235.37) .. controls (168.81,233.81) and (170.08,232.54) .. (171.64,232.54) .. controls (173.21,232.54) and (174.48,233.81) .. (174.48,235.37) .. controls (174.48,236.94) and (173.21,238.21) .. (171.64,238.21) .. controls (170.08,238.21) and (168.81,236.94) .. (168.81,235.37) -- cycle ;

\draw  [fill={rgb, 255:red, 0; green, 0; blue, 0 }  ,fill opacity=1 ] (150.6,277.86) .. controls (150.6,276.3) and (151.87,275.03) .. (153.43,275.03) .. controls (155,275.03) and (156.27,276.3) .. (156.27,277.86) .. controls (156.27,279.43) and (155,280.7) .. (153.43,280.7) .. controls (151.87,280.7) and (150.6,279.43) .. (150.6,277.86) -- cycle ;

\draw  [fill={rgb, 255:red, 0; green, 0; blue, 0 }  ,fill opacity=1 ] (208.58,277.86) .. controls (208.58,276.3) and (209.85,275.03) .. (211.42,275.03) .. controls (212.98,275.03) and (214.25,276.3) .. (214.25,277.86) .. controls (214.25,279.43) and (212.98,280.7) .. (211.42,280.7) .. controls (209.85,280.7) and (208.58,279.43) .. (208.58,277.86) -- cycle ;

\draw  [fill={rgb, 255:red, 0; green, 0; blue, 0 }  ,fill opacity=1 ] (292.09,277.86) .. controls (292.09,276.3) and (293.36,275.03) .. (294.92,275.03) .. controls (296.49,275.03) and (297.76,276.3) .. (297.76,277.86) .. controls (297.76,279.43) and (296.49,280.7) .. (294.92,280.7) .. controls (293.36,280.7) and (292.09,279.43) .. (292.09,277.86) -- cycle ;

\draw    (98.91,277.86) -- (171.37,186.36) ;

\draw    (243.67,277.86) -- (171.79,148.89) ;

\draw    (125.96,277.86) -- (171.37,186.36) ;

\draw    (211.42,277.86) -- (171.37,186.36) ;

\draw    (153.43,277.86) -- (171.64,235.37) ;

\draw  [dash pattern={on 0.84pt off 2.51pt}]  (15.5,247.7) -- (15.3,264.85) ;

\draw [line width=1.5]    (23.21,82.49) -- (33.41,82.49) ;

\draw (35.88,14.38) node  [font=\normalsize] [align=left] {\begin{minipage}[lt]{43.96pt}\setlength\topsep{0pt}
Level $\displaystyle \varepsilon $
\end{minipage}};

\draw (5.14,230) node [anchor=north west][inner sep=0.75pt]    {$\varepsilon _{4}$};

\draw (5.14,142) node [anchor=north west][inner sep=0.75pt]    {$\varepsilon _{2}$};

\draw (8.92,268.75) node [anchor=north west][inner sep=0.75pt]    {$0$};

\draw (5.14,180) node [anchor=north west][inner sep=0.75pt]    {$\varepsilon _{3}$};

\draw (8.81,38.63) node [anchor=north west][inner sep=0.75pt]  [font=\footnotesize]  {${\displaystyle \frac{1}{2}}$};

\draw (163.07,282) node [anchor=north west][inner sep=0.75pt]    {$X_{1}$};

\draw (51.87,282) node [anchor=north west][inner sep=0.75pt]    {$X_{2}$};

\draw (285.87,282) node [anchor=north west][inner sep=0.75pt]    {$X_{3}$};

\draw (236.27,282) node [anchor=north west][inner sep=0.75pt]    {$X_{4}$};

\draw (90.27,282) node [anchor=north west][inner sep=0.75pt]    {$X_{5}$};

\draw (116.27,282) node [anchor=north west][inner sep=0.75pt]    {$X_{6}$};

\draw (202.67,282) node [anchor=north west][inner sep=0.75pt]    {$X_{7}$};

\draw (5.2,78) node [anchor=north west][inner sep=0.75pt]    {$\varepsilon _{1}$};

\end{tikzpicture}

    \caption{Tree corresponding to the distribution $\mu$ constructed in Proposition~\ref{prop:ex_fast_converging}. As in Figure~\ref{fig:skeleton_tree}, the leaves of the tree represent the components of the distribution, and the distance $\xi$ coincides exactly with the natural tree metric (up to a constant factor 2). In this example, the covering numbers are $(N_k,k\leq 5)=(1,3,4,7,8)$.}
    \label{fig:thin_tree}
\end{figure}

\section{Upper bound on \mathinhead{$\Delta_n(\mu)$} via subgaussian differences}
\label{sec:subgaussian}

For $i,j\in\N$, let us define $\r[ij]=\E[X_iX_j]$.
We claim the following bound on the moment generating function
of the difference of correlated Bernoullis.
\begin{lemma}
\label{lem:pqr-mgf}
\beq
\E\exp
\set{
t[
(X_i-\p[i])
-
(X_j-\p[j])
}
&\le&
\exp\paren{
\frac{t^2}{\log {\frac{2}{\p[i]+\p[j]-2\r[ij]}}}
},
\qquad t\ge 0.
\eeq
\end{lemma}

\begin{proof}
Consider the following functions.

\beq
f (x)
& \eqdef &
\log \left(\left(\p[1]-\p[12]\right) \mathe^{\left(-\p[1]+\p[2]+1\right) x}+\left(\p[2]-\p[12]\right) \mathe^{\left(-\p[1]+\p[2]-1\right) x}+\left(-\p[1]-\p[2]+2 \p[12]+1\right) \mathe^{\left(\p[2]-\p[1]\right) x}\right),
\\
g (x)
& \eqdef &
\frac{x^2}{\log \left(\frac{2}{\p[1]+\p[2]-2 \p[12]}\right)}.
\eeq
We would like to show \( f(x) \leq g(x)\) for all \(\p[1], \p[2] \in [0,1]\), \(\p[12] \in [0\vee \p[1] + \p[2] - 1, \p[1] \wedge \p[2]]\) and \(x \in\R\). Note that, by symmetry, it is enough to consider \(x \geq 0\). \\
Re-parametrizing \(\p[1] = \frac{1}{2} (a-b+2 \p[12])\) and  \(\p[2] = \frac{1}{2} (a+b+2 \p[12])\), we have
\beq
f(x)
& \eqdef &
\log \left(\frac{1}{2} \mathe^{(b-1) x} \left(a \left(\mathe^x-1\right)^2-b \mathe^{2 x}+b+2 \mathe^x\right)\right) ,
\\
g(x)
& \eqdef &
\frac{x^2}{\log \left(\frac{2}{a}\right)}.
\eeq
Note that $g(0)-f(0) = 0$. If we show that $\frac{\partial}{\partial x} (g-f)=0$ for $1 \geq a\geq |b|,x \geq 0$, we are done. First, we multiply the first derivative by the non-negative 
\(\log \left(\frac{2}{a}\right) \left(a w^2-b (w+1)^2+b+2 (w+1)\right)\), 
and then change variables \(x \eqdef \log(w+1)\).
\beq
H &\eqdef&
\paren{\frac{\partial}{\partial x} (g-f)} \log \left(\frac{2}{a}\right) \left(a w^2-b (w+1)^2+b+2 (w+1)\right) \\
&=&
2 \log (w+1) \left(w^2 (a-b)-2 (b-1) w+2\right)+w \log \left(\frac{2}{a}\right) (b (b (w+2)+w)-a (b w+w+2)),
\eeq
whose non-negativity we must verify
for \(w \geq 0 \). \\
We now consider two cases, first \(0 \leq w<3\) and then \(w \geq 3\).

\paragraph{Case 1: \(0 \leq w < 3\).}
It can be shown that 
the coefficient of \(\log(1+w)\) is positive, because it is a quadratic polynomial in \(w\) with positive coefficients, so \(\log(1+w)\) can be replaced with something smaller (in that range), such as \(\frac{w}{\mathe}\). After doing that, we have
\[
H \geq
2 \mathe \left(b^2-a\right) \log \left(\frac{2}{a}\right)+2 w^2 (a-b)+w \left(-\mathe (b+1) (a-b) \log \left(\frac{2}{a}\right)-4 b+4\right)+4 =: I,
\]
which is non-negative for \(w=0\), so we can show that the derivative (multiplied by a non-negative),
\[ \paren{\frac{\partial}{\partial w} I} \frac{\mathe}{w}= 4 w (a-b)+\mathe (b+1) (b-a) \log \left(\frac{2}{a}\right)-4 b+4,
\] is non-negative.
The above expression is increasing in \(w\), so the worst case is \(w=0\), in which case we have
\begin{align*} 
4 w (a-b)+\mathe (b+1) (b-a) \log \left(\frac{2}{a}\right)-4 b+4 = -\mathe (b+1) (a-b) \log \left(\frac{2}{a}\right)-4 b+4.
\end{align*}

The right-hand side is a quadratic polynomial in \(b\) with the coefficient of $b^2$ being positive, therefore it is convex in \(b\), 
so every tangent lies below the expression. We use the tangent at the point where $b$ has the value $a$, and we get the following linear expression in $b$.
\begin{align*}
-\mathe a^2 \log \left(\frac{2}{a}\right)+b \left(\mathe (a+1) \log \left(\frac{2}{a}\right)-4\right)-\mathe a \log \left(\frac{2}{a}\right)+4
\\
 = b \left(\mathe (a+1) \log \left(\frac{2}{a}\right)-4\right)+\mathe (-a) (a+1) \log \left(\frac{2}{a}\right)+4.
\end{align*}

We now consider the two endpoints of the above line, \(b=a\) and \(b=-a\), and show that both are positive.
For \(b=a\) the expression is just \(4(1-a)\) and therefore non-negative. At the second endpoint, \(b=-a\), the expression has value of
\[
b \left(\mathe (a+1) \log \left(\frac{2}{a}\right)-4\right)+\mathe (-a) (a+1) \log \left(\frac{2}{a}\right)+4
=
2 \paren{2-\mathe a \log \left(\frac{2}{a}\right)} (1+a).
\]
The term \((1+a)\) is positive, so we are left with \(\paren{2-\mathe a \log \left(\frac{2}{a}\right)}\), which is convex since the second derivative is \(\frac{\mathe}{a}\) and has a minimum at \(a = \frac{2}{\mathe}\) with the value \(0\).
The \(0 \leq w < 3\) part is done, we now move on to the second part, \( w \geq 3 \). 

\paragraph{Case 2: \( w \geq 3 \).} Recall that we have to show \(H \geq 0\) for \(a \geq \abs{b}, x \geq 0\). \(H\) is convex in \(b\), because
\[
\frac{\partial^2}{\partial b^2} H
=
2 w (w+2) \log \left(\frac{2}{a}\right)
\geq
0,
\]
so we can lower bound \(H\) by any of its tangent lines. The tangent line of \(H\) where \(b = a\) is
\begin{align*}
T(w,a,b) \eqdef & 
w \log \left(\frac{2}{a}\right) \left(-\left(a^2 (w+2)\right) + a (b (w+4)-w-2)+b w\right)
\\
& + 2 \log (w+1) \left(w^2 (a-b)-2 (b-1) w+2 \right).
\end{align*}
Since we have a linear expression in \(b\), we can check the endpoints of \(b\),  \(a\) and \(-a\), and be done. Starting with \(b = a\), we have
\[
T(w,a,a) = 2 (a-1) a w \log \left(\frac{2}{a}\right)+4 (-a w+w+1) \log (w+1).
\]
Observe that \(a \log(2/a)\) is concave and has a maximum of \(\frac 2 \mathe\). Therefore, we can replace \(a \log(2/a)\) with  \(\frac 2 \mathe\), divide everything by \( 4 \) and get
\[
R(a) \eqdef a \left(\frac{w}{\mathe}-w \log (w+1)\right)-\frac{w}{\mathe}+w \log (w+1)+\log (w+1).
\]
{We now analyze two cases:}
\(
\frac{w}{\mathe}-w \log (w+1)\leq 0
\)
{and}
\(
\frac{w}{\mathe}-w \log (w+1)>0
\).
If
\(
\frac{w}{\mathe}-w \log (w+1)\leq 0
\)
\text{then the} \text{worst}
\(
a
\)
\text{is}
\(
a=1
\)
\text{, for which, we have,}
\(R(1) = \log (w+1)\),
which is non-negative.
\text{If}
\(
\frac{w}{\mathe}-w \log (w+1)>0
\)
\text{then the} \text{worst}
\(
a
\)
\text{is}
\(
a=0
\)
\text{, for that, we have}
\[
R(0) = (w+1) \log (w+1)-\frac{w}{\mathe}
\]
Which is non-negative since it is \(0\) at \(w=0\) and the first derivative,
\(
\frac{\partial}{\partial w} R = (w+1) \log (w+1)-\frac{w}{\mathe},
\)
is positive.
We now have to turn to the other endpoint of
\(b\)
\text{, where it is }
\(-a\).
In that case, we get
\beqn
\label{eq:R_b=-a}
T(w,a,-a) = 4 (w+1) (a w+1) \log (w+1)-2 a w (a (w+3)+w+1) \log \left(\frac{2}{a}\right)
.
\eeqn
We now consider two cases, \(a \leq 1/10\) and \(a>1/10\). For \(a>1/10\), we do as in the other endpoint of \(b\) and bound  \(a \log (\frac 2 a)\) by \(2/\mathe\), divide by $4$, and get
\[
T(w,a,-a)/4 \leq a \left(w (w+1) \log (w+1)-\frac{w (w+3)}{\mathe}\right)+\frac{(w+1) (\mathe \log (w+1)-w)}{\mathe}.
\]
Split to two cases, \(w (w+1) \log (w+1)-\frac{w (w+3)}{\mathe}\) is non-negative or negative. If it is negative, then the worst where \(a=1\), in which case the above expression becomes
\begin{align*}
 \left(w (w+1) \log (w+1)-\frac{w (w+3)}{\mathe}\right)+\frac{(w+1) (\mathe \log (w+1)-w)}{\mathe}
\\
=
(w+1)^2 \log (w+1)-\frac{2 w (w+2)}{\mathe}.
\end{align*}
The above is non-negative because \(w \geq 3\) it is positive for \(w=3\) and its derivative, \\ \( \frac{(w+1) (2 \mathe \log (w+1)+\mathe-4)}{\mathe} \), is also positive for \(w \geq 3\). \\
If
\( w (w+1) \log (w+1)-\frac{w (w+3)}{\mathe} \)
\text{ is positive, then the worse is }
\(a=\frac{1}{10}\). Plugging it into eq. (\ref{eq:R_b=-a}) results in
\[
T(w,1/10,-1/10)
 =
\mathe \left(w^2+11 w+10\right) \log (w+1)-w (11 w+13),
\]
which is again positive since it is positive for \( w=3\) and the first derivative is positive.
Finally, we are left with the \(a \leq 1/10\) case.
Lastly, we need to prove
\[
T(w,a,-a) = 4 (w+1) (a w+1) \log (w+1)-2 a w (a (w+3)+w+1) \log \left(\frac{2}{a}\right) \geq 0
\]
for \(0 \leq a \leq 1/10\) and \(w \geq 3\). 
We do this first by showing that
(1)
\(T(3,a,-a) > 0\) 
and
(2)
\( \left. \frac{\partial}{\partial w} T(w,a,-a) \right|_{w=3} > 0\)
for the appropriate range.
Then, proving 
(3)
\(\frac{\partial^2}{\partial w^2} T(w,a,-a) \geq 0\) for the appropriate range
completes the proof.
First, taking care of \(T(3,a,-a)\), we have
\begin{align*}
    T(3,a,-a) 
    & =
    3 \left(-5 a^2+(-7 a-5) a-3 a\right) \log \left(\frac{2}{a}\right)+2 (-6 (-a-1)+18 a+2) \log (4)
    \\
    & = 
    4 \left(4 (a \log (64)+\log (4))-3 a (3 a+2) \log \left(\frac{2}{a}\right)\right),
\end{align*}
which has the following positive second derivative, 
\begin{align*}
    \frac{\partial^2}{\partial a^2} T(3,a,-a) 
    & =
    4 \left(-3 a \left(\frac{3 a+2}{a^2}-\frac{6}{a}\right)-6 \left(3 \log \left(\frac{2}{a}\right)-\frac{3 a+2}{a}\right)\right)
    \\
    & = 
    12 \left(\frac{2}{a}-6 \log \left(\frac{2}{a}\right)+9\right).
\end{align*}
This second derivative is positive because it is decreasing by the fact \(\frac{\partial^3}{\partial a^3} T(3, a, -a) = \frac{24 (3 a-1)}{a^2} < 0\) and the minimum of \(\frac{\partial^2}{\partial a^2} T(3, a, -a) \) at \(a \in [0,1/10] \),
\[
\left. \frac{\partial^2}{\partial a^2} T(3,a,-a) \right|_{a=1/10} = 12 (29-6 \log (20)) \approx 132.307,
\]
is positive.
Knowing that \(\frac{\partial^2}{\partial a^2} T(3, a, -a) \) is convex, we lower bound it by its tangent line at \(a=0.095\), \[a \left(\frac{1371}{50}+16 \log (64)-\frac{771}{25} \log \left(\frac{400}{19}\right)\right)+16 \log (4)+\frac{57 \left(57 \log \left(\frac{400}{19}\right)-457\right)}{10000} \approx 20.566 - 0.008 a.\] 
This tangent line is decreasing and is positive at \(a \in [0,1/10] \). Therefore, \(\frac{\partial^2}{\partial a^2} T(3,a,-a)\) is positive at \(a \in [0,1/10] \).

Next, we analyze
\[
\left. \frac{\partial}{\partial w} T(w,a,-a) \right|_{w=3}
=
2 \left(2 a (3+7 \log (4))-a (9 a+7) \log \left(\frac{2}{a}\right)+2+\log (16)\right),
\]
which has a positive second derivative, \(\frac{2 (18 a-7)}{a^2}\), and thus we can use a tangent line at point \(a=0.025\) as follows. We have,
\begin{align*}
    \left. \frac{\partial}{\partial w} T(w,a,-a) \right|_{w=3}
    & \geq
    \frac{1}{800} (2911+6436 \log (2)+9 \log (5))-\frac{1}{20} a (-529+72 \log (2)+298 \log (5))
    \\
    & \approx
    9.23323 - 0.0259547 a,
\end{align*}
which is positive for \(a \in [0,1/10] \).

Third and last, we show that
\(\frac{\partial^2}{\partial w^2} T(w,a,-a) \geq 0\).
We have that
\begin{align*}
\frac{\partial^2}{\partial w^2} T(w,a,-a)
& =
\frac{12 a w-4 (a+1) a (w+1) \log \left(\frac{2}{a}\right)+8 a (w+1) \log (w+1)+8 a+4}{w+1}
\\
& \defeq
\frac{U(w)}{w+1},
\end{align*}
thus it is sufficient to prove \(U(w) > 0\). This is done by finding the critical point, proving that it is the minimum point, and then showing that \(U\) is positive at that point.
To locate the minimum point, we solve 
\[
U'(w) = 8 a \log (w+1)+20 a-4 (a+1) a \log \left(\frac{2}{a}\right) = 0
\]
for \(w\)
and get
\[
w_0 = \mathe^{\frac{1}{2} \left(a \log \left(\frac{2}{a}\right)+\log \left(\frac{2}{a}\right)-5\right)}-1.
\]
This is indeed the minimum point because
\[
U''(w) = \frac{8 a}{w+1} > 0
\]
for our ranges of \(w, a\) (unless \(a=0\) but then \(U(w) > 0\) immediately).
To show that \(U(w_0) > 0\), we first compute
\begin{align*}
    U(w_0)
    = &
    8 a+12 a \left(\mathe^{\frac{1}{2} \left(a \log \left(\frac{2}{a}\right)+\log \left(\frac{2}{a}\right)-5\right)}-1\right)-4 (a+1) a \mathe^{\frac{1}{2} \left(a \log \left(\frac{2}{a}\right)+\log \left(\frac{2}{a}\right)-5\right)} \log \left(\frac{2}{a}\right)
    \\& +8 a \mathe^{\frac{1}{2} \left(a \log \left(\frac{2}{a}\right)+\log \left(\frac{2}{a}\right)-5\right)} \log \left(\mathe^{\frac{1}{2} \left(a \log \left(\frac{2}{a}\right)+\log \left(\frac{2}{a}\right)-5\right)}\right)+4
    \\
    = &
    \frac{
    3\ 2^{\frac{a+5}{2}} \left(\frac{1}{a}\right)^{\frac{a-1}{2}}-4 \mathe^{5/2} a-2^{\frac{a+5}{2}} (a+1) \left(\frac{1}{a}\right)^{\frac{a-1}{2}} \log \left(\frac{2}{a}\right)
    }{\mathe^{5/2}}
    \\
    & + \frac{\frac{1}{2} 2^{\frac{a+7}{2}} \left(\frac{1}{a}\right)^{\frac{a-1}{2}} \left((a+1) \log \left(\frac{2}{a}\right)-5\right)+4 \mathe^{5/2}}{\mathe^{5/2}}
    \\
    & =
    -\frac{2^{\frac{a+7}{2}} \left(\frac{1}{a}\right)^{\frac{a-1}{2}}}{\mathe^{5/2}}-4 a+4,
\end{align*}
and additionally
\begin{align*}
\frac{\partial}{\partial a} U(w_0)
& = 
-\frac{2^{\frac{a+7}{2}-1} \left(\frac{1}{a}\right)^{\frac{a-1}{2}} \log (2)}{\mathe^{5/2}}-\frac{2^{\frac{a+7}{2}} \left(\frac{1}{a}\right)^{\frac{a-1}{
2}} \left(\frac{1}{2} \log \left(\frac{1}{a}\right)-\frac{a-1}{2 a}\right)}{\mathe^{5/2}}-4
\\
& =
-\frac{4 \left(2^{\frac{a+1}{2}} \left(\frac{1}{a}\right)^{\frac{a+1}{2}} (a (\log (2)-1)+1)+2^{\frac{a+1}{2}} \left(\frac{1}{a}\right)^{\frac{a-1}{2}} \log \left(\frac{1}{a}\right)+\mathe^{5/2}\right)}{\mathe^{5/2}},
\end{align*}
and note that 
the latter expression
is negative for \(a \in [0,1/10]\).
Thus, \(U (w_0) \) is decreasing in \(a\) and at the minimum point has a value of
\[
\left. U(w_0) \right|_{a=1/10}
=
\frac{18}{5}-\frac{8 \sqrt[10]{2}}{5^{9/20} \mathe^{5/2}}
\approx
3.25887
> 0.
\]
The proof is complete.
\end{proof}

\begin{lemma}
\label{lem:rho-metric}
The function $\rho:\N^2\to\R_+$
defined in 
\eqref{eq:rho-metric}
satisfies the metric axioms.
\end{lemma}
\begin{proof}
    Let $f(x) \eqdef \frac{2}{\sqrt{3}} \mn \sqrt{\frac{2}{\log \frac{2}{x}}}$ where $f(0)=0$ such that we have $\rho = f \circ \xi$. It is known that non-negative, non-decreasing, concave functions with $f\inv(0) = \set{0}$ are metric-preserving \cite[for example][p. 70]{kaplansky2001set}.
    It is easy to see that $f$ is satisfies $f\inv(0) = \set{0}$, and is non-decreasing.
    To see that $f$ is concave, observe that
    \begin{align*}
        \frac{\partial^2}{\partial x^2}
        \sqrt{\frac{2}{\log \left(\frac{2}{x}\right)}}
        &=
        \sqrt{2} \left(\frac{3 \left(\frac{1}{\log \left(\frac{2}{x}\right)}\right)^{5/2}}{4 x^2}-\frac{\left(\frac{1}{\log \left(\frac{2}{x}\right)}\right)^{3/2}}{2 x^2}\right)
        \\
        &=
        \frac{\left(3-2 \log \left(\frac{2}{x}\right)\right) \left(\frac{1}{\log \left(\frac{2}{x}\right)}\right)^{5/2}}{2 \sqrt{2} x^2}
    \end{align*}
    is negative for $ 0<x<\frac{2}{e^{3/2}}$, which is where $\sqrt{\frac{2}{\log \left(\frac{2}{x}\right)}} \leq
    \frac{2}{\sqrt{3}}$. Since the minimum of concave functions is concave, we are done.
\end{proof}

Recall the notation $\Ncal_\rho(\eps)$ for the $\eps$-covering number of $(\N,\rho)$. Because $(X_i)_{i\geq 1}$ is a subgaussian process on $(\N,\rho)$, Dudley's theorem directly gives upper bounds on $\Delta_n(\mu)$.

\vspace{2mm}

\begin{proof}{\textbf{of Proposition~\ref{prop:dudley_thm}}.}
    \cref{lem:pqr-mgf} shows that the vector $\pn-p$ is sub-Gaussian with respect to the metric $\frac{\rho}{\sqrt n}$. As a result, Dudley's theorem (\citet[Corollary 5.25]{van2014probability}) shows that
    \begin{equation*}
        \Delta_n(\mu) = \Ebb \sup_{i\in\N} |\pn[i]-\p[i]| \leq 24\int_0^\infty \sqrt{\log \Ncal_\rho(\eps\sqrt n)}\mathd\eps = \frac{24}{\sqrt n}\int_0^\infty \sqrt{\log \Ncal_\rho(\eps)}\mathd\eps.
    \end{equation*}
    In the last equality, we noted that for any $\eps\geq 1$, one has $\Ncal_\xi(\eps)=1$. This ends the proof.
\end{proof}

As a remark, we can check that the proposed condition $\int_0^1 \sqrt{\Ncal_\rho(\eps)}\mathd\eps<\infty$ in Proposition~\ref{prop:dudley_thm} is stronger than the sufficient condition $\int_0^1 \Ncal_\xi(\eps)\mathd\eps<\infty$ from \cref{thm:sufficient_condition_covering_nb}. Indeed, suppose that one has $\int_0^1 \sqrt{\Ncal_\rho(\eps)}\mathd\eps<\infty$, then in particular, $\eps\sqrt{\log\Ncal_\rho(\eps)} \underset{\eps\to 0^+}{\longrightarrow} 0.$ This implies that for any $c>0$, we have
\begin{equation*}
    e^{-\frac{c}{\eps^2}} \Ncal_\rho(\eps)  \underset{\eps\to 0^+}{\longrightarrow} 0.
\end{equation*}
Also, for $\eps\in(0,\frac{2}{\sqrt 3}]$, we have
\begin{equation*}
    \Ncal_\rho(\eps) = \Ncal_\xi\paren{2e^{-\frac{2}{\eps^2}}}.
\end{equation*}
As a result, this shows that for any $c>0$, we have
\begin{equation*}
    \eps^c \Ncal_\xi(\eps) \underset{\eps\to 0}{\longrightarrow} 0.
\end{equation*}
The above bound for any $c<1$ already shows that $\int_0^1\Ncal_\xi(\eps)\mathd\eps<\infty$.

\section{On the exact conditions for the convergence \mathinhead{$\Delta_n(\mu)\ninf 0$}}
\label{sec:examples}

In this section, we provide sufficient conditions for the convergence of the expected maximum deviation $\Delta_n$ and identify some key challenges for the general characterization.

We start by proving \cref{thm:more_general_suff_condition}. This gives a sufficient condition for the decay of $\Delta_n$ to $0$ that is a significantly weaker condition than the condition Eq~\eqref{eq:sufficient_condition} from \cref{thm:sufficient_condition_covering_nb}. To the best of our knowledge, we are not aware of any distribution that would not satisfy it, but would still exhibit the convergent behavior for $\Delta_n$. For the sake of exposition, we recall the condition from \cref{thm:more_general_suff_condition} for distributions $\mu$ on $\{0,1\}^\N$:

\paragraph{Sufficient Condition (SC)} \textit{
The metric space $(\N,\xi)$ is totally bounded and there exists $K\geq 1$ such that for any $\eps>0$, there exist events $(E_k)_{k\in\N}$ and a finite set $J\subset \N$ with
    \begin{itemize}
        \item $\Pbb(E_k)\leq \eps,\;\forall k\in\N$,
        \item $\displaystyle \sup_{k\in\N} \frac{\log(k+1)}{\log\frac{1}{\Pbb(E_k)}}<\infty$,
        \item $\forall i\in\N, \exists j\in J, \exists \Kcal\subset \N,  \text{ such that } |\Kcal|\leq K \text{ and } \{X_i\neq X_j\}\subset \bigcup_{k\in \Kcal} E_k$.
    \end{itemize}
    Then $\Delta_n(\mu) \ninf 0$.
}

\vspace{2mm}

\begin{proof}{\textbf{of \cref{thm:more_general_suff_condition}}}
    Fix $\eps>0$ and consider the events $(E_k)_{k\in\N}$ as provided by the condition. Fix $n\geq 1$. Mirroring the notation for $\pn$, we define $\qn$ (resp. $\un$) as the empirical probability vector for the variable $(\1[E_k])_{k\in\N}$ (resp. $(X_j)_{j\in J}$). That is, if we denote by an exponent $(n)$ different samples from these random variables, we pose
    \begin{equation*}
        \qn[k] = \frac{1}{n}\sum_{i=1}^n \1[E_k^{(i)}],\qquad \un[j] = \frac{1}{n}\sum_{i=1}^n X_j^{(i)}.
    \end{equation*}
    From \citet{CohenK23a,blanchard2023tight}, we know that since
    \begin{equation*}
        T(E):=\sup_{k\in\N} \frac{\log(k+1)}{\log\frac{1}{\Pbb(E_k)}}<\infty,
    \end{equation*}
    letting $\q[k] = \Pbb(E_k)\leq \eps$, one has
    \begin{equation}\label{eq:convergence_q}
        \Ebb \|\qn - \q\|_\infty \underset{n\to\infty}{\longrightarrow}0.
    \end{equation}
    Next, for any $i\in\N$, let $\Kcal_i\subset \N$ and $j_i\in [J]$ be the set of indices indices such that $\{X_i\neq Z_{j_i}\}\subset \bigcup_{k\in \Kcal_i} E_k$. Then, with $\u[j] = \Ebb[Z_j]$ for $j\in[J]$, we have
    \begin{align*}
        |\pn[i] - \p[i]| &\leq |\pn[i] - \un[j_i]| + |\un[j_i]-\u[j_i]| + |\u[j_i]-\p[i]|\\
        &\leq \frac{1}{n} \sum_{l=1}^n \1[X_i^{(l)}\neq X_{j_i}^{(l)}] + |\un[j_i]-\u[j_i]| + \Pbb(X_i\neq X_{j_i})\\
        &\leq  \frac{1}{n} \sum_{l=1}^n \sum_{k\in \Kcal_i}\1[E_k^{(l)}] + \|\un-\u\|_\infty + \sum_{k\in \Kcal_i}\Pbb(E_k)\\
        &\leq \sum_{k\in \Kcal_i}(\qn[k] + \eps) + \|\un-\u\|_\infty.
    \end{align*}
    Next, for any $k\in \Kcal_i$,
    \begin{equation*}
        \qn[k]\leq \q[k] + |\qn[k]-\q[k]| \leq \Pbb(E_k) + \|\qn-\q\|_\infty \leq \eps +\|\qn-\q\|_\infty.
    \end{equation*}
    Putting the two previous inequalities together yields
    \begin{equation*}
        \|\hat p-p\|_\infty \leq (2\eps + \|\qn -\q\|_\infty)K +  \|\un-\u\|_\infty.
    \end{equation*}
    Because $J$ is finite, $\Ebb \|\un-\u\|_\infty \to 0$. Together with Eq~\eqref{eq:convergence_q}, this gives
    \begin{equation*}
        \limsup_{n\to\infty}\Delta_n(\mu)\leq 2K\eps .
    \end{equation*}
    This holds for any $\eps>0$, hence we obtained the desired result $\Delta_n(\mu)\to 0$ as $n\to\infty$.
\end{proof}

An inspection of the proof shows that one does not need the random variables $(X_j)_{j\in J}$ used as ``centers'' to belong to the set of components $\{X_i,i\geq 1\}$. In fact, the proof holds if we put no restriction on these centers. This yields the following result.

\begin{corollary}\label{cor:condition_useless}
    Let $\mu$ be a distribution on $\{0,1\}^\N$ such that $(\N,\xi)$ is totally bounded. Suppose that there exists $K\geq 1$ such that for any $\eps>0$, there exist events $(E_k)_{k\in\N}$, and a finite sequence of random variables $(Z_j)_{j\in [J]}$ (defined on the same probability space as $\mu$) with
    \begin{itemize}
        \item $\Pbb(E_k)\leq \eps,\;\forall k\in\N$,
        \item $\displaystyle \sup_{k\in\N} \frac{\log(k+1)}{\log\frac{1}{\Pbb(E_k)}}<\infty$,
        \item $\forall i\in\N, \exists j\in J, \exists \Kcal\subset \N,  \text{ such that } |\Kcal|\leq K \text{ and } \{X_i\neq Z_j\}\subset \bigcup_{k\in \Kcal} E_k$.
    \end{itemize}
    Then $\Delta_n(\mu) \ninf 0$.
\end{corollary}

While this condition seems more general than the condition (SC), they turn out to be equivalent.

\begin{proposition}
    The condition (SC) is equivalent to the condition from Corollary~\ref{cor:condition_useless}.
\end{proposition}

\begin{proof}
    It suffices to show that if $\mu$ satisfies the condition from Corollary~\ref{cor:condition_useless}, then it also satisfies (SC). Fix such a distribution. We will use all the notations of the condition and we now aim to find adequate parameters to satisfy (SC). We will use $\tilde K := K+1$. Fix $\eps>0$. Because $\mu$ satisfies the condition in Corollary~\ref{cor:condition_useless} for $\tilde \epsilon = \frac{\epsilon}{K}$, there exists events $(E_k)_{k\in\N}$ and random variables $(Z_j)_{j\in[J]}$ satisfying the conditions for $\tilde \epsilon$. We also fix for $i\geq 1$, an element $j_i\in J$ and set $\Kcal_i\subset \N$ such that
    \begin{equation}\label{eq:def_cond}
        \{X_i\neq Z_{j_i}\}\subset \bigcup_{k\in\Kcal_i}E_k.
    \end{equation}

    Fix $j\in [J]$. First suppose that
    \begin{equation*}
        \Pbb(Z_j \neq X_i)> \epsilon,\quad i\geq 1.
    \end{equation*}
    Then, we can check that $j_i\neq j$, because
    \begin{equation*}
        \Pbb(X_i\neq Z_{j_i}) \leq \Pbb\paren{\bigcup_{k\in \Kcal_i}E_k} \leq \sum_{k\in \Kcal_i}\Pbb(E_k) \leq K\tilde\eps  = \eps.
    \end{equation*}
    As a result, the variable $Z_j$ is simply not needed and we can delete it from the set of centers $(Z_j)_{j\in [J]}$. We can therefore suppose without loss of generality that for all $j\in [J]$, there is some $i(j)\geq 1$, for which
    \begin{equation*}
        \Pbb(X_{i(j)}\neq Z_j)\leq \eps.
    \end{equation*}
    We then define the event $F_j:=\{X_{i(j)}\neq Z_j\}$ for all $j\in [J]$ and add all these to the sequence of covering events $(E_k)_{k\geq 1}$ by defining
    \begin{equation*}
        \tilde E_j:=\begin{cases}
            F_j&j \leq J\\
            E_{j-J} & j>J.
        \end{cases}
    \end{equation*}
    The first condition for (SC) is satisfied by construction of the events $F_j$ for $j\in[J]$ because $\Pbb(F_j) = \Pbb(X_{i(j)}\neq Z_j)\leq \eps$. Next, we only added a finite number of events to the sequence, hence the second property is still valid. Last, for $i\geq 1$, because Eq~\eqref{eq:def_cond} holds, we have
    \begin{equation*}
        \{X_i\neq X_{i(j_i)}\} \subset \{X_i\neq Z_{j_i}\}\cup \{X_{i(j_i)}\neq Z_{j_i}\} \subset F_j\cup \bigcup_{k\in \Kcal_i}E_k = \bigcup_{k\in \{j_i\}\cup \{k+J,k\in \Kcal_i\}}\tilde E_k
    \end{equation*}
    This ends the proof that $\mu$ satisfies condition (SC), which gives the desired result.
\end{proof}

The proposed condition (SC) essentially asks that ``bad events'' $\{X_i\neq X_j\}$ can be adequately covered by some sequence of events $(E_k)_{k\in\N}$. As discussed in \cref{sec:main_results}, this significantly generalizes the condition $\int_0^1 \Ncal_\xi(\eps)\mathd\eps<\infty$ along two directions.

\subsection{Generalization (i)} 
We showed in \cref{thm:2procs} and Proposition~\ref{prop:3rd_order_moments} that $2$nd and $3$rd order moment information on the distribution $\mu$ is not enough to have a necessary and sufficient characterization. The condition (SC) instead covers deviations via events in the probability space of $\mu$ directly, which allows for correlations with an arbitrarily large number of coordinates. 

For instance, we can check how the condition (SC) distinguishes between the two distributions $\mu$ and $\nu$ from \cref{thm:2procs}. For $\nu$, because the variables $C_t$ are independent even within each block $t\in\Scal_k$ for some fixed $k\geq 1$, there is no convenient choice of covering events $E_k$. On the other hand, for $\mu$, one can directly choose the bad events $E_k := \{B_k=1\}$:
\begin{equation*}
    \{X_t^\mu \neq Z_0\} \subset \{B_k=1\} = E_k,\quad t\in\Scal_k,k\geq 1.
\end{equation*}
We can therefore cover the deviations of all components $X_t^\mu$ for $t\in\Scal_k$ using a single event $E_k$ with small probability $\Pbb(E_k) = 2^{-k}$. However, for any $t\neq t'\in\Scal_k$, one has
\begin{equation*}
    \xi(t,t') = \Pbb(X_t^\mu\neq X_{t'}^\mu) = \frac{\Pbb(E_k)}{2}.
\end{equation*}
Hence, contrary to (SC), the covering number approach severely suffers from the size of the block $|\Scal_k|$ (so would any approach that looks at a fixed number of components at once).

\subsection{Generalization (ii)}

The condition (SC) allows to cover the bad event $\{X_i\neq Z_j\}$ potentially with several events $E_k$ (at most $K$), which departs from standard coverings for which one aims to directly cover the probability $\Pbb(X_i\neq X_j)$. The alternative condition would be written as follows.

\paragraph{Tentative Condition 1 (TC1)} \textit{
The metric space $(\N,\xi)$ is totally bounded and for any $\eps>0$, there exist events $(E_k)_{k\in\N}$ and a finite set $J\subset \N$ with
    \begin{itemize}
        \item $\Pbb(E_k)\leq \eps,\;\forall k\in\N$,
        \item $\displaystyle \sup_{k\in\N} \frac{\log(k+1)}{\log\frac{1}{\Pbb(E_k)}}<\infty$,
        \item $\forall i\in\N, \exists j\in J, \exists k\in \N,  \{X_i\neq X_j\}\subset E_k$.
    \end{itemize}
}

While this is still sufficient by \cref{thm:more_general_suff_condition}, we can show that it is not necessary.

\begin{proposition}[(TC1) is not necessary]
\label{prop:TC1_non_necessary}
    There exists a probability measure $\mu$ on $\{0,1\}^\N$ that does not satisfy condition (TC1) but $\Delta_n(\mu)\ninf 0$.
\end{proposition}

\begin{proof}
     Let $(Y_k)_k\overset{\iid}{\sim}\Bernu(1/2)$ and independent 
random variables $A_k$ such that $A_k\sim\Bernu(1/\sqrt k)$. Put
\begin{equation*}
    X_k = (1-A_k) Y_0 + A_k Y_k,\quad k\geq 1.
\end{equation*}
We then define the distributions $\mu$ such that $(X_k)_{k\geq 1}\sim\mu$.

We first show that $\Delta_n(\mu)\to 0$ as $n\to\infty$ by checking that it satisfies the condition from Corollary~\ref{cor:condition_useless}. Intuitively, the random variables become closer and closer to $Y_0$, hence we can choose $Z_1:=Y_0$. Fix $\eps>0$ and let $k_\eps\geq 1/\eps^2$. We can then pose $J = 1+k_\eps$ and $Z_j = X_{j-1}$ for $2\leq j\leq J$. For the other variables, we can simply pose $E_k:= \{X_{k+k_\eps}\neq Y_0\}$ for $k\geq 1$. For the covering sets, we can simply pose $K_k = \{1\}$ for $k\leq k_\eps$ and $K_k = \{k-k_\eps\}$ for $k>k_\eps$. 
We can check that these parameters satisfy the condition from Corollary~\ref{cor:condition_useless} in a straightforward manner. For $k\geq 1$, $\Pbb(E_k) = \Pbb(X_{k+k_\eps}\neq Y_0) = \frac{1}{\sqrt{k+k_\eps}}\leq \eps$. Next, because these probabilities decay as $\frac{1}{\sqrt{k_\eps+k}}$, the second condition in (SC) is satisfied. Last, for $i\leq k_\eps$ we have $\{X_i\neq Z_{i+1}\} = \{X_i\neq X_i\} = \emptyset \subset E_1$, and for $i>k_\eps$, $\{X_i\neq Z_1\} = \{X_i\neq Y_0\} = E_{i-k_\eps}$. This ends the proof that the condition from Corollary~\ref{cor:condition_useless} is satisfied and as a result,
\begin{equation*}
    \Delta_n(\mu)\ninf 0.
\end{equation*}

Next, suppose by contradiction that (TC1) is satisfied. We use this property for $\eps= \frac{1}{4}$, using the same notations as in the condition. For $i\geq 1$, we also denote $j_i\in J$ and $k_i\in\N$ elements such that $\{X_i\neq X_{j_i}\}\subset E_{k_i}$. Because $J$ is finite, we denote $j_{max} = \max\{j,j\in J\}$. Then, for any $i\geq 4$, we have
\begin{equation*}
    \Pbb(X_i\neq X_{j_i}) = \frac{1}{2}\paren{\frac{1}{\sqrt i}+\frac{1}{\sqrt {j_i}} - \frac{1}{\sqrt{ij_i}}} \geq \frac{1}{2}\paren{\frac{1}{\sqrt {j_i}} - \frac{1}{\sqrt{4j_i}}} = \frac{1}{4\sqrt {j_i}} \geq \frac{1}{4\sqrt {j_{max}}}.
\end{equation*}
Next, the second condition of (TC1) implies in particular that there exists $k_{max}$ such that
\begin{equation*}
    \Pbb(E_k) < \frac{1}{4K\sqrt{j_{max}}},\quad k\geq k_{max}.
\end{equation*}
Now recall that for $i\geq 4$, one has $\Pbb(X_i\neq X_{j_i}) \leq \Pbb(E_{k_i})$. Combining the two last equations shows that for any $i\geq 4$, we have $k_i< k_{max}$. Recalling that $j_i$ can only take $|J|$ values, this implies that there is some couple $(j,k)\in J\times [k_{max}-1]$ for which the set
\begin{equation*}
    \Scal(j,k):=\{4\leq i \leq (4k_{max}|J|)^2+3, (j_i,k_i)=(j,k)\},
\end{equation*}
has at least $16k_{max}|J|$ elements. Next, note that
\begin{align*}
    \bigcup_{i\in\Scal(j,k)} \{X_i\neq X_j\} \subset E_k.
\end{align*}
Hence, taking the probabilities yields
\begin{align*}
    1-\Pbb(E_k) \leq \Pbb\paren{\bigcap_{i\in\Scal(j,k)} \{X_i = X_j\}}
    &\leq \Pbb(X_j\neq Y_0) + \Pbb\paren{\bigcap_{i\in\Scal(j,k)} \{X_i = Y_0\} }\\
    &\leq \frac{1}{2\sqrt j} + \prod_{i\in\Scal(j,k)} \paren{1-\frac{1}{2\sqrt i}}\\
    &\leq \frac{1}{2} + \exp\paren{-\sum_{i\in\Scal(j,k)}  \frac{1}{2\sqrt i}}.
\end{align*}
Now we compute
\begin{equation*}
    \sum_{i\in\Scal(j,k)}  \frac{1}{2\sqrt i} \geq \frac{16k_{max}|J|}{2\sqrt{(4k_{max} |J|)^2+3}} \geq \frac{3}{2}.
\end{equation*}
Together with the previous equation, this implies
\begin{equation*}
    \Pbb(E_k) \geq \frac{1}{2} - e^{-3/2} >\frac{1}{4}=\eps.
\end{equation*}
This contradicts the first property of condition (TC1), which proves that $\mu$ does not satisfy this condition and ends the proof.
\end{proof}

In the previous example, one of the main reasons why the condition (TC1) is not satisfied is that there is no adequate finite choice of ``centers'' $(X_j)_{j\geq 1}$---the random variable $Y_0$ is missing from the sequence $\{X_j,j\geq 1\}$, while it would be a natural candidate to be used as the center. A possible tentative to fix this issue would be to allow the centers to be general random variables, in the spirit of the condition proposed in Corollary~\ref{cor:condition_useless}. This yields the following condition.

\paragraph{Tentative Condition 2 (TC2)} \textit{
The metric space $(\N,\xi)$ is totally bounded and for any $\eps>0$, there exist events $(E_k)_{k\in\N}$ and a finite set of random variables $(Z_j)_{j\in [J]}$ (defined on the same probability space as $\mu$) with
    \begin{itemize}
        \item $\Pbb(E_k)\leq \eps,\;\forall k\in\N$,
        \item $\displaystyle \sup_{k\in\N} \frac{\log(k+1)}{\log\frac{1}{\Pbb(E_k)}}<\infty$,
        \item $\forall i\in\N, \exists j\in [J], \exists k\in \N,  \{X_i\neq Z_j\}\subset E_k$.
    \end{itemize}
}

By Corollary~\ref{cor:condition_useless} and \cref{thm:more_general_suff_condition}, this is still a sufficient condition, which fortunately also encompasses the example provided in the previous result, Proposition~\ref{prop:TC1_non_necessary}. However, even with this fix, being able to cover bad events $\{X_i\neq Z_j\}$ with multiple events $E_k$ is still necessary.

\begin{proposition}[(TC2) is not necessary]
\label{prop:TC2_non_necessary}
    There exists a probability measure $\mu$ on $\{0,1\}^\N$ that does not satisfy condition (TC2) but $\Delta_n(\mu)\ninf 0$.
\end{proposition}

\begin{proof}
We partition $\N$ into $\N = \bigcup_{l\geq 1}I_l$, where $I_l =\{2^{l-1}\leq i <2^l\}$ for $l\geq 1$. We consider binary random variables $Y_i,A_i,B_i$ for $i\geq 1$, together independent and such that $Y_i\sim \Bernu(1/2)$, and $A_i, B_i\sim\Bernu(1/\sqrt i)$, for all $i\geq 1$. For $l\geq 1$ and any $i\in I_l$, put
\begin{equation*}
    X_i = (1-A_l)(1-B_i) Y_0 + (A_l+B_i-A_lB_i) Y_i.
\end{equation*}
We then define $\mu$ as the distribution of $(X_i)_{i\geq 1}$. We first show that $\Delta_n(\mu)\ninf 0$ by proving that $\mu$ satisfies the sufficient condition from Corollary~\ref{cor:condition_useless}. Here, we use $K=2$. Fix $\eps>0$ and let $i_0 = \ceil{\frac{1}{\eps^2}}$ and $i_\eps =2^{i_0}$. We then define the events $(E_k)_{k\geq 1}$ as the sequence $(\{A_{i_0}=1\},\{B_{i_0}=1\}, \{A_{i_0+1}=1\},$ $\{B_{i_0+1}=1\},\ldots)$. Because of the polynomial decay of $\Pbb(A_i=1)=\Pbb(B_i=1)=\frac{1}{\sqrt i}$, we can check easily that the events $(E_k)_{k\geq 1}$ satisfy the first two conditions from Corollary~\ref{cor:condition_useless}. Last, we consider $J=i_\eps +1$ centers $Y_0$ and $(X_i)_{i\leq i_\eps}$. The third condition from Corollary~\ref{cor:condition_useless} is trivially satisfied for $i\leq i_\eps$, since $\{X_i \neq X_i\}=\emptyset$. And for $i> i_\eps$, letting $l\geq 1$ such that $i\in I_l$, since $i_\eps=2^{i_0}$, we have $l\geq i_0$. In particular, the events $\{A_l=1\}$ and $\{B_i=1\}$ belong to the sequence $(E_k)_{k\geq 1}$. We can conclude by noting that
\begin{equation*}
    \{X_i\neq Y_0\}\subset \{A_l=1\}\cup \{B_i=1\}.
\end{equation*}
This ends the proof that $\Delta_n(\mu)\ninf 0$.

We now show that $\mu$ does not satisfy (TC2). We suppose by contradiction that it does and use the property for $\eps=\frac{1}{2}$. We use the notations of the condition and for any $i\geq 1$, we denote by $j_i\in [J]$ and $k_i\in\N$ elements such that $\{X_i\neq Z_{j_i}\}\subset E_{k_i}$. Because of the second property, there exists $C>0$ such that
\begin{equation}\label{eq:second_condition}
    \Pbb(E_k) \leq \frac{1}{k^{1/C}},\quad l\geq 1.
\end{equation}
Next, we recall that the sequence $(j_i)_{i\geq 1}$ only takes values in $[J]$. As a result, for any $l\geq |J|+1$ there exists some index $j(l)$ such that
\begin{equation*}
    |\{i\in I_l: j_i=j(l)\}| \geq \frac{|I_l|}{|J|} = \frac{2^{l-1}}{|J|}.
\end{equation*}
We denote this set $\Acal(l) = \{i\in I_l: j_i=j(l)\}$. Suppose for now that for some $i\in\Acal(l)$, we have
\begin{equation*}
    \Pbb(X_i\neq Z_{j(l)}) \leq \frac{3}{8\sqrt{l}}.
\end{equation*}
Then, for any $i'\in \Acal(l)\setminus\{i\}$, one has
\begin{equation*}
    \Pbb(X_{i'} \neq Z_{j(l)}) \geq \Pbb(X_{i'} \neq X_i) - \Pbb(X_i\neq Z_{j(l)})
    \geq \frac{3}{4}\Pbb(A_l = 1) - \frac{3}{8\sqrt l} = \frac{3}{8\sqrt l}.
\end{equation*}
As a result, in all cases, there is a set $\Bcal(l)\subset \Acal(l)$ of cardinality $|\Bcal(l)| = |\Acal(l)|-1 \geq 2^{l-1}/|J|-1$ and for which
\begin{equation*}
    \Pbb(E_{k_i}) \geq \Pbb(X_i\neq Z_{j_i}) = \Pbb(X_i\neq Z_{j(l)}) \geq \frac{3}{8\sqrt l}.
\end{equation*}
By Eq~\eqref{eq:second_condition}, this implies that for all $i\in \Bcal(l)$, one has $k_i \leq (8l)^{C/2}$. As a result, there exists $k(l)$ for which
\begin{equation*}
    |\{i\in I_l: (j_i,k_i) = (j(l),k(l))\}| \geq \frac{|\Bcal(l)|}{ (8l)^{C/2}} \geq \frac{2^{l-1}-|J|}{ (8l)^{C/2}|J|}.
\end{equation*}
We denote this set by $\Ccal(l):=\{i\in I_l: (j_i,k_i) = (j(l),k(l))\}$. In particular, we obtained that for $l\geq (|J|+2)\land \log_2(8(8l)^{C/2}|J|)$,
\begin{equation*}
    |\Ccal(l)| \geq \frac{2^l}{4(8l)^{C/2}|J|} \geq 2.
\end{equation*}
We now use similar arguments to that of Proposition~\ref{prop:TC1_non_necessary}. Fix some element $i(l)\in \Ccal(l)$. We have
\begin{align*}
    1-\Pbb(E_{k(l)}) &\leq \Pbb\paren{\bigcap_{i\in\Ccal(l)} \{X_i = Z_{j(l)}\}}\\
    &\leq \Pbb(X_{i(l)} \neq Y_0) + \Pbb\paren{\bigcap_{i\in\Ccal(l)}\{X_i = Y_0\}}\\
    &= \frac{1}{2}\paren{\frac{1}{\sqrt{i(l)}}+\frac{1}{\sqrt l}- \frac{1}{\sqrt{l i(l)}}} + \prod_{i\in\Ccal(i)}\paren{1-\frac{1}{2\sqrt i}-\frac{1}{2\sqrt l}+\frac{1}{2\sqrt{il}}}\\
    &\leq \frac{1}{2}\paren{\frac{1}{2^{(l-1)/2}}+\frac{1}{\sqrt l}} + \exp\paren{-\frac{|\Ccal(l)|}{2\sqrt l}}.
\end{align*}
For $l$ sufficiently large, this gives $1-\Pbb(E_{k(l)})\leq \frac{1}{4}$, which contradicts the hypothesis $\Pbb(E_k)\leq \eps=\frac{1}{2}$ for all $k\geq 1$. Hence $\mu$ does not satisfy (TC2), which ends the proof.
\end{proof}

\paragraph{Acknowledgements.}
AK and DC were
was partially supported by
the Israel Science Foundation
(grant No. 1602/19), an Amazon Research Award,
and the Ben-Gurion University Data Science Research Center.

\bibliographystyle{abbrvnat}
\bibliography{refs}

\begin{thebibliography}{19}
\providecommand{\natexlab}[1]{#1}
\providecommand{\url}[1]{\texttt{#1}}
\expandafter\ifx\csname urlstyle\endcsname\relax
  \providecommand{\doi}[1]{doi: #1}\else
  \providecommand{\doi}{doi: \begingroup \urlstyle{rm}\Url}\fi

\bibitem[Blanchard and Vor\'{a}\v{c}ek(2024)]{blanchard2023tight}
M.~Blanchard and V.~Vor\'{a}\v{c}ek.
\newblock Tight bounds for local glivenko-cantelli.
\newblock In \emph{Algorithmic Learning Theory, {ALT} 2024}, Proceedings of Machine Learning Research. PMLR, 2024.

\bibitem[Catoni(2012)]{catoni2012challenging}
O.~Catoni.
\newblock Challenging the empirical mean and empirical variance: a deviation study.
\newblock In \emph{Annales de l'IHP Probabilit{\'e}s et statistiques}, volume~48, pages 1148--1185, 2012.

\bibitem[Cherapanamjeri et~al.(2019)Cherapanamjeri, Flammarion, and Bartlett]{Cherapanamjeri19}
Y.~Cherapanamjeri, N.~Flammarion, and P.~L. Bartlett.
\newblock Fast mean estimation with sub-gaussian rates.
\newblock In A.~Beygelzimer and D.~Hsu, editors, \emph{Conference on Learning Theory, {COLT} 2019, 25-28 June 2019, Phoenix, AZ, {USA}}, volume~99 of \emph{Proceedings of Machine Learning Research}, pages 786--806. {PMLR}, 2019.
\newblock URL \url{http://proceedings.mlr.press/v99/cherapanamjeri19b.html}.

\bibitem[Cherapanamjeri et~al.(2020)Cherapanamjeri, Tripuraneni, Bartlett, and Jordan]{Cherapanamjeri20}
Y.~Cherapanamjeri, N.~Tripuraneni, P.~L. Bartlett, and M.~I. Jordan.
\newblock Optimal mean estimation without a variance.
\newblock \emph{CoRR}, abs/2011.12433, 2020.
\newblock URL \url{https://arxiv.org/abs/2011.12433}.

\bibitem[Chollete et~al.(2023)Chollete, {de la Pe\~{n}a}, and Klass]{CHOLLETE202351}
L.~Chollete, V.~{de la Pe\~{n}a}, and M.~Klass.
\newblock The price of independence in a model with unknown dependence.
\newblock \emph{Mathematical Social Sciences}, 123:\penalty0 51--58, 2023.
\newblock ISSN 0165-4896.
\newblock \doi{https://doi.org/10.1016/j.mathsocsci.2023.02.008}.
\newblock URL \url{https://www.sciencedirect.com/science/article/pii/S0165489623000215}.

\bibitem[Cohen and Kontorovich(2023{\natexlab{a}})]{CohenK23}
D.~Cohen and A.~Kontorovich.
\newblock Local {G}livenko-{C}antelli.
\newblock In \emph{Conference on Learning Theory}, Proceedings of Machine Learning Research, 2023{\natexlab{a}}.

\bibitem[Cohen and Kontorovich(2023{\natexlab{b}})]{CohenK23a}
D.~Cohen and A.~Kontorovich.
\newblock Open problem: log(n) factor in "local glivenko-cantelli.
\newblock In G.~Neu and L.~Rosasco, editors, \emph{The Thirty Sixth Annual Conference on Learning Theory, {COLT} 2023, 12-15 July 2023, Bangalore, India}, volume 195 of \emph{Proceedings of Machine Learning Research}, pages 5934--5936. {PMLR}, 2023{\natexlab{b}}.
\newblock URL \url{https://proceedings.mlr.press/v195/cohen23b.html}.

\bibitem[Devroye et~al.(2016)Devroye, Lerasle, Lugosi, and Oliveira]{Devroye16}
L.~Devroye, M.~Lerasle, G.~Lugosi, and R.~I. Oliveira.
\newblock {Sub-Gaussian mean estimators}.
\newblock \emph{The Annals of Statistics}, 44\penalty0 (6):\penalty0 2695 -- 2725, 2016.
\newblock \doi{10.1214/16-AOS1440}.
\newblock URL \url{https://doi.org/10.1214/16-AOS1440}.

\bibitem[Diakonikolas et~al.(2020)Diakonikolas, Kane, and Pensia]{diakonikolas2020outlier}
I.~Diakonikolas, D.~M. Kane, and A.~Pensia.
\newblock Outlier robust mean estimation with subgaussian rates via stability.
\newblock \emph{Advances in Neural Information Processing Systems}, 33:\penalty0 1830--1840, 2020.

\bibitem[Dubhashi and Ranjan(1998)]{Dubhashi:1998:BBS:299633.299634}
D.~Dubhashi and D.~Ranjan.
\newblock Balls and bins: a study in negative dependence.
\newblock \emph{Random Struct. Algorithms}, 13\penalty0 (2):\penalty0 99--124, Sept. 1998.
\newblock ISSN 1042-9832.
\newblock \doi{10.1002/(SICI)1098-2418(199809)13:2<99::AID-RSA1>3.0.CO;2-M}.
\newblock URL \url{http://dx.doi.org/10.1002/(SICI)1098-2418(199809)13:2<99::AID-RSA1>3.0.CO;2-M}.

\bibitem[Hopkins(2020)]{hopkins2020mean}
S.~B. Hopkins.
\newblock Mean estimation with sub-gaussian rates in polynomial time.
\newblock 2020.

\bibitem[Joag-Dev and Proschan(1983)]{10.1214/aos/1176346079}
K.~Joag-Dev and F.~Proschan.
\newblock {Negative Association of Random Variables with Applications}.
\newblock \emph{The Annals of Statistics}, 11\penalty0 (1):\penalty0 286 -- 295, 1983.
\newblock \doi{10.1214/aos/1176346079}.
\newblock URL \url{https://doi.org/10.1214/aos/1176346079}.

\bibitem[Kaplansky(2001)]{kaplansky2001set}
I.~Kaplansky.
\newblock \emph{Set Theory and Metric Spaces}.
\newblock AMS Chelsea Publishing Series. AMS Chelsea Publishing, 2001.
\newblock ISBN 9780821826942.
\newblock URL \url{https://books.google.co.il/books?id=FbKhAQAAQBAJ}.

\bibitem[Lee and Valiant(2022)]{lee2022optimal}
J.~C. Lee and P.~Valiant.
\newblock Optimal sub-gaussian mean estimation in $\mathbb{R}$.
\newblock In \emph{2021 IEEE 62nd Annual Symposium on Foundations of Computer Science (FOCS)}, pages 672--683, 2022.
\newblock \doi{10.1109/FOCS52979.2021.00071}.

\bibitem[Lugosi and Mendelson(2019{\natexlab{a}})]{LuMen19a}
G.~Lugosi and S.~Mendelson.
\newblock {Sub-Gaussian estimators of the mean of a random vector}.
\newblock \emph{The Annals of Statistics}, 47\penalty0 (2):\penalty0 783 -- 794, 2019{\natexlab{a}}.
\newblock \doi{10.1214/17-AOS1639}.
\newblock URL \url{https://doi.org/10.1214/17-AOS1639}.

\bibitem[Lugosi and Mendelson(2019{\natexlab{b}})]{LuMen19b}
G.~Lugosi and S.~Mendelson.
\newblock Mean estimation and regression under heavy-tailed distributions: {A} survey.
\newblock \emph{Found. Comput. Math.}, 19\penalty0 (5):\penalty0 1145--1190, 2019{\natexlab{b}}.
\newblock \doi{10.1007/s10208-019-09427-x}.
\newblock URL \url{https://doi.org/10.1007/s10208-019-09427-x}.

\bibitem[Lugosi and Mendelson(2021)]{LuMen21}
G.~Lugosi and S.~Mendelson.
\newblock {Robust multivariate mean estimation: The optimality of trimmed mean}.
\newblock \emph{The Annals of Statistics}, 49\penalty0 (1):\penalty0 393 -- 410, 2021.
\newblock \doi{10.1214/20-AOS1961}.
\newblock URL \url{https://doi.org/10.1214/20-AOS1961}.

\bibitem[Thomas(2018)]{cstheory-faster-low-ent}
Thomas.
\newblock Is uniform convergence faster for low-entropy distributions?
\newblock Theoretical Computer Science Stack Exchange, 2018.
\newblock URL \url{https://cstheory.stackexchange.com/q/42009}.
\newblock URL:https://cstheory.stackexchange.com/q/42009 (version: 2018-12-10).

\bibitem[Van~Handel(2014)]{van2014probability}
R.~Van~Handel.
\newblock Probability in high dimension.
\newblock Technical report, PRINCETON UNIV NJ, 2014.

\end{thebibliography}

\end{document}